\documentclass[12pt]{amsart}
\usepackage{amsmath,amscd,amsbsy,amssymb,amsthm,amsfonts}
\usepackage{latexsym,url,bm}
\usepackage{cite,enumitem}
\usepackage{fullpage,comment}
\usepackage{color,tikz, pgfplots}
\usepackage[colorlinks=true]{hyperref}
\usepackage{chngcntr}
\usepackage{apptools}

\AtAppendix{\counterwithin{lemma}{section}}

\newtheorem{theorem}{Theorem}[section]
\newtheorem{lemma}{Lemma}
\newtheorem{proposition}[lemma]{Proposition}

\newtheorem{definition}[lemma]{Definition}
\newtheorem{remark}[lemma]{Remark}

\numberwithin{lemma}{section}

\allowdisplaybreaks

\numberwithin{equation}{section}

\newcommand{\R}{{\mathbb R}}

\renewcommand{\H}{{\mathcal H }}

\newcommand{\nP}{{\mathbf P}}

\newcommand{\Half}{{\frac{1}{2}}}

\newcommand{\CalAO}{{\mathcal{A}_1}}
\newcommand{\CalAT}{{\mathcal{A}_{\frac{3}{2}}}}
\newcommand{\CalAZ}{{\mathcal{A}_0}}
\newcommand{\CalAZS}{{\mathcal{A}_0^\sharp}}
\newcommand{\CalAOS}{{\mathcal{A}_1^\sharp}}
\newcommand{\CalATS}{{\mathcal{A}^\sharp_{\frac{3}{2}}}}

\newcommand{\W}{{\mathbf W}}
\newcommand{\Y}{{\mathcal Y}}

\author{Lizhe Wan}
\address{Department of Mathematics, University of Wisconsin - Madison}
\email{lwan33@wisc.edu}

\keywords{water waves, free boundary problem, point vortices, enhanced lifespan.}
\subjclass[2020]{35Q31, 76B03, 76B47}
\begin{document}

\title{On the gravity-capillary water waves with point vortices}

\begin{abstract}
We consider the two-dimensional deep gravity-capillary  water waves with point vortices.
We first formulate the question in the holomorphic coordinates.
Then, we derive an a priori energy estimate for water waves, and show that the water wave system has a unique solution for initial data in $\H^s\times \Omega_t^N$, $s>\frac{3}{2}$.
Finally, we show that if there are only two vortices, and vortices, initial velocity, and initial surfaces are symmetric with respect to the vertical axis, then the solution of pure capillary water waves with two point vortices has an extended cubic lifespan.
\end{abstract}

\maketitle

\tableofcontents

\section{Introduction}
In this paper, we study the Cauchy problem of two-dimensional deep gravity-capillary water waves with point vortices.
This system solves a special kind of free boundary incompressible Euler equations with appropriate boundary conditions, and the fluid vorticity is irrotational everywhere except for finitely many point vortices.

Suppose that an incompressible, homogeneous, and inviscid fluid occupies a time-dependent unbounded domain $\Omega_t\subset \mathbb{R}^2$ at each time $t\geq 0$.
We assume that the domain $\Omega_t$ has infinite depth, and its upper free boundary $\Gamma_t$ is asymptotically flat at infinity.
Let the fluid velocity at time $t\geq 0$ be $\textbf{v}(t, \cdot): \Omega_t \rightarrow \mathbb{R}^2$.
The vorticity of the fluid is defined to be
\begin{equation*}
    \omega := \nabla^\perp \cdot \textbf{v}, \quad \nabla^\perp : = ( - \partial_y, \partial_x).
\end{equation*}
While most of the results in water waves assume zero or constant vorticity, we will consider in this paper that the vorticity consists of a finite sum of point vortices.
A \textit{point vortex} for vorticity is the case where $\omega(t) = \lambda \delta_{z(t)}.$
In other words, the vorticity is a Dirac measure supported at $z(t)\in \Omega_t$ with \textit{vortex strength} $\lambda$. 
For two-dimensional fluid, the vorticity $\omega$ satisfies the transport equation $(\partial_t + \mathbf{v}\cdot \nabla) \omega = 0.$
The transport term $\mathbf{v}\cdot \nabla \omega$ is not well-defined at the vortex center $z(t)$.
To make the transport equation well-defined, we require $\mathbf{v}$ to be a weak solution to the Euler equations away from $z(t)$.
Assuming there are $N$ point vortices with vortex strength $\lambda_j$ at vortex center $z_j(t)$ for $1\leq j \leq N$, then inside the domain $\Omega_t$, the fluid satisfies the incompressible Euler equations
\begin{equation} \label{Euler}
\left\{
             \begin{aligned}
            &\mathbf{v}_t +\mathbf{v}\cdot \nabla \mathbf{v}= -\nabla p - g\textbf{e}_y &&  \text{in } \Omega_t\backslash \cup_{j=1}^N\{ z_j(t)\}, \\
            &\nabla \cdot \mathbf{v} = 0 && \text{in } \Omega_t, \\
            &\omega = \sum_{j=1}^N \lambda_j \delta_{z_j(t)} &&\text{in } \Omega_t,
             \end{aligned}
\right.
\end{equation}
for gravitational constant $g\geq 0$, and pressure $p(t,\cdot): \Omega_t \rightarrow \mathbb{R}$.
Since this is a free boundary problem, on the upper surface $\Gamma_t$, we have the kinematic boundary condition 
\begin{equation*}
\partial_t+ \mathbf{v} \cdot \nabla \text{ is tangent to } \bigcup \Gamma_t,
\end{equation*}
as well as the dynamic boundary condition
\begin{equation*}
 p = - \sigma {\bf H}  \ \ \text{ on } \Gamma_t.
\end{equation*}
Here, ${\bf H}$ is the mean curvature of the boundary, and $\sigma>0$ is the coefficient of the surface tension.
The position of each point vortex centered at $z_j(t)$ is governed by the \textit{Helmholtz--Kirchhoff} model 
\begin{equation} \label{Helmholtz}
 \frac{d}{dt} z_j(t) = \textbf{v}(t, z_j(t)) = \textbf{u}(t, z_j(t))-\sum_{ k\neq j}\frac{\lambda_k}{2\pi} \nabla^\perp  \log{|{z_j(t) -z_k(t)}|},
\end{equation}
where $\mathbf{u}$ is the irrotational part of the velocity field such that $\nabla \times \mathbf{u} = \mathbf{0}$, and $\nabla \cdot \mathbf{u} = 0$.   
The second term on the right-hand side of \eqref{Helmholtz} vanishes in the case of a single vortex.
A rigorous derivation of the Helmholtz–Kirchhoff  model can be found in Marchioro and Pulvirenti \cite{MR1220946}, and Chapter $4$ of the textbook \cite{MR1245492}.

For free boundary incompressible Euler equations with nonzero vorticity, see the work of  Beyer-Gunther \cite{MR1637554}, Ogawa-Tani \cite{MR1946720}, Lindblad \cite{MR2178961}, Coutand-Shkoller \cite{MR2291920}, Shatah-Zeng \cite{MR2763036, MR2388661}, Zhang-Zhang \cite{MR2410409}, Bieri-Miao-Shahshahani-Wu \cite{bieri2015motion}, Ifrim-Tataru \cite{MR3869381}, Christodoulou-Lindblad \cite{MR1780703}, Wang-Zhang-Zhao-Zheng \cite{MR4263411}, Ifrim-Pineau-Tataru-Taylor \cite{ifrim2023}, and the references therein.

However, the results mentioned in the previous paragraph do not apply in the setting of point vortices.
Here, we first recall some results in fluid dynamics with point vortices in the fixed domain.
Gallay in \cite{MR2787587} considered Navier–Stokes equations with increasingly concentrated vorticity.
He showed that the vanishing viscosity limit for smooth solutions is just Euler equations with point vortices.
Glass, Munnier, and Sueur in \cite{MR3858400} showed that the motion of a point vortex inside a domain can be obtained as the limit of the motion of a rigid body immersed in the fluid when the body shrinks to a massless point particle with fixed circulation.

As for the water waves with point vortices, there are many results on the existence and stability of solitary waves, see Ter-Krikorov \cite{MR108145}, Filippov \cite{MR128208, MR151069}, and Chen-Varholm-Walsh-Wheeler \cite{chen2024} for the gravity water waves, and Shatah-Walsh-Zeng \cite{MR3053431}, Varholm \cite{MR3485858}, and Varholm-Wahl\'en-Walsh \cite{MR4164269} for the gravity-capillary water waves.
To the author's knowledge, the only well-posedness result for the Cauchy problem is the work of Su \cite{MR4179726} on the gravity water waves with point vortices.
Su also considered in \cite{MR4656809} the transition of the Rayleigh-Taylor instability for gravity water waves with a pair of point vortices.
No well-posedness result for gravity-capillary water waves with point vortices was known before.

\subsection{Water waves in holomorphic coordinates}
Instead of solving the Euler equations \eqref{Euler} in a moving domain, one can reformulate the problem and work in a fixed domain instead.
In Section $5$ of \cite{MR4164269}, Varholm, Wahl\'en, and Walsh rewrote the water waves with a single point vortex.
This formulation was previously used to study irrotational water waves by, for instance, Alazard, Burq, and Zuily \cite{MR3852259, MR2805065, MR3260858, MR2931520, MR3465379}.
In \cite{MR4179726, MR4656809}, Su reexpressed the gravity water waves with point vortices using a mix of Lagrangian and conformal formulation as in the framework of Wu \cite{MR1471885, MR2507638} for irrotational gravity water waves. 

While both two formulations are feasible for studying the Cauchy problem of gravity-capillary water waves with point vortices, 
we will follow the formulation in Hunter-Ifrim-Tataru \cite{MR3535894}, and work in the \textit{holomorphic coordinates}.
We remark that the use of holomorphic coordinates has been successfully applied in many water wave models by Ifrim, Tataru, and their collaborators, see \cite{MR3535894, MR3499085, MR3625189,  ai2023dimensional, MR4462478, MR4483135, MR3667289, MR3869381}.

In this paper, we will slightly abuse notation and write holomorphic functions either as functions on the lower half of the complex plane or as their restriction to the real line.
Let $\W(t, \alpha)$ be the differentiated holomorphic position of the free surface,  $R(t, \alpha)$ be the irrotational part of the complex conjugate of the complex velocity on the surface, and $z_j(t): = x_j(t)+iy_j(t)\in \Omega_t$ be the complex position of the $j$-th point vortex.
The water wave system can be reexpressed in terms of  unknowns $(\W, R, \{z_j\}_{j=1}^N)$.
Following the idea of Appendix $A$ of Hunter-Ifrim-Tataru \cite{MR3535894} and Ifrim-Tataru \cite{MR3667289}, the gravity-capillary water waves with point vortices are given by the system
\begin{equation}\label{e:WW}
\left\{
\begin{aligned}
& (\partial_t + b\partial_\alpha)\W + \frac{(1+\W)R_\alpha}{1+\bar{\W}} = (1+\W)M - \sum_{j=1}^N \frac{\lambda_j i}{2\pi}\frac{1+\bar{\W}}{(\overline{Z-z_j})^2}\\
& (\partial_t + b\partial_\alpha)R  - (\partial_t + b\partial_\alpha) \sum_{j=1}^N \frac{\lambda_j i}{2\pi}\frac{1}{Z-z_j} -i\frac{g\W - a}{1+\W}   = \frac{ 2\sigma}{1+\W}\mathbf{P}\Im \left[ \frac{\W_{ \alpha}}{J^{\frac{1}{2}}(1+\W)}\right]_{\alpha}\\
&\frac{d}{dt}z_j(t) = \mathcal{U}(t, z_j(t)) + \sum_{k\neq j} \frac{\lambda_k i}{2\pi}\frac{1}{\overline{z_j(t) - z_k(t)}}, \quad \forall 1\leq j \leq N,
\end{aligned} 
\right.
\end{equation}
where $J := |1+\W|^2$ is the Jacobian, $\mathbf{P} = \frac{1}{2}(I -iH)$ is the holomorphic projection, $Z_\alpha = 1+\W$, $\mathcal{U}(t,z)$ is the irrotational part of velocity at the point $z$, which is given by
\begin{equation} \label{IrrotationalV}
   \mathcal{U}(t,z) :=  \frac{i}{2\pi}\int_{\mathbb{R}} \frac{(1+\W)\cdot \bar{R} (t,\alpha)}{Z(t, \alpha)-z} \,d\alpha,
\end{equation}
and auxiliary functions $a$, $b$ and $M$ are defined in \eqref{FrequencyShit}, \eqref{bDef} and \eqref{MDef} respectively.
The detailed computations and derivations of the water wave system \eqref{e:WW} will be carried out in Section \ref{s:Eqn}. 

When there are no point vortices, the irrotational gravity-capillary water waves are locally well-posed for $(\W_0, R_0)\in H^{s+\frac{1}{2}}(\mathbb{R})\times H^s(\mathbb{R})$ for $s>1$, see \cite{wan2024}.
In this paper, we will also use the functional space
\begin{equation*}
    \H^s: = H^{s+\frac{1}{2}}(\mathbb{R})\times H^s(\mathbb{R}),
\end{equation*}
and consider $(\W(t), R(t), \{z_j(t)\}_{j=1}^N)\in \H^s \times \Omega_t^N$.

\subsection{Main results}
The first main result of this paper is the energy estimate for water waves \eqref{e:WW}.
Before stating the energy estimate, we first introduce the control norms.
Let $\epsilon>0$ be an arbitrarily small positive constant, and fix time $t$, define
\begin{align*}
&\CalAZ : = \|\W \|_{ C^{\epsilon}_{*}}+ \|R \|_{ C^{-\frac{1}{2}+\epsilon}_{*}}+ \|Z-\alpha\|_{L^\infty} + \sum_{j=1}^N |\lambda_j| + \max_{j\neq k}\frac{|\lambda_j|}{|z_j - z_k|}, \\ 
&\CalAO : = \|\W \|_{ C^{1+\epsilon}_{*}}+\|R\|_{C_{*}^{\frac{1}{2}+\epsilon}}+  \|Z-\alpha\|_{C^1_{*}}+ \sum_{j=1}^N |\lambda_j|+ \max_{j\neq k}\frac{|\lambda_j|}{|z_j - z_k|},\\
&\CalAT : = \|\W \|_{ C^{\frac{3}{2}+\epsilon}_{*}}+ \|R \|_{ C^{1+\epsilon}_{*}}+ \|Z-\alpha\|_{C^{\frac{3}{2}}_{*}}+ \sum_{j=1}^N |\lambda_j|+ \max_{j\neq k}\frac{|\lambda_j|}{|z_j - z_k|}.
\end{align*}
Here, $C^s_{*}$ is the Zygmund space that will be defined in Section \ref{s:Norm}. 
Each control norm in $\CalAO$ has $1$ more derivative than the corresponding term in $\CalAZ$, and each control norm in $\CalAT$ has $\frac{3}{2}$ more derivatives than the corresponding term in $\CalAZ$.
In addition, if the term
\begin{equation*}
\lim_{t\nearrow T^{*}}\max_{j\neq k}\frac{|\lambda_j|}{|z_j(t) - z_k(t)|} = +\infty,
\end{equation*}
then at least two point vortices collide at time $T^{*}$.
The model \eqref{Euler} is no longer true and we will not consider this situation.

The first main result in this paper is the a priori energy estimate for the system \eqref{e:WW}.
\begin{theorem} \label{t:MainEnergyEstimate}
 For any $s>0$, suppose  $(\W, R, \{z_j\}_{j=1}^N)\in C( [0,T]; \H^s)\times C^1([0,T];\Omega_t)^N$ solve the water waves \eqref{e:WW}.
 Then, there exists an energy functional $E_s$ which has the following properties:
 \begin{enumerate}
     \item Norm equivalence:
       \begin{equation}
           E_s \approx_\CalAZ \|(\W,  R)\|^2_{\mathcal{H}^s} + \| Z-\alpha\|_{L^2}^2+ \sum_{j=1}^N |\lambda_j|^2.\label{normEquivalence}
       \end{equation}
     \item Energy estimate:
      \begin{equation}
     \frac{d}{dt}E_s\lesssim_\CalAO \CalAT E_s.\label{EnergyEstimate}
 \end{equation} 
 \end{enumerate} 
\end{theorem}

As a consequence of the a priori energy estimate, one can prove the local well-posedness result of water waves \eqref{e:WW}.
\begin{theorem} \label{t:Wellposed}
For $s> \frac{3}{2}$, if the initial datum $(\W_0, R_0)\in \H^s$, and disjoint point vortices $z_{j}(0)\in \Omega_0$, $\lambda_j\in \mathbb{R}$
 for $1\leq j \leq N$,  then the system \eqref{e:WW} has a unique solution $(\W, R) \times \{z_j(t) \}_{j=1}^N\in C( [0,T]; \H^s)\times C^1([0,T];\Omega_t)^N$ for some time $0<T< T^{*}$.
Moreover, $T^{*}$ is the maximal lifespan such that either $T^{*} = +\infty$ or the solution blows up at $T^{*}$:
\begin{equation*}
   \lim_{\tau\nearrow T^{*}} \| \CalAO(t)\|_{L^\infty_t[0, \tau)} + \| \CalAT(t)\|_{L^1_t[0,\tau)} = +\infty.
\end{equation*}
\end{theorem}

\begin{remark}
\begin{enumerate}
\item For the blow-up criterion of gravity water waves with point vortices in \cite{MR4179726}, the author uses Sobolev norms to characterize the blow-up of solutions.
Here, in Theorem \ref{t:MainEnergyEstimate}, we prove the tame energy estimate that involves Zygmund control norms $\CalAZ, \CalAO, \CalAT$.
This is why  the blow-up criterion in Theorem \ref{t:Wellposed} does not involve the Sobolev control norms.
\item An interesting question one may ask is whether it is possible to further show that $z_j(t)\in C^2([0,T]; \Omega_t)$ for each $j$ as in the corresponding result in \cite{MR4179726}?
In order to estimate $|\frac{d^2}{dt^2}z_j|$, one will have to estimate $\|R_t\|_{L^\infty},$ which is bounded when $s>2$(or $s>\frac{7}{4}$ if further using the Strichartz estimate).
It may not be bounded in the low regularity setting $\frac{3}{2}<s\leq \frac{7}{4}$.
Hence, we only prove that $z_j(t)\in C^1([0,T]; \Omega_t)$.
\item For reference, the local well-posedness of two dimensional irrotational gravity-capillary water waves was established in, for instance, Alazard-Burq-Zuily \cite{MR2805065}, Ifrim-Tataru \cite{MR3667289}, de Poyferr\'{e}-Nguyen \cite{MR3487264}, Nguyen \cite{MR3724757}, and Ai \cite{ai2023improved}.
The local well-posedness of two-dimensional irrotational gravity water waves was established in Wu \cite{MR1471885}, Lannes \cite{MR2138139}, Alazard-Burq-Zuily \cite{MR3260858,MR3465379,MR3852259}, Hunter-Ifrim-Tataru \cite{MR3348783}, Ai \cite{MR4161284, MR4098033, MR4035330}, and Ai-Ifrim-Tataru \cite{ai2023dimensional} for the low regularity well-posedness.
\end{enumerate}
\end{remark}

The dynamics of point vortices when $N \geq 3$ can be highly complex.
Su in Section 3.3 in \cite{MR4179726} has a brief discussion on some possible situations for the dynamics of two vortices $N = 2$. 

When there are no point vortices and no gravity, the first global well-posedness result belongs to Ifrim-Tataru \cite{MR3667289}.  
In their work, Ifrim-Tataru proved a low regularity global well-posednss result, but also showed that the solutions do enjoy a cubic lifespan, as well as, scale-invariant energy estimates which both findings allowed them to work close to the regularity threshold. 
This is the first result of this type for capillary water waves in 2D.
This result was followed, later on,  by a  global well-posedness result  of Ionescu-Pusateri \cite{MR3862598}, however obtained in a much higher regularity setting.

For water waves with point vortices \eqref{e:WW}, since the terms with point vortices are not quadratic, in general, one can only expect quadratic lifespans for solutions.
However, our next main result shows that when $N = 2$, and these two point vortices are symmetric with respect to the vertical axis, small initial data leads to the extended cubic lifespan.

\begin{theorem} \label{t:lifespan}
Let $s>\frac{3}{2}$, and the gravity is zero $g=0$, we assume that the water wave system \eqref{e:WW} satisfies the following conditions:
\begin{enumerate}
\item 2 point vortices are symmetric, and are not close to the surface: The number of vortices $N =2$, with position $z_1(t) = -x(t)+ iy(t)$, $z_2(t) = x(t)+iy(t)$ and strength $\lambda$, $-\lambda$, where $x(t)>0$, $y(t)<0$, and $\lambda$ satisfies either $\lambda<0$, $|y(0)|\gtrsim 1$ or $\lambda>0$, $|y(0)| \gg \epsilon^{-2} $.
\item The free upper surface and the complex velocity field on the surface are symmetric: Following the definition of variables in Section \ref{s:Eqn},  $(\Re \W_0, \Im R_0, \Im Z_0)$ are even functions and $(\Im \W_0, \Re R_0, \Re Z_0-\alpha)$ are odd functions.
\item The initial datum of \eqref{e:WW} and vortex strength are small such that
\begin{equation*}
 \|(\W_0, R_0, Z_0 -\alpha) \|_{\H^s\times L^2}\leq \epsilon, \quad |\lambda |^2 + |\lambda x(0)|\leq \epsilon
\end{equation*}
for some small constant $0<\epsilon\leq \epsilon_0 \ll 1$.
\item The initial rotational part of the velocity dominates the irrotational part of the velocity: $\frac{|\lambda|}{x(0)}\geq M $, for some constant $M\gg 1$.
\end{enumerate} 
Then the solution of \eqref{e:WW} has an extended cubic lifespan $T \geq \delta \epsilon^{-2}$ for some constant $\delta>0$ that depends on $\epsilon_0$.
Moreover, the energy of the solution is bounded:
\begin{equation} \label{WRZNormBound}
 \|(\W(t), R(t), Z(t) -\alpha) \|_{\H^s\times L^2}\lesssim \epsilon, \quad \forall  t\in [0, \delta\epsilon^{-2}].
\end{equation}
\end{theorem}
A detailed energy inequality will be given in Theorem \ref{t:CubicEnergy}.

\begin{remark}
These conditions ensure that two point vortices either move downward or move upward but their positions are far from the surface.
In either case, the terms caused by point vortices are under control.
\end{remark}

For the proof of Theorem \ref{t:lifespan}, the strategy is to show that point vortices have bounded vertical velocities, and their horizontal movements are also controlled, so that the rotational terms are small.
For the irrotational part of the system, we use the normal form analysis in the spirit of \cite{MR3667289}.
As a consequence, we will prove the modified energy estimate Theorem \ref{t:CubicEnergy}, which will then lead to the cubic lifespan of the solution.

\subsection{The organization of the paper}
Section \ref{s:Def} is devoted to basic estimates and derivation of the water wave system \eqref{e:WW}.
In Section \ref{s:Norm}, we define the function spaces we will use in this paper.
In Section \ref{s:ParaEst}, we define paradifferential operators and list paradifferential estimates that will be used in the analysis.
We then follow the derivation in \cite{MR3535894} to derive the system of equations governing the evolution of water waves in holomorphic coordinates in Section \ref{s:Eqn}.

Next, in Section \ref{s:Estimate}, we derive some water wave estimates including estimates of $a, b$, the rotational part of the complex velocity, and the time derivatives of point vortices.
We also compute  the para-material derivatives $T_{D_t} := \partial_t + T_b\partial_\alpha$ of $\W, R$ and $J^s$.

In the next section, we consider the local well-posedness of water waves.
In Section \ref{s:Energy}, we derive the modified energy estimate Theorem \ref{t:MainEnergyEstimate} by constructing appropriate cubic energy corrections.
Later, in Section \ref{s:ProofWell}, we prove Theorem \ref{t:Wellposed} by building a sequence of approximate solutions to show that this sequence converges to the solution of water waves \eqref{e:WW}.
We also show that the solution is unique.

Finally, in Section \ref{s:CubicLifespan}, we consider the case of two symmetric small point vortices as in the case of Theorem \ref{t:lifespan}. 
We make the bootstrap assumption on the Sobolev norms of $(\W, R, Z-\alpha)$.
We obtain lower bounds for the distance of the vortices to the surface $d_S(t)$ \eqref{DefDistance}, and show that rotational terms in water waves can be bounded by $\epsilon d_S(t)^{-\frac{3}{2}}$.
Working with normal form variables and choosing appropriate energy corrections, we prove the modified energy estimate \eqref{t:CubicEnergy}.
At the end, we close the bootstrap assumption \eqref{bootstrap} and finish the proof of Theorem \ref{t:lifespan}.

\section{Preliminaries and derivation of water waves} \label{s:Def}

\subsection{Norms and function spaces} \label{s:Norm}
In this section, we list the definitions of norms and function spaces we will need for the analysis. 

In two space dimensions, we identify $\mathbb{R}^2$ with $\mathbb{C}$, so that a point $(x, y)\in \mathbb{R}^2$ is equivalent to a complex number $x + iy\in \mathbb{C}$.
In the following, we will write $\mathbb{R}^2$ vectors as complex numbers to simplify the computation.

Let $H$ be the Hilbert transform on the real line:
\begin{equation*}
     Hf(x)  := p_{.}v_{.} \frac{1}{\pi}\int_{\mathbb{R}} \frac{f(y)}{x-y}\,dy.
\end{equation*}
The Hilbert transform can also be represented as a Fourier multiplier
\begin{equation*}
    \widehat{Hf}(\xi) = -i\mbox{sgn}(\xi)\hat{f}(\xi).
\end{equation*}
As a consequence, $H^2 = -I$.
An important property of the Hilbert transform is the convolution identity  
\begin{equation} \label{HilbertCon}
  H\left[uH[v]+vH[u]\right] = H[u]H[v] - uv, \quad \forall u,v \in L^2,
\end{equation}

We define $\mathbf{P}$ as the projection onto negative frequencies
\begin{equation*}
  \widehat{\nP f}(\xi): = \mathbf{1}_{(-\infty,0)}\hat{f}(\xi)
\end{equation*}
The operator $\nP$ can also be expressed using the Hilbert transform
\begin{equation*}
    \mathbf{P} = \frac{1}{2}(\mathbf{I} - iH).
\end{equation*}
A complex-valued function $f$ on the real line that can be extended to be holomorphic in the lower half-plane satisfies
\begin{equation*}
 \nP f = f,
\end{equation*}
which means that its real and imaginary parts satisfy the relations
\begin{equation} \label{ReIm}
    \Re f = H\Im f, \quad \Im f = -H \Re f.
\end{equation}
We will see later in the derivation of water waves, unknowns $\W$, $R$, and $Z-\alpha$ are all holomorphic functions, and therefore satisfy the above properties.
On the other hand, a complex-valued function $f$ on the real line that can be extended to be holomorphic in the upper half-plane satisfies $\bar{\nP} f = f$.

Consider the  Littlewood-Paley frequency decomposition,
\begin{equation*}
    I = \sum_{k\in \mathbb{N}} P_k, 
\end{equation*}
where for each $k\geq 1$, the symbols of the operator $P_k$ are smooth and localized at $2^k$, and $P_0$ selects the low frequency components $|\xi|\leq 1$.

 Let $s\in \mathbb{R}$, and $p,q \in [1, \infty]$, the non-homogeneous Besov space $B^s_{p,q}(\mathbb{R})$ is defined as the space of all  tempered distributions $u$ such that
\begin{equation*}
\| u\|_{B^s_{p,q}} : = \left\|(2^{ks}\|P_k u \|_{L^p})_{k=0}^\infty \right\|_{l^q} < +\infty.
\end{equation*}
When $p = q = \infty$, the Besov space $B^s_{\infty, \infty}$ becomes the \textit{Zygmund space} $C^s_{*}$.
When $p = q =2$, the Besov space $B^s_{2,2}$ becomes the \textit{Sobolev space} $H^s$.

 Let $1\leq p_1 \leq p_2 \leq \infty$, $1\leq r_1 \leq r_2 \leq \infty$, then for any real number $s$,
\begin{equation*}
    B^s_{p_1, r_1}(\mathbb{R}) \hookrightarrow B^{s-(\frac{1}{p_1} - \frac{1}{p_2})}_{p_2, r_2}(\mathbb{R}).
\end{equation*}
As a special case when $p_1 = r_1 =2$ and $p_2 = r_2 = \infty$, 
\begin{equation*}
 H^{s+\frac{1}{2}}(\mathbb{R}) \hookrightarrow C^s_{*}(\mathbb{R}) \quad \forall s. 
\end{equation*}
The Sobolev space $H^{s+\frac{1}{2}}(\mathbb{R})$ can be embedded into the Zygmund space $C^s_{*}(\mathbb{R})$.

 The Zygmund space $C^s_{*}(\mathbb{R})$ is just the usual H\"{o}lder space $W^{s, \infty}(\mathbb{R})$ when $s\in (0,\infty)\backslash \mathbb{N}$.
One has the embedding properties
\begin{align*}
  &C_{*}^s(\mathbb{R}) \hookrightarrow L^\infty(\mathbb{R}), \quad s>0; \qquad L^\infty(\mathbb{R}) \hookrightarrow C_{*}^s, \quad s<0;\\
  &C_{*}^{s_1}(\mathbb{R})\hookrightarrow C_{*}^{s_2}(\mathbb{R}), \quad H^{s_1}(\mathbb{R})\hookrightarrow H^{s_2}(\mathbb{R}), \qquad s_1>s_2.
\end{align*}

\subsection{Paradifferential estimates} \label{s:ParaEst}
In this section, we recall some paradifferential estimates that will be used in the rest part of the paper.

\begin{definition}
\begin{enumerate}
\item Let $\rho\in [0,\infty)$, $m\in \mathbb{R}$. 
$\Gamma^m_\rho(\mathbb{R})$ denotes the space of locally bounded functions $a(x, \xi)$ on $\mathbb{R}\times (\mathbb{R}\backslash \{0\})$, which are $C^\infty$ with respect to $\xi$ for $\xi \neq 0$ and such that for all $k \in \mathbb{N}$ and $\xi \neq 0$, the function $x\mapsto \partial_\xi^k a(x,\xi)$ belongs to $W^{\rho,\infty}(\mathbb{R})$ and there exists a constant $C_k$ with
\begin{equation*}
\forall |\xi|\geq \frac{1}{2}, \quad \|\partial_\xi^k a(\cdot,\xi) \|_{W^{\rho,\infty}} \leq C_k (1+ |\xi|)^{m-k}.
\end{equation*}
Let $a(x,\xi)\in \Gamma^m_\rho$,  we define the semi-norm
\begin{equation*}
M^m_{\rho}(a) = \sup_{k \leq \frac{3}{2}+\rho} \sup_{|\xi|\geq \frac{1}{2}} \|(1+ |\xi|)^{k-m}\partial_\xi^k a(\cdot,\xi)  \|_{W^{\rho,\infty}}.
\end{equation*}
\item Given $a(x,\xi)\in \Gamma^m_\rho(\mathbb{R})$, let $C^\infty$ functions $\chi(\theta, \eta)$ and $\psi(\eta)$ be such that for some $0<\epsilon_1 < \epsilon_2<1$,
\begin{align*}
    &\chi(\theta, \eta) = 1,  \text{ if } |\theta| \leq \epsilon_1(1+ |\eta|), \qquad \chi(\theta, \eta) = 0,  \text{ if } |\theta| \geq \epsilon_2(1+ |\eta|),\\
    &\psi(\eta) = 0, \text{ if } |\eta|\leq \frac{1}{5}, \qquad \psi(\eta) =1, \text{ if } |\eta|\geq \frac{1}{4}.
\end{align*}
We define the paradifferential operator $T_a$ in Weyl quantization by
\begin{align*}
    \widehat{T_a u}(\xi) = \frac{1}{2\pi}\int \chi(\xi -\eta, \xi+\eta) \hat{a}\left(\xi-\eta, \frac{\xi+\eta}{2}\right)\psi(\eta)\hat{u}(\eta) d\eta,
\end{align*}
where $\hat{a}(\theta, \xi)$ is the Fourier transform of $a(x,\xi)$ with respect to the spacial variable x.
\item Let $m\in \mathbb{R}$, an operator  is said to be of order $m$ if, for all $s\in \mathbb{R}$, it is bounded from $H^s$ to $H^{s-m}$.
\end{enumerate}
\end{definition}

We recall the basic symbolic calculus for paradifferential operators in the following result.
\begin{lemma}[\hspace{1sp}\cite{MR3052498}]
Let $m\in \mathbb{R}$ and $\rho\in [0, +\infty)$.
\begin{enumerate}
\item If $a\in \Gamma^m_0$, then the paradifferential operator $T_a$ is of order $m$. 
Moreover, for all $s\in \mathbb{R}$, 
\begin{equation}
\|T_a\|_{H^s\rightarrow H^{s-m}} \lesssim M^m_0(a). \label{TABound}
\end{equation}
\item If $a\in \Gamma^m_\rho$, and $b\in \Gamma^{m^{'}}_\rho$ with $\rho\in(1,2]$, then 
\begin{equation*}
 T_a T_b - T_b T_a = iT_{\partial_x a \partial_\xi b - \partial_\xi a \partial_x b} + E, \quad T_a T_b - T_{ab} = \frac{i}{2}T_{\partial_x a \partial_\xi b - \partial_\xi a \partial_x b} + E,
\end{equation*}
where the remainder operator $E$ is of order $m+m^{'} -\rho$.
Moreover, for all $s\in \mathbb{R}$, 
\begin{equation}
  \|E \|_{H^s \rightarrow H^{s-m-m^{'}+\rho}} \lesssim M^m_\rho(a)M^{m^{'}}_0(b) + M^m_0(a)M^{m^{'}}_\rho(b). \label{CompositionPara}
\end{equation}
When $\rho \in (0,1]$, then $T_a T_b - T_b T_a = E$, and $T_a T_b - T_{ab} = E$ with $E$ satisfies the same operator bound \eqref{CompositionPara}.
\item The adjoint operator of the paradifferential operator $T_a$ is $(T_a)^{*} = T_{\bar{a}}$.
In particular, if $a$ is a real-valued function, then $T_a = (T_a)^{*}$.
\end{enumerate}
\end{lemma}

In Besov spaces, we have similar results for symbolic calculus.
\begin{lemma}[\hspace{1sp}\cite{MR3585049}]
Let $m, m^{'}, s\in \mathbb{R}$, $q\in [1,\infty]$ and $\rho\in [0,1]$.
\begin{enumerate}
\item If $a\in \Gamma^m_0$, then
\begin{equation*}
\|T_a\|_{B^s_{\infty, q}\rightarrow B^{s-m}_{\infty, q}} \lesssim M^m_0(a).
\end{equation*}
\item If $a\in \Gamma^m_\rho$, and $b\in \Gamma^m_\rho$, then
\begin{equation}
  \|T_a T_b-T_{a b} \|_{B^s_{\infty, q} \rightarrow B^{s-m-m^{'}+\rho}_{\infty, q}} \lesssim M^m_\rho(a)M^{m^{'}}_0(b) + M^m_0(a)M^{m^{'}}_\rho(b). \label{CompositionTwo}
\end{equation}
\end{enumerate}
In particular, when $q = \infty$, above symbolic calculus results hold for Zygmund spaces $C^s_{*}$.
\end{lemma}

When $a$ is only a function of $x$, $T_a u$ is the low-high paraproduct.
We then define 
\begin{equation*}
\Pi(a, u) := au -T_a u -T_u a
\end{equation*}
to be the balanced paraproduct.
For later use, we record below some estimates for balanced and low-high paraproducts.

\begin{lemma}[\hspace{1sp}\cite{MR2768550}] \label{t:ParaProductEst}
\begin{enumerate}
\item Let $\alpha, \beta \in \mathbb{R}$. 
If $\alpha+ \beta >0$, then
\begin{align}
&\|\Pi(a, u)\|_{H^{\alpha + \beta}(\mathbb{R})} \lesssim \| a\|_{C_{*}^\alpha(\mathbb{R})} \| u\|_{H^\beta(\mathbb{R})}, \label{HCHEstimate}\\
& \|\Pi(a, u)\|_{C_{*}^{\alpha + \beta}(\mathbb{R})} \lesssim \| a\|_{C_{*}^\alpha(\mathbb{R})} \| u\|_{C_{*}^\beta(\mathbb{R})}  \label{CCCEstimate}
\end{align}
\item Let $m > 0$ and $s\in \mathbb{R}$, then
\begin{align}
&\|T_{a} u \|_{H^{s-m}(\mathbb{R})} \lesssim \|a\|_{C_{*}^{-m}(\mathbb{R})} \| u \|_{H^s(\mathbb{R})} \label{HsCmStar}, \\
&\|T_{a} u \|_{H^{s}(\mathbb{R})} \lesssim \|a\|_{L^\infty(\mathbb{R})} \| u \|_{H^s(\mathbb{R})}, \label{HsLinfty}\\
&\|T_{a} u \|_{H^{s-m}(\mathbb{R})} \lesssim \|a\|_{H^{-m}(\mathbb{R})} \| u \|_{C_{*}^s(\mathbb{R})}, \label{HsHmCStar}\\
&\|T_{a} u \|_{C_{*}^{s-m}(\mathbb{R})} \lesssim \|a\|_{C_{*}^{-m}(\mathbb{R})} \| u \|_{C_{*}^s(\mathbb{R})}, \label{CsCmStar}\\
&\|T_{a} u \|_{C_{*}^{s}(\mathbb{R})} \lesssim \|a\|_{L^\infty(\mathbb{R})} \| u \|_{C_{*}^s(\mathbb{R})}. \label{CsLInfty}
\end{align}
\end{enumerate}
\end{lemma}

Using the above paraproduct estimates, we get the following result.
\begin{lemma}[\hspace{1sp}\cite{MR2768550}]
\begin{enumerate}
\item If $s \geq 0$, then
\begin{align}
&\|uv\|_{H^s} \lesssim \|u\|_{H^s}\|v\|_{L^\infty}+ \|u\|_{L^\infty}\|v\|_{H^s},\label{HsProduct} \\
&\|uv\|_{C_{*}^s} \lesssim \|u\|_{C_{*}^s}\|v\|_{L^\infty}+ \|u\|_{L^\infty}\|v\|_{C_{*}^s}. \label{CsProduct}
\end{align}    
\item Let $F\in C^\infty(\mathbb{C}^N)$ be a smooth function that satisfies $F(0) = 0$.
There exists a nondecreasing function $\mathcal{F}: \mathbb{R}_{+} \rightarrow \mathbb{R}_{+}$ such that,
\begin{align}
&\|F(u) \|_{H^s} \leq \mathcal{F}(\|u\|_{L^\infty}) \|u\|_{H^s},\quad s\geq 0,  \label{MoserOne}\\
&\|F(u) \|_{C_{*}^s} \leq \mathcal{F}(\|u\|_{L^\infty}) \|u\|_{C_{*}^s},\quad s>0. \label{MoserTwo}
\end{align}
\item Let $s_1 > s_2 > 0$, then
\begin{equation}
    \|uv\|_{C^{-s_2}_*}\lesssim \|u\|_{C^{s_1}_*} \|v\|_{C^{s_2}_*}. \label{CNegativeAlpha}
\end{equation}
\end{enumerate}
\end{lemma}

When we compute normal form variables later, we need the estimates for low-high and balanced paradifferential bilinear forms.
Let $\chi_1(\theta_1, \theta_2), \chi_2(\theta_1, \theta_2)$ be two non-negative smooth bump functions 
\begin{equation}
    \chi_1(\theta_1, \theta_2) = \left\{
\begin{aligned}
1, \text{ when } |\theta_1|\leq \frac{1}{20}|\theta_2|\\
0, \text{ when } |\theta_1|\geq \frac{1}{10}|\theta_2|,
\end{aligned}
\right.  \label{ChiOnelh}
\end{equation}
\begin{equation}
  \chi_2(\theta_1, \theta_2) = \left\{
\begin{aligned}
&1, \text{ when } \frac{1}{10}\leq \frac{|\theta_1|}{|\theta_2|} \leq 10\\
&0, \text{ when } |\theta_1|\leq \frac{1}{20}|\theta_2| \text{ or } |\theta_2|\leq \frac{1}{20}|\theta_1|,
\end{aligned}
\right.  \label{ChiTwohh}
\end{equation}
and such that $\chi_1(\theta_1, \theta_2) + \chi_1(\theta_2, \theta_1) + \chi_2(\theta_1, \theta_2) =1$.
For a bilinear form $B(u,v)$ with symbol $m(\xi, \eta)$, we  define  the paradifferential bilinear forms:
\begin{itemize}
\item Low-high part and balanced part of the holomorphic bilinear forms:
\begin{align*}
 \widehat{B_{lh}(u,v)}(\zeta) &= \int_{\zeta = \xi +\eta} \chi_1\left(\xi, \eta+\xi\right) m(\xi, \eta)\hat{u}(\xi)\hat{v}(\eta) d\xi,\\
  \widehat{B_{hh}(u,v)}(\zeta) &= \int_{\zeta = \xi +\eta} \chi_2\left(\xi, \eta+\xi\right) m(\xi, \eta)\hat{u}(\xi)\hat{v}(\eta) d\xi.
\end{align*}
\item Low-high part and balanced part of the mixed bilinear forms:
\begin{align*}
 \widehat{B_{lh}(u,v)}(\eta) &= 1_{\eta<0}\int_{\eta = \zeta-\xi} \chi_1\left(\xi, \zeta-\xi\right) m(\xi, \zeta)\bar{\hat{u}}(\xi)\hat{v}(\zeta) d\xi,\\
  \widehat{B_{hh}(u,v)}(\eta) &=1_{\eta<0}\int_{\eta = \zeta-\xi} \chi_2\left(\xi, \zeta-\xi\right) m(\xi, \zeta)\bar{\hat{u}}(\xi)\hat{v}(\zeta) d\xi.
\end{align*}
\end{itemize}

These represent low-high and balanced paradifferential parts of the bilinear forms $B(u,v)$, respectively $\nP B(\bar{u},v)$, restricted to the holomorphic class. 
We will always assume that bilinear symbols $m$ are homogeneous, and smooth away from $(0,0)$.

We have the following estimates for balanced and low-high bilinear forms, which can be seen as generalizations of paraproduct estimates.

\begin{lemma}[\hspace{1sp}\cite{wan2024}] \label{t:SymbolPara}
Let $B^\mu(f,g)$ be a homogeneous  bilinear form of order $\mu\geq 0$ as above. 
For the balanced bilinear forms, when $\alpha+\beta+\mu>0$, $\mu_1 + \mu_2 = \mu$
\begin{align}
&\|B^\mu_{hh}(f, g)\|_{H^{\alpha + \beta}} \lesssim \| f\|_{C_{*}^{\alpha + \mu_1}} \| g\|_{H^{\beta + \mu_2}}, \label{HHBilinearHCH}\\
& \|B^\mu_{hh}(f, g)\|_{C_{*}^{\alpha + \beta}} \lesssim \| f\|_{C_{*}^{\alpha+\mu_1}} \| g\|_{C_{*}^{\beta+\mu_2}}.  \label{HHBilinearCCC}
\end{align}

For the estimates of low-high bilinear forms,
\begin{align}
&\|B^\mu_{lh}(f, g) \|_{H^{s-m}} \lesssim \|f\|_{C_{*}^{-m}} \| g \|_{H^{s+\mu}} \label{BFHsCmStar}, \\
&\|B^\mu_{lh}(f, g) \|_{H^{s}} \lesssim \|f\|_{L^\infty(\mathbb{R})} \| g \|_{H^{s+\mu}}, \label{BFHsLinfty}\\
&\|B^\mu_{lh}(f, g) \|_{C_{*}^{s-m}} \lesssim \|f\|_{C_{*}^{-m}} \| g \|_{C_{*}^{s+\mu}}, \label{BFCsCmStar}\\
&\|B^\mu_{lh}(f, g) \|_{C_{*}^{s}} \lesssim \|f\|_{L^\infty} \| g \|_{C_{*}^{s+\mu}}. \label{BFCsLInfty}     
\end{align}
\end{lemma}

\subsection{Derivation of water waves in holomorphic coordinates} \label{s:Eqn}
In this section, we derive the water wave system with point vortices in holomorphic coordinates  following closely the work of Hunter-Ifrim-Tataru \cite{MR3535894}, but of course adding the point vortices dynamics into the system.
This is essentially a system of nonlinear dispersive partial differential equations coupled with $N$ ordinary differential equations that describe the motion of $N$ point vortices.
When there is no point vortex, the system is reduced to the irrotational water waves studied in \cite{MR3667289}.

We first use complex analysis to transform the lower half-plane $\mathbb{H}:= \{\alpha + i\beta| \beta \leq 0 \}$ to the fluid domain $\Omega_t$.
At each time $t$, by the Riemann mapping theorem, there exists a unique bijective conformal map $h(t, \cdot): \mathbb{H} \rightarrow \Omega_t$ that maps the lower half-plane $\mathbb{H}$ to the fluid domain $\Omega_t$, see Figure \ref{f:Figure}.

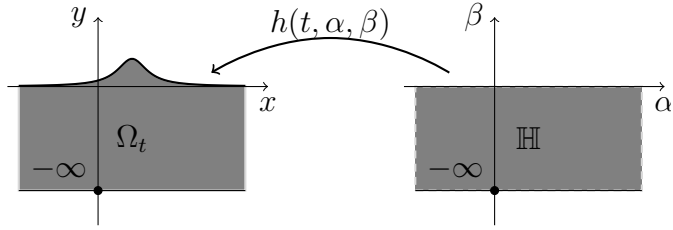
\begin{figure}[!ht] 
  \centering
 \begin{tikzpicture}
\begin{axis}[ xmin=-12, xmax=47.5, ymin=-2.5, ymax=1.5, axis x
      line = none, axis y line = none, samples=100, width=300, height=150]

      \addplot+[mark=none,domain=-10:10,stack plots=y]
      {0.4/(x*x*0.4+1)};
      \addplot+[mark=none,fill=gray,draw=gray!30,thick,domain=-10:10,stack
      plots=y] {-1.5-0.4/(x*x*0.4+1)} \closedcycle;
      \addplot+[black, thick,mark=none,domain=-10:10,stack plots=y]
      {1.5+0.4/(x*x*0.4+1)}; \addplot+[black,
      mark=none,domain=-10:10,stack plots=y]
      {-1.5-0.4/(x*x*0.4+1)};

      \draw[black,->] (axis cs:-3,-2) -- (axis cs:-3,1) node[left] {\(y\)};
      \draw[black,->] (axis cs:-11,0) -- (axis cs:12,0) node[below] {\(x\)};
      \filldraw (axis cs:-3,-1.5) circle (1.5pt) node[above left]
      {\(-\infty\)}; \node at (axis cs:0,-0.75) {\(\Omega_t\)};

     \addplot+[mark=none,domain=25:45]
      {0};
      \addplot+[mark=none,fill=gray,draw=gray!30,thick,domain=25:45] {-1.5} \closedcycle;
      \draw[black] (axis cs:25,-1.5) -- (axis cs:45,-1.5);

      \draw[black,->] (axis cs:32,-2) -- (axis cs:32,1) node[left] {\(\beta\)};
      \draw[black,->] (axis cs:24,0) -- (axis cs:47,0) node[below] {\(\alpha\)};
      \filldraw (axis cs:32,-1.5) circle (1.5pt) node[above left]
      {\(-\infty\)}; \node at (axis cs:35,-0.75) {\(  \mathbb{H}\)};

\draw [->, black, thick] plot [smooth, tension=1] coordinates { (axis cs:28,0.2) (axis cs:17.5,0.7) (axis cs:7,0.2)};

\node at (axis cs:17.5,0.9) {\(h(t,\alpha,\beta)\)};

\end{axis}
\end{tikzpicture}
 \caption{The conformal mapping of the domain} \label{f:Figure}
\end{figure}

The conformal map $h(t, \alpha, \beta)$ satisfies 
\begin{equation*}
    h(t, \alpha, \beta) - (\alpha + i\beta) \rightarrow 0, \quad \text{as } \alpha + i\beta \rightarrow \infty.  
\end{equation*}
The real and imaginary parts of the conformal map $x(t, \alpha, \beta)$ and $y(t, \alpha, \beta)$ satisfy the Cauchy-Riemann equations
\begin{equation}
    x_\alpha = y_\beta, \quad x_\beta = -y_\alpha. \label{CREqn}
\end{equation}

For the complex fluid velocity $v = v_1 +iv_2$ in $\Omega_t$, it can be decomposed into the irrotational part $v^{ir}$ and the rotational part $v^{rot}: = \sum_{j=1}^N v^{rot}_j$ caused by point vortices:
\begin{equation*}
    v = v^{ir} + \sum_{j=1}^Nv^{rot}_j, \quad v^{rot}_j(t, z): = \frac{\lambda_j i}{2\pi} \frac{1}{\overline{z-z_j(t)}}, \quad \text{for } z\neq z_j(t).
\end{equation*}
Then there exists a velocity potential $\phi$ of the fluid such that it can be further decomposed into the rotational and irrotational parts:
\begin{equation*}
\phi = \phi^{ir} + \sum_{j=1}^N \phi_j^{rot}, \quad (\partial_x+i\partial_y) \phi^{ir} = v^{ir}, \quad (\partial_x+i\partial_y) \phi^{rot}_j = v^{rot}_j.
\end{equation*}
Let $\psi = \phi\circ h: \mathbb{H} \rightarrow \mathbb{R}$ be the conformal velocity potential
\begin{equation*}
    \psi(t, \alpha, \beta) = \phi(t, x(t, \alpha, \beta), y(t,\alpha, \beta)).
\end{equation*}
Again, we can decompose the fluid potential $\psi$ into the irrotational and rotational parts $\psi = \psi^{ir} + \sum_{j=1}^N \psi_j^{rot}.$
Using the chain rule,
\begin{equation*}
\psi_\alpha = \phi_x x_\alpha + \phi_y y_\alpha, \quad \psi_\beta = \phi_x x_\beta + \phi_y y_\beta,
\end{equation*}
and the Cauchy-Riemann relations \eqref{CREqn}, the real and imaginary parts of the complex fluid velocity $v = v_1+iv_2 = \phi_x +i \phi_y$ are given by
\begin{equation} \label{uvformula}
    v_1 = \frac{1}{j}(x_\alpha \psi^{ir}_\alpha + x_\beta\psi^{ir}_\beta)+ \sum_{j=1}^N\Re \{ v^{rot}_j \circ h \},\quad v_2 = \frac{1}{j}(y_\alpha \psi^{ir}_\alpha + y_\beta\psi^{ir}_\beta) + \sum_{j=1}^N\Im \{ v^{rot}_j \circ h \},
\end{equation}
for the Jacobian $j = x^2_\alpha + y^2_\alpha.$

We parametrize the free surface using the conformal variables $X(t, \alpha): = x(t, \alpha, 0)$, and $Y(t, \alpha): = y(t, \alpha, 0)$.
Since $(x-\alpha) + i (y-\beta)$  is holomorphic in the lower half-plane and vanishes at infinity, we have by choosing $\beta = 0$, 
\begin{equation}
    X -\alpha = HY, \quad Y = -H(X-\alpha), \label{XYHilbert}
\end{equation}
where $H$ is the Hilbert transform.

Due to the incompressibility condition for the irrotational part of the velocity $\nabla \cdot \mathbf{u} = 0$, $\phi^{ir}$ and therefore $\psi^{ir}$ are harmonic.
One can find $\theta^{ir}(t, \alpha, \beta)$ as the harmonic conjugate of $\psi^{ir}$ such that they satisfy the Cauchy-Riemann equations,
\begin{equation*}
 \psi^{ir}_\alpha = \theta^{ir}_\beta, \quad \psi^{ir}_\beta = -\theta^{ir}_\alpha.
\end{equation*}
Let $\Psi(t, \alpha) = \psi^{ir}(t,\alpha,0)$ denote the boundary value of the conformal irrotational velocity potential,
and $\Theta(t, \alpha) = \theta^{ir}(t, \alpha,0)$ be the boundary value of its harmonic conjugate.
Then,
\begin{equation}
    \psi^{ir}_\beta|_{\beta = 0} = -\Theta_\alpha = H \Psi_\alpha. \label{Conjugate}
\end{equation}

We write $Z(t, \alpha) = h(t, \alpha, 0) = X(t,\alpha) + iY(t,\alpha)$.
The equations \eqref{Helmholtz} become 
\begin{equation}
\frac{d}{dt}z_j(t) = \mathcal{U}(t, z_j(t)) + \sum_{k\neq j} \frac{\lambda_k i}{2\pi}\frac{1}{\overline{z_j(t) - z_k(t)}}, \quad \forall 1\leq j \leq N, \label{Kirchhoff}
\end{equation}
where $\mathcal{U}(t, z)$ is the irrotational part of the complex velocity at $z$, and is given using the Cauchy's integral formula and \eqref{uvformula}, \eqref{Conjugate},
\begin{equation*}
\mathcal{U}(t,z) = \frac{i}{2\pi}\int_{\Gamma_t}\frac{(v_1^{ir}+iv_2^{ir})(t,Z)}{Z-z}\,dZ = \frac{i}{2\pi}\int_{\mathbb{R}} \frac{Z_\alpha\cdot(\Psi_\alpha - i\Theta_\alpha)(t,\alpha)}{\bar{Z}_\alpha(Z(t, \alpha) -z)}\,d\alpha.  
\end{equation*}
In the following, we will derive the equations representing the kinematic and dynamic boundary conditions on the free surface.

The kinematic boundary condition says that  the normal component of the velocity of the surface $\Gamma_t$ is equal to the normal component of the fluid velocity.
In other words,
\begin{equation*}
(X_t, Y_t)\cdot (-Y_\alpha, X_\alpha) = (v_1,v_2) \cdot (-Y_\alpha, X_\alpha) \quad \text{on } \Gamma_t.
\end{equation*}
Applying \eqref{uvformula} and \eqref{Conjugate} in the above equation, we get
\begin{equation}
    X_\alpha Y_t -Y_\alpha X_t = -\Theta_\alpha + \frac{X_\alpha}{2\pi} \sum_{j=1}^N \Re \left\{ \frac{\lambda_j}{\overline{Z-z_j(t)}} \right\} + \frac{Y_\alpha}{2\pi} \sum_{j=1}^N \Im \left\{ \frac{\lambda_j}{\overline{Z-z_j(t)}} \right\}. \label{XYOne}
\end{equation}
Note that both $h_t$ and $h_\alpha$ are holomorphic functions on the lower complex plane, then $\frac{h_t}{h_\alpha}$ is also holomorphic and decays at infinity.
The boundary value of $\frac{h_t}{h_\alpha}$ restricted on the real line is given by
\begin{equation*}
    \frac{X_\alpha X_t + Y_\alpha Y_t}{J} + i \frac{X_\alpha Y_t - Y_\alpha X_t}{J}, \quad J: = X_\alpha^2 + Y^2_\alpha = |Z_\alpha|^2.
\end{equation*}
For the holomorphic function $\frac{Z_t}{Z_\alpha}$, it satisfies \eqref{ReIm}, so that from \eqref{XYOne}, one gets
\begin{equation} \label{XYTwo}
\begin{aligned}
 X_\alpha X_t + Y_\alpha Y_t = &-J H \left[ \frac{\Theta_\alpha}{J} \right]  + JH\left[\frac{X_\alpha}{2\pi J} \sum_{j=1}^N \Re \left\{ \frac{\lambda_j}{\overline{Z-z_j(t)}} \right\}\right] \\
 &+ JH\left[\frac{Y_\alpha}{2\pi J} \sum_{j=1}^N \Im \left\{ \frac{\lambda_j}{\overline{Z-z_j(t)}} \right\}\right].
 \end{aligned}
\end{equation}
Solving equations \eqref{XYOne} and \eqref{XYTwo} for $X_t$ and $Y_t$, we obtain
\begin{equation*}
\begin{aligned}
X_t  = & -H\left[ \frac{\Theta_\alpha}{J} \right]X_\alpha+ \frac{\Theta_\alpha}{J}Y_\alpha - \frac{X_\alpha Y_\alpha}{2\pi J}\sum_{j=1}^N \Re \left\{ \frac{\lambda_j}{\overline{Z-z_j(t)}} \right\} - \frac{Y_\alpha^2}{2\pi J}\sum_{j=1}^N \Im \left\{ \frac{\lambda_j}{\overline{Z-z_j(t)}} \right\}\\
& + X_\alpha H\left[\frac{X_\alpha}{2\pi J} \sum_{j=1}^N \Re \left\{ \frac{\lambda_j}{\overline{Z-z_j(t)}} \right\}\right] + X_\alpha H\left[\frac{Y_\alpha}{2\pi J} \sum_{j=1}^N \Im \left\{ \frac{\lambda_j}{\overline{Z-z_j(t)}} \right\}\right]         ,\\
Y_t  = & -\frac{\Theta_\alpha}{J}X_\alpha  -H\left[ \frac{\Theta_\alpha}{J} \right]Y_\alpha + \frac{X^2_\alpha}{2\pi J}\sum_{j=1}^N \Re \left\{ \frac{\lambda_j}{\overline{Z-z_j(t)}} \right\} + \frac{X_\alpha Y_\alpha}{2\pi J}\sum_{j=1}^N \Im \left\{ \frac{\lambda_j}{\overline{Z-z_j(t)}} \right\} \\
&  + Y_\alpha H\left[\frac{X_\alpha}{2\pi J} \sum_{j=1}^N \Re \left\{ \frac{\lambda_j}{\overline{Z-z_j(t)}} \right\}\right] + Y_\alpha H\left[\frac{Y_\alpha}{2\pi J} \sum_{j=1}^N \Im \left\{ \frac{\lambda_j}{\overline{Z-z_j(t)}} \right\}\right].
\end{aligned}
\end{equation*}
By further choosing the unknown $Q = \Psi + i\Theta$, and using the holomorphic projection $\mathbf{P} = \frac{1}{2}(I -iH)$, the above two real-valued equations can be rewritten as the complex-valued equation
\begin{equation}
    Z_t + \nP \left[ \frac{Q_\alpha - \bar
{Q}_\alpha}{J}\right]Z_\alpha =  2i Z_\alpha \nP \Im \left[ \frac{1}{ \bar{Z}_\alpha} \sum_{j=1}^N \frac{\lambda i}{2\pi}\frac{1}{ Z - z_j}\right], \label{Zeqn}
\end{equation}
where we use the fact that
\begin{align*}
    &  i Z_\alpha \left(\frac{Y_\alpha}{2\pi J}\sum_{j=1}^N\Im \frac{\lambda_j}{\overline{Z-z_j}}- iH\left[ \frac{Y_\alpha}{2\pi J}\sum_{j=1}^N\Im \frac{\lambda_j}{\overline{Z-z_j}}\right] +\frac{X_\alpha}{2\pi J}\sum_{j=1}^N\Re \frac{\lambda_j}{\overline{Z-z_j}} - iH\left[ \frac{X_\alpha}{2\pi J}\sum_{j=1}^N\Re \frac{\lambda_j}{\overline{Z-z_j}}\right]\right)\\
    &= iZ_\alpha \nP \left[\frac{Y_\alpha}{\pi J}\sum_{j=1}^N\Im \frac{\lambda_j}{\overline{Z-z_j}} + \frac{X_\alpha}{\pi J}\sum_{j=1}^N\Re \frac{\lambda_j}{\overline{Z-z_j}}\right] \\
    & = iZ_\alpha \nP \Re\left[\frac{X_\alpha - iY_\alpha}{\pi J}\cdot \left(\sum_{j=1}^N\Re \frac{\lambda_j}{\overline{Z-z_j}} + i \sum_{j=1}^N\Im \frac{\lambda_j}{\overline{Z-z_j}}\right) \right] \\
    & = \frac{i}{\pi}Z_\alpha \nP \Re\left[ \frac{1}{Z_\alpha} \sum_{j=1}^N \frac{\lambda_j}{\overline{Z-z_j}}\right] = 2i Z_\alpha \nP \Im \left[ \frac{i}{2\pi \bar{Z}_\alpha} \sum_{j=1}^N \frac{\lambda_j}{Z - z_j}\right].
\end{align*}
Define the \textit{advection velocity} $b$ to be $ b := b_1+ b_2$,
\begin{equation} \label{bDef}
b_1 := \nP\bigg[\frac{Q_\alpha}{J}\bigg] + \bar{\nP}\bigg[\frac{\bar{Q}_\alpha}{J}\bigg], \quad b_2 : = -\nP \left[ \frac{i}{2\pi \bar{Z}_\alpha} \sum_{j=1}^N \frac{\lambda_j}{Z - z_j}\right] + \bar{\nP} \left[ \frac{i}{2\pi Z_\alpha} \sum_{j=1}^N \frac{\lambda_j}{ \overline{Z - z_j}}\right].
\end{equation}
The equation \eqref{Zeqn} can also be written as
\begin{equation} \label{ZEqnTwo}
    (\partial_t + b\partial_\alpha)Z = \bar{R} + \sum_{j=1}^N v^{rot}_j \circ Z,
\end{equation}
where $R: = \frac{Q_\alpha}{Z_\alpha}$, and two terms on the right-hand side of \eqref{ZEqnTwo} are the irrotational and rotational parts of the complex velocities on the surface.

We now  differentiate \eqref{ZEqnTwo} with respect to $\alpha$ and simplify the equation, we have
\begin{align*}
 &Z_{\alpha t}+ bZ_{\alpha \alpha} + \frac{1}{\bar{Z}_\alpha}\left(Q_{\alpha \alpha} - \frac{Q_\alpha}{Z_\alpha}Z_{\alpha \alpha} \right) 
 \\
 =& Z_\alpha \bar{\nP} \left[ \frac{Q_\alpha - \bar 
{Q}_\alpha}{J}\right]_\alpha + \left(\frac{\bar{Q}_\alpha}{\bar{Z}_\alpha} \right)_\alpha
+\frac{Q_\alpha \bar{Z}_{\alpha \alpha}}{\bar{Z}^2_\alpha}- \sum_{j=1}^N \frac{\lambda_j i}{2\pi}\frac{\bar{Z}_\alpha}{(\overline{Z-z_j})^2}.
\end{align*}
Choosing the differentiated holomorphic position $\W: = Z_\alpha -1$ and $R = \frac{Q_\alpha}{Z_\alpha}$, we get the equation that represents the kinematic boundary condition:
\begin{equation} \label{WEqn}
\W_t + b\W_\alpha + \frac{(1+\W)R_\alpha}{1+\bar{\W}} = (1+\W)M - \sum_{j=1}^N \frac{\lambda_j i}{2\pi}\frac{1+\bar{\W}}{(\overline{Z-z_j})^2},
\end{equation}
where the real-valued auxiliary function $M$ is given by
\begin{equation} \label{MDef}
M : = \frac{R_\alpha}{1+\bar{\W}} + \frac{\bar{R}_\alpha}{1+\W} - b_\alpha. 
\end{equation}
The complex conjugate of $R(t, \cdot)$ is just the irrotational part of the complex velocity on the surface $\Gamma_t$.

The dynamic boundary condition says the pressure of 
the fluid on the free surface is proportional to the mean curvature of the surface.
Assuming that the density of the fluid is $1$, the dynamic boundary condition becomes the  Bernoulli’s equation
\begin{equation} \label{Bernoulli}
\phi_t + \frac{1}{2}|\nabla \phi|^2 + gy = \sigma \mathbf{H} \quad \text{on } \Gamma_t. 
\end{equation}
For the first term of \eqref{Bernoulli}, we define $\Psi^{rot}(t, \alpha) = \phi^{rot}(t, \alpha, 0)$.
If $f(t, \cdot): \Omega_t \rightarrow \mathbb{C}$ is a time-dependent function, then for the composition $g(t,\cdot): = f(t,\cdot)\circ h: \mathbb{H}\rightarrow \mathbb{C}$, one has the chain rule
\begin{equation*}
    f_t = g_t -x_t f_x -y_tf_y.
\end{equation*}
Using this property of the composition and \eqref{uvformula}, \eqref{XYOne}, \eqref{XYTwo}, we compute
\begin{align*}
    &\phi_t|_{\beta = 0} = \Psi_t + \Psi^{rot}_t - (v_1|_{\beta = 0}X_t + v_2|_{\beta = 0}Y_t) \\
    =& \Psi_t + \Psi^{rot}_t - \frac{X_\alpha X_t \Psi_\alpha + X_\beta X_t \Psi_\beta}{J} - \frac{Y_\alpha Y_t \Psi_\alpha + Y_\beta Y_t \Psi_\beta}{J} - X_t\sum_{j=1}^N \Re \{ v_j^{rot}\circ Z\} 
 \\&- Y_t\sum_{j=1}^N \Im \{ v_j^{rot}\circ Z\} + H\left[\frac{Y_\alpha}{J} \sum_{j=1}^N\Re \{v_j^{rot}\circ Z \}\right]\Psi_\alpha - H\left[\frac{X_\alpha}{J} \sum_{j=1}^N\Im \{v_j^{rot}\circ Z \}\right]\Psi_\alpha \\ 
 &- \frac{Y_\alpha}{J} \sum_{j=1}^N\Re \{v_j^{rot}\circ Z \}\Theta_\alpha + \frac{X_\alpha}{J}\sum_{j=1}^N\Im \{v_j^{rot}\circ Z \}\Theta_\alpha\\
    =& \Psi_t + \Psi^{rot}_t + H\left[ \frac{\Theta_\alpha}{J}\right]\Psi_\alpha - \frac{1}{J}\Theta^2_\alpha -\Re \left\{ \bar{Z}_t \cdot \sum_{j=1}^Nv_j^{rot}\circ Z \right\} +2 \Im \left \{\nP \Im \left[\frac{\sum_{j=1}^N v_j^{rot}\circ Z}{Z_\alpha}\right] Q_\alpha \right\}.
\end{align*}
Here, we use the computation 
\begin{align*}
&H\left[\frac{Y_\alpha}{J} \sum_{j=1}^N\Re \{v_j^{rot}\circ Z \} -\frac{X_\alpha}{J} \sum_{j=1}^N\Im \{v_j^{rot}\circ Z \}\right]\Psi_\alpha - \left[ \frac{Y_\alpha}{J} \sum_{j=1}^N\Re \{v_j^{rot}\circ Z \}  -\frac{X_\alpha}{J}\sum_{j=1}^N\Im \{v_j^{rot}\circ Z\}\right]\Theta_\alpha \\
&= H \Re \left[ \frac{Y_\alpha + iX_\alpha}{J} \sum_{j=1}^N v_j^{rot}\circ Z \right]\Psi_\alpha - \Re \left[ \frac{Y_\alpha + iX_\alpha}{J} \sum_{j=1}^N v_j^{rot}\circ Z \right]\Theta_\alpha \\
& = \Im \left[\frac{\sum_{j=1}^N v_j^{rot}\circ Z}{Z_\alpha} \right]\Theta_\alpha -  H\Im \left[\frac{\sum_{j=1}^N v_j^{rot}\circ Z}{Z_\alpha} \right]\Psi_\alpha =  \Re \left\{-i(1-iH)\Im\left [\frac{\sum_{j=1}^N v_j^{rot}\circ Z}{Z_\alpha}\right] Q_\alpha \right\} \\
& = 2 \Im \left \{\nP\Im\left[\frac{\sum_{j=1}^N v_j^{rot}\circ Z}{Z_\alpha}\right] Q_\alpha \right\}.
\end{align*}

For the second term of \eqref{Bernoulli}, we have
\begin{align*}
 \frac{1}{2}|\nabla \phi(t, \cdot, 0)|^2 = \frac{1}{2}(u^2 +v^2)|_{\beta = 0} = \frac{1}{2J}(\Psi^2_\alpha + \Theta^2_\alpha) + \frac{1}{2} \Big|\sum_{j=1}^N v_j^{rot}\circ Z\Big|^2 + \Re \left\{R \cdot \sum_{j=1}^Nv_j^{rot}\circ Z \right\} .
\end{align*}

In holomorphic coordinates, the mean curvature is a function of $Z_\alpha$, see \cite{MR3667289},
\begin{equation*}
    \mathbf{H} = \frac{-i}{2(Z_\alpha + \bar{Z}_\alpha)}\left( \frac{Z_\alpha - \bar{Z}_\alpha}{|Z_\alpha|}\right)_\alpha = -\frac{i}{2} \left( \frac{Z_{\alpha \alpha}}{J^{\frac{1}{2}}Z_\alpha} -  \frac{\bar{Z}_{\alpha \alpha}}{J^{\frac{1}{2}}\bar{Z}_\alpha} \right).
\end{equation*}

Using above computations, the Bernoulli’s equation \eqref{Bernoulli} becomes
\begin{align*}
&\Psi_t  + H\left[ \frac{\Theta_\alpha}{J}\right]\Psi_\alpha - \frac{1}{2J}(\Psi_\alpha^2 - \Theta^2_\alpha ) + gY + \Re \left\{b \bar{Z}_\alpha \cdot \sum_{j=1}^Nv_j^{rot}\circ Z \right\}\\
  -& \frac{1}{2} \left |\sum_{j=1}^N v_j^{rot}\circ Z\right |^2 + \Psi^{rot}_t +2 \Im \left \{\nP \Im \left[\frac{\sum_{j=1}^N v_j^{rot}\circ Z}{Z_\alpha}\right] Q_\alpha \right\} +\frac{i\sigma}{2} \left( \frac{Z_{\alpha \alpha}}{J^{\frac{1}{2}}Z_\alpha} -  \frac{\bar{Z}_{\alpha \alpha}}{J^{\frac{1}{2}}\bar{Z}_\alpha} \right)= 0. 
\end{align*}
Applying the operator $I-iH$ to the above equation of $\Psi$, and using the convolution identity of the Hilbert transform \eqref{HilbertCon}, we obtain the complex-valued equation
\begin{equation} \label{QEqn}
\begin{aligned}
  &(\partial_t + b\partial_\alpha)Q+ 2\nP \Psi_t^{rot}  -ig(Z-\alpha) -\nP \left| \sum_{j=1}^N v_j^{rot}\circ Z\right|^2  \\
  = &  \bar{\nP}\left[ \frac{|Q_\alpha|^2}{J} \right]+\frac{\sum_{j=1}^N v_j^{rot}\circ Z}{Z_\alpha} Q_\alpha +2\sigma \nP \Im\left[ \frac{Z_{\alpha \alpha}}{J^{\frac{1}{2}}Z_\alpha} \right].
\end{aligned}
\end{equation}

Note that $2\nP \Psi^{rot}_\alpha$ term can be written as
\begin{align*}
2\nP \Psi^{rot}_\alpha &= (1-iH)\Psi^{rot}_\alpha = \psi_\alpha^{rot}|_{\beta=0} -i\psi_\beta^{rot}|_{\beta=0} \\
&= \psi_x^{rot}|_{\beta =0}X_\alpha + \psi_y^{rot}|_{\beta = 0} Y_\alpha -i(\psi_x^{rot}|_{\beta = 0}X_\beta +\psi_y^{rot}|_{\beta =0}Y_\beta) \\
&= \psi_x^{rot}|_{\beta = 0}(X_\alpha + iY_\alpha) -i\psi_y^{rot}|_{\beta = 0}(X_\alpha + iY_\alpha) \\
&= (\psi_x^{rot}-i\psi_y^{rot})|_{\beta = 0}(X_\alpha + iY_\alpha) = Z_\alpha\sum_{j=1}^N\overline{v_j^{rot}\circ Z}.
\end{align*}
Hence, we have
\begin{align*}
    &\frac{2}{1+\W}\partial_t \nP \Psi^{rot}_\alpha = \partial_t \sum_{j=1}^N\overline{v_j^{rot}\circ Z} + \frac{\W_t}{1+\W}\sum_{j=1}^N\overline{v_j^{rot}\circ Z} \\
    =&  \partial_t \sum_{j=1}^N\overline{v_j^{rot}\circ Z}+\left(\frac{\bar{R}_\alpha}{1+\W} - \frac{b\W_\alpha}{1+\W} +\frac{1}{1+\W}\partial_\alpha \sum_{j=1}^N v^{rot}_j\circ Z - b_\alpha \right) \sum_{j=1}^N\overline{v_j^{rot}\circ Z}.
\end{align*}

Taking the $\alpha$-derivative of the equation \eqref{QEqn}, using \eqref{ZEqnTwo} and the fact 
\begin{equation*}
 (\partial_t + b\partial_\alpha)R = \frac{1}{1+\W}(\partial_t + b\partial_\alpha)Q_\alpha - \frac{R}{1+\W}(\partial_t + b\partial_\alpha)\W,   
\end{equation*}
 we get that $R$ satisfies the equation
\begin{equation} \label{REqn}
 (\partial_t + b\partial_\alpha)R  - (\partial_t + b\partial_\alpha) \sum_{j=1}^N \frac{\lambda_j i}{2\pi}\frac{1}{Z-z_j} -i\frac{g\W - a}{1+\W}   = \frac{ 2\sigma}{1+\W}\mathbf{P}\Im \left[ \frac{\W_{ \alpha}}{J^{\frac{1}{2}}(1+\W)}\right]_{\alpha},
\end{equation}
where $a$ is the  real \textit{frequency-shift} defined by
\begin{equation} \label{FrequencyShit}
    a: = i\left(\bar{\nP}[(\bar{R}+ v^{rot}\circ Z)\partial_\alpha(R+  \overline{v^{rot} \circ Z})] - \nP[(R+ \overline{v^{rot}\circ Z})\partial_\alpha(\bar{R}+  v^{rot} \circ Z)]\right),
\end{equation}
and 
\begin{equation} \label{vrot}
    v^{rot}(t,z) = \sum_{j=1}^N v^{rot}_j(t,z) = \sum_{j=1}^N \frac{\lambda_j i}{2\pi} \frac{1}{\overline{z(t) - z_j(t)}}. 
\end{equation}

Equations \eqref{Kirchhoff}, \eqref{WEqn} and \eqref{REqn} form the water wave system \eqref{e:WW}.
Once we solve unknowns $(\W, R, \{z_j\}_{j=1}^N)$ at each time $t$, then
\begin{equation*}
    Z(t, \alpha) = \alpha + \int_{-\infty}^\alpha \W(t, y) \,dy,
\end{equation*}
which  determines the fluid domain at each time $t$.
The irrotational and the rotational parts of the complex fluid velocity at point $z = x+iy$ inside the domain are given by \eqref{IrrotationalV} and \eqref{vrot}.
As a consequence, we then solve the incompressible free boundary Euler equations \eqref{Euler} with the kinematic and dynamic boundary conditions.

\section{Water waves related estimates} \label{s:Estimate}
In this section, we first derive some estimates that are related to the water wave system \eqref{e:WW}, and then compute the leading term of the para-material derivative $T_{D_t} := \partial_t + T_b\partial_\alpha$ of $\W, R$, and $J^s$.
These estimates will later be needed when we obtain the a priori energy estimate for water waves.

We start by computing estimates for the rotational part of the complex velocity on the surface.
\begin{lemma} 
For $z_j\in \Omega_t$, and let $h(t, \cdot)$ be the Riemann mapping in Figure \ref{f:Figure}, we set $c_j : = h^{-1}(t, z_j)$.
Then, we have the following identities:
\begin{align}
\bar{\nP} \frac{1}{Z(t,\alpha)-z_j} &= \frac{1}{h_z(t, c_j)(\alpha - c_j)}, \label{projection1}\\
\bar{\nP} \frac{1}{(Z(t,\alpha)-z_j)^2}& = \frac{1}{h_z(t, c_j)^2(\alpha -c_j)^2}. \label{projection2}
\end{align}

As a consequence, for the rotational part of the velocity $v^{rot}$ \eqref{vrot} on the surface, 
\begin{equation} \label{barPvRot}
 \nP (v^{rot}\circ Z) = \sum_{j=1}^N \frac{\lambda_j i}{2\pi} \frac{1}{\overline{h_z(t, c_j)(\alpha - c_j)}}.
\end{equation}
$\nP (v^{rot}\circ Z) \in H^s$ for all $s\geq 0$.
\end{lemma}
\begin{proof}
Recall that $Z(t,\alpha)-z_j$ is the boundary value of $h(t, \alpha, \beta) - z_j$ in the lower half plane.
Since $h(t,\cdot)$ is a bijective map, $h(t, \alpha, \beta) -z_j$ has a unique zero $c_j = h^{-1}(t, z_j)$.
Hence $\frac{1}{Z(t, \alpha)-z_j}$ is a meromorphic function with a pole $h^{-1}(t, z_j)$ of multiplicity one.
For $z$ in a small neighborhood of $c_j$, we have the Taylor series expansion
\begin{equation*}
    h(t, z) - z_j = h_z(t, c_j)(z - c_j) + \sum_{n=2}^\infty c_{n,j} (z -c_j)^n.
\end{equation*}
Therefore, $\frac{1}{Z(t, \alpha) - z_j}-\frac{1}{h_z(t, c_j)(\alpha - c_j)}$ is a holomorphic function on $\mathbb{H}$, and we have
\begin{equation*}
    \bar{\nP} \frac{1}{Z(t,\alpha)-z_j} = \bar{\nP}\frac{1}{h_z(t, c_j)(\alpha - c_j)}.
\end{equation*}
Since $\frac{1}{h_z(t, c_j)(\alpha - c_j)}$ can be extended to be holomorphic in the upper half-plane, $\bar{\nP}\frac{1}{h_z(t, c_j)(\alpha - c_j)} = \frac{1}{h_z(t, c_j)(\alpha - c_j)}$, and this gives the identity \eqref{projection1}.
Similarly, 
\begin{equation*}
    \frac{1}{(Z(t,\alpha)- z_j)^2} - \frac{1}{h_z(t, c_j)^2(\alpha - c_j)^2}
\end{equation*}
is a holomorphic function on $\mathbb{H}$, so that 
\begin{equation*}
    \bar{\nP} \frac{1}{(Z(t,\alpha)-z_j)^2} = \bar{\nP}\frac{1}{h_z(t, c_j)^2(\alpha - c_j)^2} = \frac{1}{h_z(t, c_j)^2(\alpha - c_j)^2}.
\end{equation*}

As for $\nP (v^{rot}\circ Z)$, it is given by
\begin{equation*}
 \nP (v^{rot}\circ Z) = \sum_{j=1}^N \frac{\lambda_j i}{2\pi} \nP \frac{1}{\overline{Z(t,\alpha)-z_j}}   = \sum_{j=1}^N \frac{\lambda_j i}{2\pi} \frac{1}{\overline{h_z(t, c_j)(\alpha - c_j)}}. 
\end{equation*}
Since points $z_j$ are in the interior of $\Omega_t$ and $h(t,\cdot)$ is bijective, $|\alpha - c_j| > 0$, and $\frac{1}{\alpha - c_j} \in L^2$.
Therefore, $\nP (v^{rot}\circ Z) \in H^s\cap C^s_{*}$ for all $s\geq 0$.
\end{proof}

\begin{lemma}
Let $s>0$, then for the rotational part of the complex velocity on the surface,
\begin{equation} \label{PvrotZHsCs}
 \|\bar{\nP}(v^{rot}\circ Z) \|_{H^s} \lesssim_{\CalAZ} \| Z -\alpha\|_{H^{s}}, \quad \|\bar{\nP}(v^{rot}\circ Z) \|_{C^s_{*}} \lesssim_{\CalAZ} \| Z-\alpha\|_{C_{*}^{s}}.
\end{equation}
Hence, $v^{rot}\circ Z$ can be decomposed into the smooth holomorphic part and the antiholomorphic part that can be controlled by $Z-\alpha$ in some appropriate norm.
\end{lemma}

\begin{proof}
 We express
\begin{equation*}
 \bar{\nP}(v^{rot}\circ Z) = v^{rot}\circ Z-\nP(v^{rot}\circ Z) = \sum_{j=1}^N \frac{\lambda_j i}{2\pi} \left(\frac{1}{\overline{Z - z_j}} - \frac{1}{\overline{h_z(t, c_j)(\alpha - c_j)}} \right).
\end{equation*}
When $Z(t, \alpha)- \alpha = 0$, the fluid domain $\Omega_t = \mathbb{H}$, so that $h_z = 1$, and $c_j = z_j$.
Hence,
\begin{equation*}
 \frac{1}{\overline{Z - z_j}} - \frac{1}{\overline{h_z(t, c_j)(\alpha - c_j)}} = 0, \quad \text{when } Z = \alpha.   
\end{equation*}
Using Moser-type estimates \eqref{MoserOne}, \eqref{MoserTwo}, for $s>0$,
\begin{align*}
 &\left\| \frac{1}{\overline{Z - z_j}} - \frac{1}{\overline{h_z(t, c_j)(\alpha - c_j)}} \right\|_{H^s} \lesssim_{\CalAZ} \|Z-\alpha\|_{H^{s}}, \\
 &\left\| \frac{1}{\overline{Z - z_j}} - \frac{1}{\overline{h_z(t, c_j)(\alpha - c_j)}} \right\|_{C_{*}^s}  \lesssim_{\CalAZ} \|Z-\alpha \|_{C_{*}^s}.
\end{align*}
This gives the proof of \eqref{PvrotZHsCs} by adding each term of the complex rotational velocity.
\end{proof}
Using a similar way for the proof, we can show that
\begin{equation} \label{nPZsquare}
 \left\|\nP \frac{1}{(Z(t,\alpha)-z_j)^2}\right\|_{H^s} \lesssim_\CalAZ \|Z-\alpha \|_{H^s}, \quad  \left\|\nP \frac{1}{(Z(t,\alpha)-z_j)^2}\right\|_{C^s_{*}} \lesssim_\CalAZ \|Z-\alpha \|_{C^s_{*}}.
\end{equation}

Define $d_S(t)$ to be the minimum distance of the point vortices to the free surface, 
\begin{equation} \label{DefDistance}
    d_S(t) := \inf_{\alpha \in \R}\min_{1\leq j \leq N} |Z(t,\alpha)- z_j(t)|. 
\end{equation}
We have the following results for $d_S(t)$ from Lemma 2.8 and Corollary 2.1 in \cite{MR4179726}:
\begin{lemma}[\hspace{1sp}\cite{MR4179726}]
Assume that $Z(t, \alpha)$ satisfies the chord-arc condition 
\begin{equation} \label{chordarc}
    C_0|\alpha - \beta|\leq |Z(t,\alpha) - Z(t, \beta)| \leq C_1|\alpha - \beta|.
\end{equation}
Then for $k>1$, 
\begin{equation} \label{Lkestimate}
 \int_{-\infty}^\infty \frac{1}{|Z(t,\beta) - z_j(t)|^k}\,d\beta \leq C d_S(t)^{-k+1},
\end{equation}
where the constant $C = 4C_0^{-1}+ \frac{4C_1^{k-1}}{(k-1)C_0^{k-1}}$. 
Suppose that $d_S(t)\gtrsim 1$, and $s>\frac{3}{2}$, then 
\begin{equation} \label{LkHsEstimate}
 \left\| \frac{1}{(\overline{Z(t, \alpha) - z_j(t)})^k} \right\|_{H^s} \leq C d_S(t)^{-k+\frac{1}{2}},
\end{equation}
where the positive constant $C$ depends on $C_0, C_1, \|\W \|_{C([0,T]; H^{s}(\R))}$. 
\end{lemma}
We remark that the chord-arc condition \eqref{chordarc} holds as long as $\|\W \|_{C([0,T]; C^{\epsilon}_{*}(\R))}$ is small enough.

Define the auxiliary function $\mathcal{Y}: = \frac{\W}{1+\W}$, then $\Y$ is a holomorphic function such that $\Y = \nP \Y$.
Again by Moser-type estimates \eqref{MoserOne}, \eqref{MoserTwo},
\begin{equation} \label{YBounds}
  \|\mathcal{Y}\|_{H^s} \lesssim_{\CalAZ} \|\W \|_{H^s}, \quad  \|\mathcal{Y}\|_{C_{*}^s} \lesssim_{\CalAZ} \|\W \|_{C_{*}^s}. 
\end{equation}

Next, we consider the estimates for the advection velocity $b$ \eqref{bDef} and the frequency-shift $a$ \eqref{FrequencyShit}.
Using the definition of $\Y$, $b$ can be rewritten as 
\begin{equation*}
    b = b_1 + b_2, \quad  b_1=2\Re \nP[(1-\bar{\Y})R], \quad b_2 = 2\Re \bar{\nP}[(1-\Y)v^{rot}].
\end{equation*}
\begin{lemma}
For $s>0$, the advection velocity $b$ is bounded as follows: 
\begin{equation} \label{bEst}
 \|b\|_{C^s_{*}}\lesssim_\CalAZ \|R\|_{C^s_{*}} + \|Z-\alpha \|_{C^s_{*}}, \quad  \|b\|_{H^s}\lesssim_\CalAZ \|R\|_{H^s} + \|Z-\alpha \|_{H^s}.    
\end{equation}
\end{lemma}
\begin{proof}
The estimates for $b_1$ were carried out in \cite{ai2023dimensional,wan2024}; it suffices to consider the bounds for the antiholomorphic part of $b_2$.
We expand the antiholomorphic part as the sum of paraproducts:
\begin{equation*}
 \bar{\nP}[(1-\Y)(v^{rot}\circ Z)] = T_{1-\Y}\bar{\nP}(v^{rot}\circ Z) - \bar{\nP}\Pi(\Y, \bar{\nP}(v^{rot}\circ Z)).
\end{equation*}
By \eqref{HCHEstimate}, \eqref{CCCEstimate}, \eqref{HsLinfty}, \eqref{CsLInfty} we have
\begin{align*}
 &\|T_{1-\Y}\bar{\nP}(v^{rot}\circ Z) \|_{C^s_{*}} + \| \bar{\nP}\Pi(\Y, \bar{\nP}(v^{rot}\circ Z))\|_{C^s_{*}} \lesssim (1+\|\Y \|_{L^\infty}) \|\bar{\nP}(v^{rot}\circ Z) \|_{C^s_{*}} \lesssim_\CalAZ \|Z-\alpha \|_{C^s_{*}}, \\
  &\|T_{1-\Y}\bar{\nP}(v^{rot}\circ Z) \|_{H^s} + \| \bar{\nP}\Pi(\Y, \bar{\nP}(v^{rot}\circ Z))\|_{H^s} \lesssim (1+\|\Y \|_{L^\infty}) \|\bar{\nP}(v^{rot}\circ Z) \|_{H^s} \lesssim_\CalAZ \|Z-\alpha \|_{H^s}.
\end{align*}
Hence, we get estimates in \eqref{bEst}.
\end{proof}

\begin{lemma}
The frequency-shift $a$ satisfies the estimate
\begin{equation}
\|a\|_{C_{*}^{\epsilon}} \lesssim_\CalAZ  \mathcal{A}_1^2, \label{ACHalf}
\end{equation}
as well as the Sobolev estimate
\begin{equation}
\| a\|_{H^{s}}  \lesssim_\CalAZ \CalAT  \left(\sum_{j=1}^N |\lambda_j|+\|R\|_{H^s}+ \|Z-\alpha \|_{H^s}\right), \quad s>0. \label{AHsEst}  
\end{equation}
\end{lemma}
\begin{proof}
It suffices to consider the antiholomorphic part of $a$.
The estimates for $i\bar{\nP}[\bar{R}R_\alpha]$ are given in \cite{ai2023dimensional,wan2024}.
For other parts of $a$, we write
\begin{align*}
    \bar{\nP}[R_\alpha (v^{rot}\circ Z)] =& T_{R_\alpha} \bar{\nP}(v^{rot}\circ Z) + \bar{\nP}\Pi[R_\alpha, \bar{\nP} (v^{rot}\circ Z)], \\
    \bar{\nP}[\bar{R} \partial_\alpha(\overline{v^{rot}\circ Z})] =& T_{\partial_\alpha(\overline{v^{rot}\circ Z})}\bar{R} + T_{\bar{R}}\bar{\nP}\partial_\alpha(\overline{v^{rot}\circ Z}) + \bar{\nP} \Pi[\bar{R}, \partial_\alpha(\overline{v^{rot}\circ Z})],\\
    \bar{\nP}[(v^{rot}\circ Z) \partial_\alpha(\overline{v^{rot}\circ Z})] =& T_{v^{rot}\circ Z}\bar{\nP}\partial_\alpha(\overline{v^{rot}\circ Z}) + T_{\partial_\alpha(\overline{v^{rot}\circ Z})} \bar{\nP}(v^{rot}\circ Z) \\
    &+ \bar{\nP}\Pi[\partial_\alpha(\overline{v^{rot}\circ Z}),  v^{rot}\circ Z].
\end{align*}
We estimate each term using \eqref{HCHEstimate}, \eqref{CCCEstimate}, \eqref{HsCmStar}, \eqref{CsCmStar},
\begin{align*}
 &\|T_{R_\alpha} \bar{\nP}(v^{rot}\circ Z) \|_{C^\epsilon_{*}} +  \|\bar{\nP}\Pi[R_\alpha, \bar{\nP} (v^{rot}\circ Z)] \|_{C^\epsilon_{*}} \lesssim \| R\|_{C^{\frac{1}{2}+\epsilon}_{*}}\|\bar{\nP}(v^{rot}\circ Z) \|_{C^{\frac{1}{2}}_{*}},\\
 &\|T_{R_\alpha} \bar{\nP}(v^{rot}\circ Z) \|_{H^s} +  \|\bar{\nP}\Pi[R_\alpha, \bar{\nP} (v^{rot}\circ Z)] \|_{H^s} \lesssim \| R\|_{C^{1+\epsilon}_{*}}\|\bar{\nP}(v^{rot}\circ Z) \|_{H^{s}}, \\
 &\|T_{\partial_\alpha(\overline{v^{rot}\circ Z})}\bar{R} \|_{C^\epsilon_{*}} + \|\bar{\nP} \Pi[\bar{R}, \partial_\alpha(\overline{v^{rot}\circ Z})] \|_{C^\epsilon_{*}} \lesssim \| R\|_{C^{\frac{1}{2}+\epsilon}_{*}}\|(\nP+\bar{\nP})(\overline{v^{rot}\circ Z}) \|_{C^{\frac{1}{2}}_{*}} ,\\
  &\|T_{\partial_\alpha(\overline{v^{rot}\circ Z})}\bar{R} \|_{H^s} + \|\bar{\nP} \Pi[\bar{R}, \partial_\alpha(\overline{v^{rot}\circ Z})] \|_{H^s} \lesssim \| R\|_{H^{s}}\|(\nP+\bar{\nP})(\overline{v^{rot}\circ Z}) \|_{C^{1}_{*}},\\
  & \| T_{\bar{R}}\bar{\nP}\partial_\alpha(\overline{v^{rot}\circ Z})\|_{C^\epsilon_{*}}  \lesssim \|R\|_{C^\epsilon_{*}}\|\bar{\nP}(\overline{v^{rot}\circ Z}) \|_{C^{1+\epsilon}_{*}}, \\
  &\| T_{\bar{R}}\bar{\nP}\partial_\alpha(\overline{v^{rot}\circ Z})\|_{H^s} \lesssim \|R\|_{C^{-\frac{1}{2}}_{*}}\|\bar{\nP}(\overline{v^{rot}\circ Z}) \|_{H^{s+\frac{1}{2}}} \lesssim \|R\|_{H^s}\|\bar{\nP}(\overline{v^{rot}\circ Z}) \|_{H^{s+\frac{1}{2}}}, \\
  & \|T_{v^{rot}\circ Z}\bar{\nP}\partial_\alpha(\overline{v^{rot}\circ Z}) \|_{C^\epsilon_{*}} \lesssim \|(\nP + \bar{\nP})(v^{rot}\circ Z) \|_{C^\epsilon_{*}}\|\bar{\nP}(\overline{v^{rot}\circ Z}) \|_{C^{1+\epsilon}_{*}}, \\
  & \|T_{v^{rot}\circ Z}\bar{\nP}\partial_\alpha(\overline{v^{rot}\circ Z}) \|_{H^s} \lesssim \|(\nP + \bar{\nP})(v^{rot}\circ Z) \|_{H^s}\|\bar{\nP}(\overline{v^{rot}\circ Z}) \|_{H^{s+\frac{3}{2}}}, \\
  &\|T_{\partial_\alpha(\overline{v^{rot}\circ Z})} \bar{\nP}(v^{rot}\circ Z) \|_{C^\epsilon_{*}} \lesssim \|(\nP + \bar{\nP})(\overline{v^{rot}\circ Z}) \|_{C^\epsilon_{*}}\|\bar{\nP}(v^{rot}\circ Z) \|_{C^{1+\epsilon}_{*}}, \\
   &\|T_{\partial_\alpha(\overline{v^{rot}\circ Z})} \bar{\nP}(v^{rot}\circ Z) \|_{H^s} \lesssim \|(\nP + \bar{\nP})(\overline{v^{rot}\circ Z}) \|_{C^\epsilon_{*}}\|\bar{\nP}(v^{rot}\circ Z) \|_{H^{s+1}},\\
&\|\bar{\nP}\Pi[\partial_\alpha(\overline{v^{rot}\circ Z}),  v^{rot}\circ Z] \|_{C^\epsilon_{*}} \lesssim \|(\nP + \bar{\nP})(v^{rot}\circ Z) \|_{C^\epsilon_{*}}\|(\nP+\bar{\nP})(\overline{v^{rot}\circ Z}) \|_{C^{1+\epsilon}_{*}}, \\
&\|\bar{\nP}\Pi[\partial_\alpha(\overline{v^{rot}\circ Z}),  v^{rot}\circ Z] \|_{H^s} \lesssim \|(\nP + \bar{\nP})(v^{rot}\circ Z) \|_{C^\epsilon_{*}}\|(\nP+\bar{\nP})(\overline{v^{rot}\circ Z}) \|_{H^{s+1}}.
\end{align*}
For each part of the estimate, we use the estimate \eqref{PvrotZHsCs}, and the fact that
\begin{equation*}
 \| \nP(v^{rot}\circ Z)\|_{H^s} \lesssim \sum_{j=1}^N |\lambda_j|, \quad  \| \nP(v^{rot}\circ Z)\|_{C_{*}^s} \lesssim \sum_{j=1}^N |\lambda_j|, \quad \forall s \geq 0.
\end{equation*}
These estimates lead to the bounds \eqref{ACHalf} and \eqref{AHsEst}.
\end{proof}

For the upper bounds of the time derivatives of $z_j$, $1\leq j \leq N$, we have the following result.
\begin{lemma}
The time derivatives of $z_j$, $1\leq j \leq N$ are bounded by
\begin{equation} 
 \left|\frac{d}{dt}z_j(t)\right| \leq \|R(t, \cdot)\|_{L^\infty} + (N-1)\max_{j\neq k} \frac{|\lambda_k|}{2\pi} \frac{1}{|z_j(t) - z_k(t)|} \lesssim \CalAO, \label{zjtimedot} 
\end{equation}
\end{lemma}

\begin{proof}
According to \eqref{Kirchhoff}, 
\begin{equation*}
\left|\frac{d}{dt}z_j(t)\right| \leq |\mathcal{U}(t, z_j)| + \sum_{k\neq j} \frac{|\lambda_k|}{2\pi} \frac{1}{|z_j(t) - z_k(t)|}.
\end{equation*}
$\mathcal{U}(t, \cdot)$ is the extension of the antiholomorphic irrotational part of the complex velocity in $\Omega_t$.
By the maximum principle, $|\mathcal{U}(t, z_j)|\leq \|v^{ir}(t, \cdot)\|_{L^\infty(\Gamma_t)}$, and the fact $ \|v^{ir}(t, \cdot)\|_{L^\infty(\Gamma_t)} = \|R(t, \cdot)\|_{L^\infty}$, we obtain the bound for the time derivative of $z_j$. 
\end{proof}

Define the para-material derivative $T_{D_t} := \partial_t + T_b\partial_\alpha$, it is the paradifferential version of the material derivative $D_t : = \partial_t + b\partial_\alpha$.
In the last part of this section, we express the para-material derivatives of $\W$ and $R$ as the sum of leading terms of order $\frac{3}{2}$, sub-leading terms of order $\frac{1}{2}$ and remaining perturbative terms.

\begin{lemma}
The leading term of the para-material derivative of $\W$ is given by
\begin{equation}
       T_{D_t}\W  = -\partial_\alpha T_{(1+\W)(1-\bar{\Y})}R + G,  \label{WParaMat}
\end{equation}
where for $s>\frac{1}{2}$, the source term $G$  satisfies the bound
\begin{equation*}
    \|G\|_{H^s} \lesssim_\CalAZ \CalAT \left(\|(\W,R)\|_{\mathcal{H}^{s-\frac{1}{2}}}+ \|Z-\alpha \|_{H^{s+1}}\right).
\end{equation*}
\end{lemma}

\begin{proof}
 We rewrite the first equation of \eqref{e:WW} as
\begin{equation} \label{WMatDerivative}
D_t\W = \bar{R}_\alpha - (1+\W)b_\alpha + \partial_\alpha(v^{rot}\circ Z). 
\end{equation}
Writing $\W$ using the paraproduct decomposition and applying the holomorphic projection $\nP$, we obtain the holomorphic paradifferential equation
\begin{equation*}
 T_{D_t}\W = -T_{1+\W}\nP b_\alpha - T_{\W_\alpha} \nP b -T_{b_\alpha}\W -\nP \Pi(b_\alpha, \W) -\nP \Pi(\W_\alpha, b) + \nP\partial_\alpha(v^{rot}\circ Z),
\end{equation*}
where we use the fact that $\nP$ commutes with $T_{D_t}$.

For the first term on the right-hand side of $T_{D_t}\W$, 
\begin{align*}
 &-T_{1+\W}\nP b_\alpha = -T_{1+\W}\nP\partial_\alpha[(1-\bar{\Y})(R+\overline{v^{rot} \circ Z})] \\
 =& - T_{1+\W}\partial_\alpha T_{1-\bar{\Y}}R - T_{1+\W}\partial_\alpha T_{1-\bar{\Y}}\nP(\overline{v^{rot} \circ Z})
 +T_{1+\W}\nP \partial_\alpha \Pi(R+\overline{v^{rot} \circ Z}, \bar{\Y}).
\end{align*}
The first term $- T_{1+\W}\partial_\alpha T_{1-\bar{\Y}}R$ is one of the leading terms of $T_{D_t}\W$. 
For the other three terms, we use inequalities \eqref{HsLinfty}, \eqref{HCHEstimate}, and bounds \eqref{PvrotZHsCs}, \eqref{YBounds}, \eqref{bEst},
\begin{align*}
& \|T_{1+\W}\partial_\alpha T_{1-\bar{\Y}}\nP(\overline{v^{rot} \circ Z}) \|_{H^s}\\
\lesssim& (1+\|\W\|_{C^\epsilon_{*}})(1+\|\Y\|_{C^\epsilon_{*}})\|\nP(\overline{v^{rot} \circ Z}) \|_{H^s} \lesssim_\CalAZ \CalAT\|Z-\alpha \|_{H^s},\\
& \|T_{1+\W}\nP\partial_\alpha \Pi(R+\overline{v^{rot} \circ Z}, \bar{\Y}) \|_{H^s} \lesssim_\CalAZ \|\Y\|_{C^{\frac{3}{2}+\epsilon}_{*}}\|R\|_{H^{s-\frac{1}{2}}}+ \| \Y\|_{C^{\epsilon}_*}\|v^{rot} \circ Z\|_{H^{s+1}}\\
 \lesssim&_\CalAZ \|\Y\|_{C^{\frac{3}{2}+\epsilon}_{*}} \left(\|R\|_{H^{s-\frac{1}{2}}} +\|Z-\alpha \|_{H^{s+1}}\right) + \sum_{j=1}^N|\lambda_j|\|\Y\|_{H^s} \\
 \lesssim&_\CalAZ \CalAT \left(\|(\W, R)\|_{H^{s-\frac{1}{2}}}+\|Z-\alpha\|_{H^{s+1}}\right),
\end{align*}
so that these terms can be put into $G$.
For the second term on the right-hand side of $T_{D_t}\W$, 
\begin{align*}
 &-T_{\W_\alpha}\nP b = -T_{\W_\alpha}\nP[(1-\bar{\Y})(R+\overline{v^{rot} \circ Z})]\\
=& - T_{\W_\alpha} T_{1-\bar{\Y}}R - T_{\W_\alpha} T_{1-\bar{\Y}}\nP(\overline{v^{rot} \circ Z})
 +T_{\W_\alpha}\nP \Pi(R+\overline{v^{rot} \circ Z}, \bar{\Y}).
\end{align*}
We add the first terms of $-T_{\W_\alpha}\nP b$ and $-T_{1+\W}\nP b_\alpha$ together,
\begin{equation*}
- T_{1+\W}\partial_\alpha T_{1-\bar{\Y}}R - T_{\W_\alpha} T_{1-\bar{\Y}}R = -\partial_\alpha T_{(1+\W)(1-\bar{\Y})}R - \partial_\alpha\left(T_{\W\bar{\Y}} - T_{\W}T_{\bar{\Y}}\right)R. 
\end{equation*}
Using the composition for the paradifferential operator \eqref{CompositionPara}, 
\begin{equation*}
\left\|\partial_\alpha\left(T_{\W\bar{\Y}} - T_{\W}T_{\bar{\Y}}\right)R \right\|_{H^s} \lesssim \left(\|\W\|_{C^\epsilon_{*}}\|\Y\|_{C^{\frac{3}{2}}_{*}}+ \|\Y\|_{C^\epsilon_{*}}\|\W\|_{C^{\frac{3}{2}}_{*}}\right) \|R\|_{H^{s-\frac{1}{2}}} \lesssim \CalAZ \CalAT \|R\|_{H^{s-\frac{1}{2}}}.
\end{equation*}
For the next two terms of $-T_{\W_\alpha}\nP b$, we estimate
\begin{align*}
 &\|T_{\W_\alpha} T_{1-\bar{\Y}}\nP(\overline{v^{rot} \circ Z}) \|_{H^s} \lesssim_\CalAZ \| \W\|_{C^1_{*}}\|\nP(\overline{v^{rot} \circ Z}) \|_{H^s} \lesssim_\CalAZ \CalAT \|Z-\alpha\|_{H^{s+1}}, \\
 &\|T_{\W_\alpha}\nP \Pi(R+\overline{v^{rot} \circ Z}, \bar{\Y}) \|_{H^s} \lesssim_\CalAZ \| \W\|_{C^{\epsilon}_{*}}(\|R\|_{H^{s-\frac{1}{2}}}\| \Y\|_{C^{\frac{3}{2}+\epsilon}_{*}}+ \|\Y\|_{H^s}\| \bar{\nP}(v^{rot}\circ Z)\|_{C^1_{*}})\\
 &+ \| \W\|_{C^{\epsilon}_{*}}\|\Y\|_{H^s}\| \nP(v^{rot}\circ Z)\|_{C^1_{*}} \lesssim \CalAZ \CalAT \| (\W, R)\|_{\H^{s-\frac{1}{2}}}.
\end{align*}

For other terms in $T_{D_t}\W$, they are perturbative, since
\begin{align*}
 &\|T_{b_\alpha}\W \|_{H^s} + \|\nP \Pi(b_\alpha, \W) \|_{H^s} + \|\nP\Pi(\W_\alpha, b) \|_{H^s} \lesssim \| b\|_{C^1_{*}}\|\W\|_{H^s}\lesssim_\CalAZ \CalAT \|\W\|_{H^s}, \\
 &\|\nP\partial_\alpha(v^{rot}\circ Z) \|_{H^s} \lesssim \|Z-\alpha \|_{H^{s+1}}.
\end{align*}
Putting all the perturbative terms into $G$, we obtain the leading term of $T_{D_t}\W$ \eqref{WParaMat}.
\end{proof}
We remark here the source term $G$ in $\eqref{WParaMat}$ also satisfies the Zygmund estimate:
\begin{equation}
    \|G\|_{C_*^s} \lesssim_\CalAZ \CalAT \left(\|(\W,R)\|_{C^s_{*}\times C_{*}^{s-\frac{1}{2}}}+ \|Z-\alpha \|_{C_{*}^{s+1}}\right), \quad \forall s>\frac{1}{2}. \label{GZygmund}
\end{equation}
The proof of this estimate is similar to the proof of the Sobolev bound; one just needs to replace the corresponding Sobolev estimate by the corresponding Zygmund estimate.

For $T_{D_t}R$, we write the second equation of \eqref{e:WW} as a paradifferential equation and apply the projection $\nP$, then we get the holomorphic paradifferential equation 
\begin{equation}  \label{ParaRExpression}
\begin{aligned}
 T_{D_t} R =& -\nP D_t \overline{v^{rot}\circ Z} -\nP T_{R_\alpha}b -\nP\Pi(R_\alpha, b)+i\nP[(g+a)\Y] \\
 -& \nP[(R+ \overline{v^{rot}\circ Z})\partial_\alpha(\bar{R}+  v^{rot} \circ Z)]+ \frac{ 2\sigma}{1+\W}\mathbf{P}\Im \partial_\alpha\left[ \frac{\W_{ \alpha}}{J^{\frac{1}{2}}(1+\W)}\right]. 
 \end{aligned}
\end{equation}

The first term on the right-hand side of \eqref{ParaRExpression} is the material derivative of the rotational part of the complex velocity on the surface.
We show that this term is perturbative.

\begin{lemma}
Let $s>0$, then the term $\nP D_t \overline{v^{rot}\circ Z}$ is perturbative in the sense that
\begin{equation} \label{nPMatvrotZ}
 \|\nP D_t \overline{v^{rot}\circ Z}\|_{H^s} \lesssim \CalAO\left(\| R\|_{H^s}+\|Z-\alpha \|_{H^s}+ \sum_{j=1}^N |\lambda_j|\right).
\end{equation}
\end{lemma}

\begin{proof}
Using the equation for $D_t Z$ \eqref{ZEqnTwo},
\begin{equation*}
D_t \overline{v^{rot}\circ Z} = \sum_{j=1}^N \frac{\lambda_j i}{2\pi}\frac{\bar{R}+v^{rot}\circ Z-\frac{d}{dt}z_j}{(Z-z_j)^2}.
\end{equation*}

We rewrite using the paraproduct decomposition,
\begin{align*}
  \nP D_t \overline{v^{rot}\circ Z} &= \sum_{j=1}^N \frac{\lambda_j i}{2\pi}\Big(T_{\bar{R}+v^{rot}\circ Z - \frac{d}{dt}z_j}\nP (Z-z_j)^{-2} \\
  &+T_{(Z-z_j)^{-2}}\nP (v^{rot}\circ Z) + \nP\Pi\big(\bar{R}+ v^{rot}\circ Z, (Z-z_j)^{-2}\big)\Big).
\end{align*}
Using the projection results \eqref{projection2}, \eqref{barPvRot}, \eqref{PvrotZHsCs}, \eqref{nPZsquare}, and the bounds for $z_j$ \eqref{zjtimedot}, we estimate
\begin{align*}
 &\|T_{\bar{R}+v^{rot}\circ Z - \frac{d}{dt}z_j}\nP (Z-z_j)^{-2} \|_{H^s} \\
 \lesssim& \left(\|R\|_{L^\infty} + \|(\nP+ \bar{\nP}) (v^{rot}\circ Z)\|_{C^\epsilon_{*}}+ \left|\frac{d} {dt}z_j\right|\right)\|\nP (Z- z_j)^{-2} \|_{H^s} \lesssim \CalAO \|Z-\alpha\|_{H^s},\\
 &\|T_{(Z-z_j)^{-2}}\nP (v^{rot}\circ Z) \|_{H^s} \lesssim \|(\nP+\bar{\nP} )(Z-z_j)^{-2} \|_{C^\epsilon_*}\|\nP (v^{rot}\circ Z)\|_{H^s}\lesssim  \CalAO \sum_{j=1}^N |\lambda_j|,\\
 & \|\nP\Pi\big(\bar{R}+ v^{rot}\circ Z, (Z-z_j)^{-2}\big) \|_{H^s} \lesssim (\|\bar{R} \|_{H^s}+ \|(\nP + \bar{\nP})(v^{rot}\circ Z) \|_{H^s})\|(Z-z_j)^{-2} \|_{C^\epsilon_{*}}\\
 \lesssim & \CalAO \left(\| R\|_{H^s}+\|Z-\alpha \|_{H^s}+ \sum_{j=1}^N |\lambda_j|\right).
\end{align*}
This concludes the proof of the estimate \eqref{nPMatvrotZ}.
\end{proof}

For the next four terms on the right-hand side of \eqref{ParaRExpression}, we have the following result.
\begin{lemma}
Let $s>0$, then one can write
\begin{equation}
  -\nP T_{R_\alpha}b -\nP\Pi(R_\alpha, b)+i\nP[(g+a)\Y] -\nP[(R+ \overline{v^{rot}\circ Z})\partial_\alpha(\bar{R}+  v^{rot} \circ Z)] 
=  K,\label{RNonCapillary}
\end{equation}
where the remainder term $K$ satisfies
\begin{equation*}
\| K\|_{H^{s}}\lesssim_\CalAO \CalAT \left(\sum_{j=1}^N |\lambda_j|+\|(\W, R)\|_{\H^s}+ \|Z-\alpha \|_{H^{s+1}}\right).
\end{equation*}
\end{lemma}
\begin{proof}
We first show that 
\begin{equation*}
  i\nP[(g+a)\Y] = ig\Y+ i\nP[a \Y] + K = K.  
  \end{equation*}
Using the product estimate \eqref{HsProduct} and estimates for $a$ \eqref{ACHalf}, \eqref{AHsEst},
\begin{equation*}
\|\nP[a \Y] \|_{H^{s}}\lesssim \| a\|_{C_{*}^{\epsilon}} \|\Y\|_{H^s} + \| a\|_{H^{s}}\|\Y\|_{L^\infty} \lesssim_\CalAO \CalAT \left(\sum_{j=1}^N |\lambda_j|+\|(\W, R)\|_{\H^s}+ \|Z-\alpha \|_{H^s}\right).
\end{equation*}
Then, we show that the rest terms on the left of \eqref{RNonCapillary} are perturbative and can be put into $K$.
Using \eqref{HCHEstimate}, \eqref{HsLinfty} and the estimate for $b$ \eqref{bEst},
\begin{align*}
&\|\nP T_{R_\alpha} b \|_{H^{s}}+ \|\nP \Pi(R_\alpha , b) \|_{H^{s}} \lesssim \|R_\alpha\|_{C^{\epsilon}_{*}}\|b\|_{H^{s}} \lesssim_\CalAZ \|R\|_{C^{1+\epsilon}_{*}} (\|R\|_{H^{s}}+ \|Z-\alpha\|_{H^s}), \\
& \|\nP[R\bar{R}_\alpha] \|_{H^{s}} \leq \|T_{\bar{R}_\alpha} R\|_{H^{s}} + \|\nP\Pi(\bar{R}_\alpha, R) \|_{H^{s}}\lesssim \| R\|_{C^{1+\epsilon}_{*}}\|R\|_{H^{s}},\\
& \|\nP[\overline{v^{rot}\circ Z}\bar{R}_\alpha] \|_{H^s} \leq \|T_{\bar{R}_\alpha}\nP(\overline{v^{rot}\circ Z}) \|_{H^s} + \| \Pi(\bar{R}_\alpha, \nP(\overline{v^{rot}\circ Z})) \|_{H^s}\\
&\lesssim \|R\|_{C^{1+\epsilon}_{*}}\| \nP(\overline{v^{rot}\circ Z}) \|_{H^s} + \|R\|_{H^s} \| \bar{\nP}(\overline{v^{rot}\circ Z}) \|_{C^1_{*}},\\
&\|\nP[R (v^{rot}\circ Z)_\alpha] \|_{H^s} \leq \| T_R \nP(v^{rot}\circ Z)_\alpha\|_{H^s} + \|T_{(v^{rot}\circ Z)_\alpha}R \|_{H^s} + \|\nP\Pi(R, (v^{rot}\circ Z)_\alpha)\|_{H^s}\\
& \lesssim \|R\|_{C^\epsilon_{*}}\|Z-\alpha\|_{H^{s+1}} + \|R\|_{H^s}\left(\|Z-\alpha\|_{C^{1+\epsilon}_{*}}+\sum_{j=1}^N |\lambda_j| \right), \\
&\|\nP[\overline{v^{rot}\circ Z}(v^{rot}\circ Z)_\alpha] \|_{H^s} \leq \|T_{\overline{v^{rot}\circ Z}}\nP(v^{rot}\circ Z)_\alpha \|_{H^s} + \|T_{(v^{rot}\circ Z)_\alpha}\nP\overline{v^{rot}\circ Z} \|_{H^s}\\
&+ \|\nP\Pi[\overline{v^{rot}\circ Z},(v^{rot}\circ Z)_\alpha]  \|_{H^s}\lesssim \|Z-\alpha\|_{H^{s+1}}\left(\|Z-\alpha \|_{C^{1+\epsilon}_{*}} +\sum_{j=1}^N |\lambda_j| \right).
\end{align*}
These terms are all controlled by 
\begin{equation*}
 \CalAT \left(\|(\W, R)\|_{\H^s}+ \|Z-\alpha \|_{H^{s+1}}+\sum_{j=1}^N |\lambda_j|\right),    
\end{equation*}
and can be absorbed into $K$. 
The proof of \eqref{RNonCapillary} is complete.
\end{proof}

For the last term on the right-hand side of \eqref{ParaRExpression}, we have the following lemma in \cite{MR4891579}:
\begin{lemma}[\hspace{1sp}\cite{MR4891579}]
 Let $s>0$, then
 \begin{equation}  \label{CurvaturePara}
   \frac{ 2\sigma}{1+\W}\mathbf{P}\Im \partial_\alpha\left[ \frac{\W_{ \alpha}}{J^{\frac{1}{2}}(1+\W)}\right] 
   = -i\sigma T_{J^{-\frac{1}{2}}(1-\Y)^2}\W_{\alpha \alpha}+ 3i\sigma T_{J^{-\frac{1}{2}}(1-\Y)^3 \W_\alpha} \W_\alpha + K,  
 \end{equation}
 where the remainder term $K$ satisfies
\begin{equation*}
\| K\|_{H^{s}}\lesssim_\CalAO \CalAT \|\W\|_{H^{s+\frac{1}{2}}}.
\end{equation*}
\end{lemma}

Collecting the results \eqref{nPMatvrotZ}, \eqref{RNonCapillary}, and \eqref{CurvaturePara}, we obtain the expression for the para-material derivative of $R$:
\begin{lemma}
The para-material derivative of $R$ is given by
\begin{equation}
       T_{D_t}R  = -i\sigma T_{J^{-\frac{1}{2}}(1-\Y)^2}\W_{\alpha \alpha}+ 3i\sigma T_{J^{-\frac{1}{2}}(1-\Y)^3 \W_\alpha} \W_\alpha + K,  \label{RParaMat}
\end{equation}
where for $s>0$, the source term $K$  satisfies bounds
\begin{equation*}
    \|K\|_{H^s} \lesssim_\CalAO \CalAT \left(\|(\W,R)\|_{\mathcal{H}^{s}}+ \|Z-\alpha \|_{H^{s+1}}+ \sum_{j=1}^N |\lambda_j|\right).
\end{equation*}
\end{lemma}

In the final part of this section, we compute the leading term of the material derivative of $J^{s}$ for $s\neq 0$.
\begin{lemma}
Let $s\neq 0$, the para-material derivative of $J^s$ satisfies the relation 
\begin{equation}
D_t J^s = E, \quad   \| E\|_{L^\infty} \lesssim_\CalAO \CalAT.\label{JsParaMat}
\end{equation}
The same expression also holds for the time derivative and the para-material derivative of $J^s$.
\end{lemma}

\begin{proof}
We compute the material derivative of $J^s$ using \eqref{WMatDerivative},
\begin{equation*}
D_t J^s = 2\Re sJ^{s}(1-\Y)D_t\W = 2\Re sJ^{s}(1-\Y)(\bar{R}_\alpha - (1+\W)b_\alpha + \partial_\alpha(v^{rot}\circ Z)).
\end{equation*}
Using the estimate for $b$ \eqref{bEst}, and for $v^{rot}\circ Z$ \eqref{barPvRot}, \eqref{PvrotZHsCs},
\begin{equation*}
\|D_t J^s \|_{L^\infty} \lesssim  \|J^s \|_{L^\infty}(1+\|\Y \|_{C^\epsilon_{*}})(\|R\|_{C^1_{*}} +(1+\|\W\|_{C^\epsilon_*})\|b\|_{C^1_{*}}+ \|v^{rot}\circ Z \|_{C^1_{*}}) \lesssim_\CalAZ \CalAT.
\end{equation*}
Here we use Moser-type estimate \eqref{MoserTwo} for $J^s$,
\begin{equation}\label{JsBound}
\|J^s \|_{L^\infty} \lesssim 1+ \|J^s-1 \|_{C^\epsilon_*} \lesssim  1+ \|\W \|_{C^\epsilon_*} \lesssim 1+\CalAZ.
\end{equation}
Note that using \eqref{CCCEstimate} and \eqref{CsCmStar},
\begin{equation*}
 \|T_b \W_\alpha\|_{L^\infty} + \|T_{\W_\alpha} b\|_{L^\infty} + \| \Pi(\W_\alpha, b)\|_{L^\infty} \lesssim \|b\|_{C^\epsilon_{*}}\|\W\|_{C^{1+\epsilon}_{*}} + \|\W\|_{C^\epsilon_{*}}\|b\|_{C^1_{*}} \lesssim \CalAO \CalAT,
\end{equation*}
so that terms $T_b \W_\alpha$ and $b\W_\alpha$ can be put into $E$, and $T_{D_t}J^s, \partial_t J^s$ have the similar expression as \eqref{JsParaMat}.
\end{proof}

\section{Local well-posedness of water waves} \label{s:Wellposed}
\subsection{A priori energy estimate for water waves}\label{s:Energy} 
Following the more recent paradifferential approach introduced by Alazard-Burq-Zuily \cite{MR2805065, MR2931520}, Ai-Ifrim--Tataru \cite{ai2023dimensional, MR4483135} and Ai \cite{ai2023improved}  in the water waves context, we will proceed by getting more sharper estimates;  here is where the  paradifferential calculus plays a crucial role. 
In this section, we prove the a priori energy estimate for water waves Theorem \ref{t:MainEnergyEstimate}.
Recall that in Section \ref{s:Estimate}, we compute the para-material derivatives of $(\W, R)$:
\begin{equation} \label{WRParaMatSys}
\left\{
\begin{aligned}
&  T_{D_t}\W  = -\partial_\alpha T_{(1+\W)(1-\bar{\Y})}R + G\\
&  T_{D_t}R  = -i\sigma T_{J^{-\frac{1}{2}}(1-\Y)^2}\W_{\alpha \alpha}+ 3i\sigma T_{J^{-\frac{1}{2}}(1-\Y)^3 \W_\alpha} \W_\alpha + K,
\end{aligned} 
\right.
\end{equation}
where the perturbative source terms $(G,K)$ satisfy,
\begin{equation*}
 \|(G, K) \|_{\H^s}\lesssim_\CalAO \CalAT \left(\|(\W,R)\|_{\mathcal{H}^{s}}+ \|Z-\alpha \|_{H^{s+1}}+ \sum_{j=1}^N |\lambda_j|\right), \quad s>0.
\end{equation*}

A natural attempt for the energy would be choosing
\begin{equation*}
    \int \sigma |\langle D \rangle^{s+\Half}\W|^2 + |\langle D \rangle^{s}R|^2 +|Z-\alpha|^2 \,d\alpha + \sum_{j=1}^N |\lambda_j|^2.
\end{equation*}
This choice of energy satisfies the norm equivalence \eqref{normEquivalence}.
However, its time derivative contains non-perturbative terms that cannot be bounded by the right-hand side of \eqref{EnergyEstimate}.
We need to add appropriate energy corrections in order to satisfy the energy estimate \eqref{EnergyEstimate}.

\begin{proof}[Proof of Theorem \ref{t:MainEnergyEstimate}]
We choose the energy
\begin{align*}
    E_s :=&  \int \sigma T_{J^{-\frac{3}{2}}}\langle D \rangle^{s+\Half}\W \cdot \langle D \rangle^{s+\Half}\bar{\W}  + |\langle D \rangle^{s}R|^2 +|Z-\alpha|^2 \,d\alpha + \sum_{j=1}^N |\lambda_j|^2\\
    &+ 9\sigma\Re \int i T_{\W_\alpha}T_{J^{-\frac{3}{2}}(1-\Y)}\langle D \rangle^{s-\frac{1}{2}}\W \cdot \langle D \rangle^{s+\frac{1}{2}}\bar{\W} \,d\alpha\\
    &+   \left(3s -\frac{15}{2}\right)\Im \int  T_{\W_\alpha}\langle D \rangle^{s-1}R \cdot T_{1-\Y}\langle D \rangle^{s}\bar{R} \, d\alpha.
\end{align*}
Here, the para-coefficient $T_{J^{-\frac{3}{2}}}$ is added to the first term of the energy in order to cancel the leading term of the time derivative of $E_s$:
\begin{equation}
J^{-\frac{3}{2}}(1-\bar{\Y})(1+\W) = J^{-\frac{1}{2}}(1-\bar{\Y})^2. \label{LeadingNonPerturb}
\end{equation}
The last two integrals of $E_s$ are energy corrections whose time derivatives eliminate sub-leading terms of the time derivative of the first integral of $E_s$.

Due to \eqref{JsBound} for $J^{-\frac{3}{2}}$ and \eqref{HsLinfty}, \eqref{HsHmCStar}, \eqref{YBounds}, we estimate
\begin{align*}
&\left| \int  T_{J^{-\frac{3}{2}}-1} \langle D \rangle^{s+\Half} \W \cdot \langle D \rangle^{s+\Half}\bar{\W} \,d\alpha\right| \lesssim  \|J^{-\frac{3}{2}}-1 \|_{L^\infty}\|\W\|_{H^{s+\frac{1}{2}}}^2 \lesssim \CalAZ \|\W\|_{H^{s+\frac{1}{2}}}^2, \\
& \left| \Re \int i T_{\W_\alpha}T_{J^{-\frac{3}{2}}(1-\Y)}\langle D \rangle^{s-\frac{1}{2}}\W \cdot \langle D \rangle^{s+\frac{1}{2}}\bar{\W} \,d\alpha \right| \\
\lesssim & \| \W\|_{C^\epsilon_{*}} \|J^{-\frac{3}{2}}(1-\Y) \|_{C^\epsilon_{*}}\|\W\|_{H^{s+\frac{1}{2}}}^2 \lesssim \CalAZ(1+\CalAZ) \|\W\|_{H^{s+\frac{1}{2}}}^2,\\
& \left|\Im \int  T_{\W_\alpha}\langle D \rangle^{s-1}R \cdot T_{1-\Y}\langle D \rangle^{s}\bar{R} \, d\alpha \right|\lesssim \|\W\|_{C^\epsilon_{*}}(1+\|\Y\|_{L^\infty})\|R\|_{H^s}^2 \lesssim \CalAZ(1+\CalAZ) \|R\|_{H^{s}}^2.
\end{align*}
Hence, the energy $E_s$ satisfies the norm equivalence \eqref{normEquivalence}.
It suffices to prove the energy estimate \eqref{EnergyEstimate}.

We compute the time derivative of $E_s$, and replace some time derivatives with para-material derivatives.
\begin{align*}
 \frac{d}{dt}E_s =&  2\Re\int \sigma T_{ J^{-\frac{3}{2}}}\langle D \rangle^{s+\Half} \W_t \cdot \langle D \rangle^{s+\Half}\bar{\W} \,d\alpha + 2\Re \int \langle D \rangle^{s}R \cdot \langle D \rangle^{s}\bar{R}_t \,d\alpha \\
  &+ 9\sigma \Re \partial_t\int i T_{\W_\alpha}T_{J^{-\frac{3}{2}}(1-\Y)}\langle D \rangle^{s-\frac{1}{2}}\W \cdot \langle D \rangle^{s+\frac{1}{2}}\bar{\W} \,d\alpha\\
  &+  \left(3s -\frac{15}{2}\right)\Im \partial_t\int T_{\W_\alpha}\langle D \rangle^{s-1}R \cdot T_{1-\Y}\langle D \rangle^{s}\bar{R} \,d\alpha\\
  &+2\Re \int (Z-\alpha)\bar{Z}_t \,d\alpha + \int \sigma T_{\partial_t J^{-\frac{3}{2}}}\langle D \rangle^{s+\Half}\W \cdot \langle D \rangle^{s+\Half}\bar{\W} \,d\alpha\\
 =& 2\Re \int \sigma T_{ J^{-\frac{3}{2}}}\langle D \rangle^{s+\Half}   T_{D_t}\W \cdot \langle D \rangle^{s+\Half}\bar{\W} +\langle D \rangle^{s}R \cdot \langle D \rangle^{s}T_{D_t}\bar{R} \,d\alpha \\
 &+ 9\sigma \Re \partial_t\int i T_{\W_\alpha}T_{J^{-\frac{3}{2}}(1-\Y)}\langle D \rangle^{s-\frac{1}{2}}\W \cdot \langle D \rangle^{s+\frac{1}{2}}\bar{\W} \,d\alpha\\
  &+\left(3s -\frac{15}{2}\right)\Im \partial_t\int T_{\W_\alpha}\langle D \rangle^{s-1}R \cdot T_{1-\Y}\langle D \rangle^{s}\bar{R} \,d\alpha \\
&- 2\Re \int  \sigma \langle D \rangle^{s+\Half}T_{ J^{-\frac{3}{2}}}T_b \W_\alpha \cdot \langle D \rangle^{s+\Half}\bar{\W} +\langle D \rangle^{s}R \cdot \langle D \rangle^{s}T_{b}\bar{R}_\alpha \,d\alpha \\
  &+2\Re \int (Z-\alpha)\bar{Z}_t \,d\alpha
  + \int \sigma T_{\partial_t J^{-\frac{3}{2}}}\langle D \rangle^{s+\Half}\W \cdot \langle D \rangle^{s+\Half}\bar{\W} \,d\alpha.
\end{align*}
We use the symbolic calculus \eqref{CompositionPara}  to distribute paraproducts and integrate by parts to estimate each part of $\frac{d}{dt}E_s$. 
In the following, for simplicity, we write perturbative integral terms that can be bounded by the right-hand side of \eqref{EnergyEstimate} as $ O_\CalAO \left(\CalAT E_s \right)$.

We begin by estimating the first term of $\frac{d}{dt}E_s$.
\begin{align*}
&2\Re \int  \sigma T_{ J^{-\frac{3}{2}}}\langle D \rangle^{s+\Half}T_{D_t}\W \cdot \langle D \rangle^{s+\Half}\bar{\W}  +\langle D \rangle^{s} R \cdot \langle D \rangle^{s}T_{D_t}\bar{R} \,d\alpha\\
= & -2\Re \int \sigma  T_{J^{-\frac{3}{2}}}\langle D \rangle^{s+\frac{1}{2}}  \partial_\alpha T_{(1-\bar{\Y})(1+\W)}R \cdot \langle D \rangle^{s+\frac{1}{2}} \bar{\W} - i\sigma\langle D \rangle^{s} R\cdot \langle D \rangle^{s} T_{J^{-\Half}(1-\bar{\Y})^2}\bar{\W}_{\alpha \alpha} \,d\alpha \\
&-2\Re \int 3i\sigma\langle D \rangle^{s} R\cdot \langle D \rangle^{s} T_{J^{-\Half}(1-\bar{\Y})^3 \bar{\W}_\alpha}\bar{\W}_{\alpha} \,d\alpha + 2\Re \int  \sigma T_{ J^{-\frac{3}{2}}}\langle D \rangle^{s+\Half}G \cdot \langle D \rangle^{s+\Half}\bar{\W} \,d\alpha \\
& +2\Re \int \langle D \rangle^{s} R \cdot \langle D \rangle^{s}\bar{K} \,d\alpha \\
= &  2\Re \int i\sigma T_{J^{-\frac{3}{2}}} \langle D \rangle^{s+\frac{3}{2}}  T_{(1-\bar{\Y})(1+\W)}R \cdot \langle D \rangle^{s+\frac{1}{2}} \bar{\W} + i\sigma\langle D \rangle^{s+\frac{3}{2}} R\cdot \langle D \rangle^{s-\frac{3}{2}} T_{J^{-\Half}(1-\bar{\Y})^2}\bar{\W}_{\alpha \alpha} \,d\alpha \\
&-6\Re \int i\sigma\langle D \rangle^{s} R\cdot T_{J^{-\Half}(1-\bar{\Y})^3 \bar{\W}_\alpha} \langle D \rangle^{s}\bar{\W}_{\alpha} \,d\alpha +  O_\CalAO \left(\CalAT E_s \right) \\
= & 2\Re \int i\sigma \left(T_{J^{-\frac{3}{2}}} T_{(1-\bar{\Y})(1+\W)} - T_{J^{-\frac{1}{2}}(1-\bar{\Y})^2} \right) \langle D \rangle^{s+\frac{3}{2}} R \cdot \langle D \rangle^{s+\frac{1}{2}} \bar{\W} \,d\alpha \\
 &-  \sigma\Re \int  \left((2s+3) T_{J^{-\frac{3}{2}}} T_{((1-\bar{\Y})(1+\W))_\alpha}-6T_{J^{-\frac{1}{2}}(1-\bar{\Y})^3 \bar{\W}_\alpha}\right)\langle D \rangle^{s+\frac{1}{2}} R \cdot \langle D \rangle^{s+\frac{1}{2}}\bar{\W} \,d\alpha \\
 &-\sigma \Re \int (2s-3)T_{(J^{-\frac{1}{2}}(1-\bar{\Y})^2)_\alpha} \langle D \rangle^{s+\frac{1}{2}} R \cdot \langle D \rangle^{s+\frac{1}{2}}\bar{\W} \,d\alpha  +  O_\CalAO \left(\CalAT E_s \right) \\
 = & 3 \sigma \Re \int \left(s+\frac{1}{2}\right)T_{J^{-\frac{3}{2}}(1-\bar{\Y})\W_\alpha}\langle D \rangle^{s+\frac{1}{2}} R \cdot \langle D \rangle^{s+\frac{1}{2}}\bar{\W} \,d\alpha \\
 &+  3\sigma \Re \int \left(s-\frac{11}{2}\right)T_{J^{-\frac{1}{2}}(1-\bar{\Y})^3 \bar{\W}_\alpha}\langle D \rangle^{s+\frac{1}{2}} R \cdot \langle D \rangle^{s+\frac{1}{2}}\bar{\W} \,d\alpha  +  O_\CalAO \left(\CalAT E_s \right).
\end{align*}
Here, the leading terms get canceled due to \eqref{CompositionPara} and the relation \eqref{LeadingNonPerturb}.
What remains are two non-perturbative sub-leading integral terms, which will be eliminated by the time derivative of the next two integral terms of $\frac{d}{dt}E_s$.
When the time derivative acts on para-coefficients, the integral is always perturbative.
For instance, using the estimates \eqref{HsCmStar}, \eqref{YBounds}, \eqref{bEst}, and \eqref{GZygmund}, we have
\begin{align*}
 &\left|\Im \int T_{\partial_t\W_\alpha}\langle D \rangle^{s-1}R \cdot T_{1-\Y}\langle D \rangle^{s}\bar{R} \,d\alpha \right|
 \leq  \|  T_{\partial_\alpha^2T_{(1+\W)(1-\bar{\Y})}R +(T_b \W_\alpha)_\alpha-G_\alpha}\langle D \rangle^{s-1}R\|_{L^2} \|\langle D \rangle^{s}\bar{R}  \|_{L^2} \\
 & \lesssim (\|T_{(1+\W)(1-\bar{\Y})} R\|_{C^1_{*}}+\|T_b\W_\alpha \|_{C^1_{*}}+\|G\|_{C^\epsilon_{*}}) \|R\|_{H^s}^2 \lesssim_\CalAO \CalAT \|R\|_{H^s}^2.
\end{align*}
In addition, for this part of the integral, only the leading terms of $(\W_t, R_t)$ will contribute to the non-perturbative parts of the integral.
The integrals with transport terms $(T_b \W_\alpha, T_b R_\alpha)$ or the sub-leading terms of of $(\W_t, R_t)$ belong to $ O_\CalAO \left(\CalAT E_s \right)$.
Hence, we simply replace $(\W_t, R_t)$ with $(-T_{(1+\W)(1-\bar{\Y})}R_\alpha, -i\sigma T_{J^{-\frac{1}{2}}(1-\Y)^2}\W_{\alpha\alpha})$ when we compute the time derivative of next two terms,
\begin{align*}
 &9\sigma \Re \partial_t\int i T_{\W_\alpha}T_{J^{-\frac{3}{2}}(1-\Y)}\langle D \rangle^{s-\frac{1}{2}}\W \cdot \langle D \rangle^{s+\frac{1}{2}}\bar{\W} \,d\alpha\\
 +& \left(3s -\frac{15}{2}\right)\Im \partial_t\int T_{\W_\alpha}\langle D \rangle^{s-1}R \cdot T_{1-\Y}\langle D \rangle^{s}\bar{R} \,d\alpha \\
 =& 9\sigma \Re \int i T_{\W_\alpha}T_{J^{-\frac{3}{2}}(1-\Y)}\langle D \rangle^{s-\frac{1}{2}}\W_t \cdot \langle D \rangle^{s+\frac{1}{2}}\bar{\W} \,d\alpha\\
 &+9\sigma \Re \int i T_{\W_\alpha}T_{J^{-\frac{3}{2}}(1-\Y)}\langle D \rangle^{s-\frac{1}{2}}\W \cdot \langle D \rangle^{s+\frac{1}{2}}\bar{\W}_t \,d\alpha\\
 &+ \left(3s -\frac{15}{2}\right)\Im \int T_{\W_\alpha}\langle D \rangle^{s-1}R_t \cdot T_{1-\Y}\langle D \rangle^{s}\bar{R}\,d\alpha\\
 &+ \left(3s -\frac{15}{2}\right)\Im \int T_{\W_\alpha}\langle D \rangle^{s-1}R \cdot T_{1-\Y}\langle D \rangle^{s}\bar{R}_t \,d\alpha  +  O_\CalAO \left(\CalAT E_s \right) \\
 =& 9\sigma \Re \int - T_{J^{-\frac{3}{2}}(1-\bar{\Y})\W_\alpha}\langle D \rangle^{s+\frac{1}{2}}R \cdot \langle D \rangle^{s+\frac{1}{2}}\bar{\W} \,d\alpha\\
 &+ 9\sigma \Re \int T_{J^{-\frac{1}{2}}(1-\Y)^3 \W_\alpha}\langle D \rangle^{s+\frac{1}{2}}\W \cdot \langle D \rangle^{s+\frac{1}{2}}\bar{R} \,d\alpha\\
 &+ \left(3s -\frac{15}{2}\right)\sigma\Im \int iT_{J^{-\frac{1}{2}}(1-\Y)^3\W_\alpha}\langle D \rangle^{s+\frac{1}{2}}\W \cdot \langle D \rangle^{s+\frac{1}{2}}\bar{R}\\
 &- \left(3s -\frac{15}{2}\right)\sigma \Im \int iT_{J^{-\frac{3}{2}}(1-\bar{\Y})\W_\alpha}\langle D \rangle^{s+\frac{1}{2}}R \cdot \langle D \rangle^{s+\frac{1}{2}}\bar{\W} \,d\alpha +  O_\CalAO \left(\CalAT E_s \right) \\
 =& -9\sigma \Re \int T_{J^{-\frac{3}{2}}(1-\bar{\Y})\W_\alpha}\langle D \rangle^{s+\frac{1}{2}}R \cdot \langle D \rangle^{s+\frac{1}{2}}\bar{\W}\\
 &+ 9\sigma \Re \int T_{J^{-\frac{1}{2}}(1-\bar{\Y})^3 \bar{\W}_\alpha}\langle D \rangle^{s+\frac{1}{2}}R \cdot \langle D \rangle^{s+\frac{1}{2}}\bar{\W} \,d\alpha\\
  &- \left(3s -\frac{15}{2}\right)\sigma\Re \int T_{J^{-\frac{3}{2}}(1-\bar{\Y})\W_\alpha}\langle D \rangle^{s+\frac{1}{2}}R \cdot \langle D \rangle^{s+\frac{1}{2}}\bar{\W}\\
  &-\left(3s -\frac{15}{2}\right)\sigma \Re \int T_{J^{-\frac{1}{2}}(1-\bar{\Y})^3\bar{\W}_\alpha}\langle D \rangle^{s+\frac{1}{2}}R \cdot \langle D \rangle^{s+\frac{1}{2}}\bar{\W} \,d\alpha+  O_\CalAO \left(\CalAT E_s \right)\\
 = &-3 \sigma \Re \int \left(s+\frac{1}{2}\right)T_{J^{-\frac{3}{2}}(1-\bar{\Y})\W_\alpha}\langle D \rangle^{s+\frac{1}{2}} R \cdot \langle D \rangle^{s+\frac{1}{2}}\bar{\W} \,d\alpha \\
 &-3\sigma \Re \int \left(s-\frac{11}{2}\right)T_{J^{-\frac{1}{2}}(1-\bar{\Y})^3 \bar{\W}_\alpha}\langle D \rangle^{s+\frac{1}{2}} R \cdot \langle D \rangle^{s+\frac{1}{2}}\bar{\W} \,d\alpha  +  O_\CalAO \left(\CalAT E_s \right).
\end{align*}
This part of integral eliminates the first part of the non-perturbative integrals, so that the remaining terms are  $O_\CalAO \left(\CalAT E_s \right)$.

For the next integral term of $\frac{d}{dt}E_s$, using integration by parts to distribute $\alpha$-derivative,
\begin{align*}
 &- 2\Re \int  \sigma T_{ J^{-\frac{3}{2}}}\langle D \rangle^{s+\Half}T_b\W_\alpha \cdot \langle D \rangle^{s+\Half}\bar{\W} +\langle D \rangle^{s}R \cdot \langle D \rangle^{s}T_{b}\bar{R}_\alpha \,d\alpha \\
 =& - 2\Re \int  \sigma T_{b} T_{ J^{-\frac{3}{2}}}\langle D \rangle^{s+\Half}\W_\alpha \cdot \langle D \rangle^{s+\Half}\bar{\W} +T_{b}\langle D \rangle^{s}R \cdot \langle D \rangle^{s}\bar{R}_\alpha \,d\alpha  +  O_\CalAO \left(\CalAT E_s \right)\\
 =&  2\Re \int  \sigma T_{b} T_{ J^{-\frac{3}{2}}}\langle D \rangle^{s+\Half}\W_\alpha \cdot \langle D \rangle^{s+\Half}\bar{\W} +T_{b}\langle D \rangle^{s}R_\alpha \cdot \langle D \rangle^{s}\bar{R} \,d\alpha+  O_\CalAO \left(\CalAT E_s \right)\\
 &+ 2\Re \int  \sigma \left(T_{b_\alpha} T_{ J^{-\frac{3}{2}}} + T_{b} T_{ (J^{-\frac{3}{2}})_\alpha}\right)\langle D \rangle^{s+\Half}\W \cdot \langle D \rangle^{s+\Half}\bar{\W} +T_{b_\alpha}\langle D \rangle^{s}R \cdot \langle D \rangle^{s}\bar{R} \,d\alpha\\
 =& \Re \int  \sigma \left(T_{b_\alpha} T_{ J^{-\frac{3}{2}}} + T_{b} T_{ (J^{-\frac{3}{2}})_\alpha}\right)\langle D \rangle^{s+\Half}\W \cdot \langle D \rangle^{s+\Half}\bar{\W} +T_{b_\alpha}\langle D \rangle^{s}R \cdot \langle D \rangle^{s}\bar{R} \,d\alpha \\
 &+ O_\CalAO \left(\CalAT E_s \right) =  O_\CalAO \left(\CalAT E_s \right).
\end{align*}
It can be bounded by the right-hand side of \eqref{EnergyEstimate}.
As for the last two integral terms of $\frac{d}{dt}E_s$, we apply the formulae for $Z_t$ \eqref{ZEqnTwo} and $\partial_t J^{s}$ \eqref{JsParaMat},
\begin{align*}
 &\left|2\Re \int (Z-\alpha)\bar{Z}_t \,d\alpha \right| +  \left|\int \sigma T_{\partial_t J^{-\frac{3}{2}}}\langle D \rangle^{s+\Half}\W \cdot \langle D \rangle^{s+\Half}\bar{\W} \,d\alpha \right|\\
 \lesssim& \|Z-\alpha\|_{L^2} (\|b\|_{L^\infty}(\| \W\|_{L^2}+\| b\|_{L^2})+\|R \|_{L^2}+ \| v^{rot}\circ Z\|_{L^2}) + \|\partial_t J^{-\frac{3}{2}} \|_{L^\infty} \| \W\|_{H^{s+\frac{1}{2}}}^2\\
 \lesssim& (1+\|b\|_{L^\infty}+ \|\partial_t J^{-\frac{3}{2}} \|_{L^\infty})\left(\|Z-\alpha \|^2 + \|\W \|_{H^{s+\frac{1}{2}}}^2 + \|R\|_{H^s}^2 + \| v^{rot}\circ Z\|_{L^2}^2 \right) \lesssim_\CalAO\CalAT E_s.
\end{align*}
These two terms are perturbative.
Therefore, the energy $\frac{d}{dt}E_s$ satisfies the energy estimate \eqref{EnergyEstimate}.
This concludes the proof of Theorem \ref{t:MainEnergyEstimate}.
\end{proof}

In the next subsection, we will prove the unique existence of solutions by using the iteration.
At each step of iteration, we will consider the following system of linear differential equations: 
\begin{equation} \label{LinearIterationEqn}
\left\{
\begin{aligned}
&  (\partial_t+T_{\tilde{b}}\partial_\alpha)\W  +\partial_\alpha T_{1+\tilde{A}}R = \tilde{G}\\
&  (\partial_t+T_{\tilde{b}}\partial_\alpha)R  +i\sigma T_{1+\tilde{B}}\W_{\alpha \alpha}- 3i\sigma T_{\tilde{C}} \W_\alpha = \tilde{K},
\end{aligned} 
\right.
\end{equation}
where $\tilde{b}, \tilde{A}, \tilde{B}, \tilde{C}$ and $\tilde{G},\tilde{K}$ are fixed functions.
We show that \eqref{LinearIterationEqn} has a unique solution, provided that these functions satisfy the appropriate regularity conditions.
\begin{proposition}\label{t:LinearExist}
Let $s>0$, and time $T>0$.
Suppose the initial condition $(\W_0, R_0)\in \H^s$, the fixed functions satisfy $(\tilde{G},\tilde{K})\in C([0,T];\H^s)$, $\tilde{b}\in C([0,T];C^{1+\epsilon}_{*})$, $\tilde{A}, \tilde{B}\in C([0,T];C^{\frac{3}{2}+\epsilon}_{*})$, $\tilde{C}\in C([0,T];C^{\frac{1}{2}+\epsilon}_{*})$, $\tilde{A}_t, \tilde{B}_t\in C([0,T]; C^\epsilon_{*})$, and $\tilde{b}, \frac{1+\tilde{B}}{1+\tilde{A}}$ are real-valued functions.
Then the linear equation \eqref{LinearIterationEqn} has a unique solution $(\W, R)\in C([0,T];\H^s)$.
Moreover, there exists a modified energy $\tilde{E}_s$ such that 
\begin{equation} \label{LinEstimate}
 \tilde{E}_s \approx \sigma\|\W \|^2_{H^{s+\frac{1}{2}}}+\|R \|_{H^s}^2, \quad \frac{d}{dt}\tilde{E}_s \lesssim  \tilde{E}_s + \| (G, K)\|^2_{\H^s},  
\end{equation}
where the implicit constant depends on Zygmund norms of $\tilde{b}$, $\tilde{A}, \tilde{B}$, $\tilde{C}$, $\tilde{A}_t$ and $\tilde{B}_t$.
\end{proposition}

\begin{proof}[Sketch of the proof]
The existence of the solution follows from the standard theory of linear PDEs, see for example Section $5$ in \cite{MR1471885}.
The difference of two solutions of \eqref{LinearIterationEqn} satisfies \eqref{LinearIterationEqn} with zero initial data and zero source term.
Hence, the solution is unique by applying Gronwall's inequality and the modified energy estimate \eqref{LinEstimate}.

For the modified energy estimate, we choose 
\begin{align*}
\tilde{E}_s =& \int \sigma T_{\frac{1+\tilde{B}}{1+\tilde{A}}}\langle D \rangle^{s+\Half}\W \cdot \langle D \rangle^{s+\Half}\bar{\W}  + |\langle D \rangle^{s}R|^2 \,d\alpha + \Re \int i T_{\tilde{D}}\langle D \rangle^{s-\frac{1}{2}}\W \cdot \langle D \rangle^{s+\frac{1}{2}}\bar{\W} \,d\alpha\\
    &+   \Im \int  T_{\tilde{E}}\langle D \rangle^{s-1}R \cdot \langle D \rangle^{s}\bar{R} \,d\alpha,
\end{align*}
where functions $\tilde{D}, \tilde{E} \in C([0,T];C^{-1}_{*})$ are chosen such that
\begin{equation*}
   -2i(1+\tilde{A})\Im \tilde{D} -2\sigma(1+\tilde{B})\Re \tilde{E}  = (2s+3)\frac{1+\tilde{B}}{1+\tilde{A}}\tilde{A}_\alpha -6\bar{\tilde{C}}+(2s-3)\bar{\tilde{B}}_\alpha.
\end{equation*}
These choices of $\tilde{D}$ and $\tilde{E}$ ensure that the sub-leading terms of $\frac{d}{dt}\tilde{E}_s$ are removed, and the leading term of $\frac{d}{dt}\tilde{E}_s$ is eliminated with the para-coefficient $\frac{1+\tilde{B}}{1+\tilde{A}}$.
Then the modified energy estimate \eqref{LinEstimate} is virtually the same as the proof of Theorem \ref{EnergyEstimate}, and we omit it.
\end{proof}

\subsection{The proof of local well-posedness} \label{s:ProofWell} 
Having proved the energy estimate Theorem \ref{t:MainEnergyEstimate}, we now use this result to construct the solution and prove Theorem \ref{t:Wellposed}.
We first use the iteration method to construct a sequence of approximate solutions. 
By computing the uniform and difference bounds, we will show that the sequence of approximate solutions converges to the solution of \eqref{e:WW} on $[0,T]$ for some time $T$.

For the zero-th approximate solution, we take $(\W^{(0)}, R^{(0)})\times\{ z_j^{(0)}\}_{j=1}^N : = (\W_0, R_0)\times \{z_j(0)\}_{j=1}^N$, 
\begin{equation*}
    Z^{(0)}(t, \alpha) = \alpha + \int_{-\infty}^\alpha \W(0, y) \,dy,
\end{equation*}
and choose auxiliary functions $b^{(0)}: = b(\W^{(0)}, R^{(0)}, \{z_j(0)\}_{j=1}^N)$, $\Y^{(0)}: = \Y(\W^{(0)})$, $J^{(0)}: = J(\W^{(0)})$.
Once we construct the $n$-th approximate solution $(\W^{(n)}, R^{(n)})\times\{ z_j^{(n)}\}_{j=1}^N \in C( [0,T]; \H^s)\times C^1([0,T];\Omega_t)^N$ for $n\in \mathbb{N}$ and define those auxiliary functions, in view of the paradifferential system \eqref{WRParaMatSys}, we define the $(n+1)$-th approximate solution by the following scheme:
\begin{enumerate}
\item $(\W^{(n+1)}, R^{(n+1)})$ are the solution of the linear paradifferential system
\begin{equation} \label{WRniteration}
\left\{
\begin{aligned}
&  T_{D_t}^{(n)}\W^{(n+1)}  = -\partial_\alpha T_{(1+\W^{(n)})(1-\bar{\Y}^{(n)})}R^{(n+1)} + G^{(n)}\\
&  T_{D_t}^{(n)}R^{(n+1)}  = -i\sigma T_{(J^{(n)})^{-\frac{1}{2}}(1-\Y^{(n)})^2}\W^{(n+1)}_{\alpha \alpha}+ 3i\sigma T_{(J^{(n)})^{-\frac{1}{2}}(1-\Y^{(n)})^3 \W^{(n)}_\alpha} \W^{(n+1)}_\alpha + K^{(n)},
\end{aligned} 
\right.
\end{equation}
with the initial condition $(\W^{(n+1)}(0, \alpha),R^{(n+1)}(0, \alpha) ):= (\W_0, R_0)$.
Here, the $n$-th para-material derivative is given by $T_{D_t}^{(n)} : = \partial_t + T_{b^{(n)}}\partial_\alpha$, and source terms 
\begin{equation*}
  \left(G^{(n)}, K^{(n)}\right) : = \left(G(\W^{(n)}, R^{(n)}, Z^{(n)}, \{z_j^{(n)}\}_{j=1}^N), K(\W^{(n)}, R^{(n)}, Z^{(n)},\{z_j^{(n)}\}_{j=1}^N)\right).  
\end{equation*}
\item $\{z_j^{(n+1)}(t) \}_{j=1}^N$ are the solutions of the system of ODEs
\begin{equation} \label{ODEIterationN}
    \frac{d}{dt}z_j^{(n+1)}(t) = \frac{i}{2\pi}\int_{\mathbb{R}} \frac{(1+\W^{(n)})\cdot \bar{R}^{(n)} (t,\alpha)}{Z^{(n)}(t, \alpha) -z^{(n)}_j(t)}\,d\alpha + \sum_{k\neq j} \frac{\lambda_k i}{2\pi}\frac{1}{\overline{z_j^{(n)}(t) - z_k^{(n)}(t)}},  
\end{equation}
with initial conditions $z_j^{(n+1)} = z_j(0)$, for $1\leq j \leq N$.
\item Auxiliary functions at $(n+1)$-th iteration are  $b^{(n+1)}: = b(\W^{(n+1)}, R^{(n+1)}, \{z_j^{(n+1)}\}_{j=1}^N)$, $\Y^{(n+1)}: = \Y(\W^{(n+1)})$, $J^{(n+1)}: = J(\W^{(n+1)})$.
\end{enumerate}
Note that in this iteration, $\W^{(n)}, R^{(n)}$ and $Z^{(n)}-\alpha$ are all holomorphic functions on the real line.

Indeed, as in \eqref{WRParaMatSys}, the source terms $\left(G^{(n)}, K^{(n)}\right)$ satisfy for $s>0$,
\begin{equation} \label{SourceTermN}
 \|(G^{(n)}, K^{(n)}) \|_{\H^s}\lesssim_{\mathcal{A}^{(n)}_{1}} \mathcal{A}^{(n)}_{\frac{3}{2}} \left(\|(\W^{(n)},R^{(n)})\|_{\mathcal{H}^{s}}+ \|Z^{(n)}-\alpha \|_{H^{s+1}}+ \sum_{j=1}^N |\lambda_j|\right), 
\end{equation}
and $\mathcal{A}^{(n)}_{\frac{3}{2}}<+\infty$ by Sobolev embedding when $s>\frac{3}{2}$.
The auxiliary function $b^{(n)}$ satisfies the estimate 
\begin{equation*}
\|b^{(n)}\|_{C^{1+\epsilon}_{*}}\lesssim_\CalAZ \|R^{(n)}\|_{C^{1+\epsilon}_{*}} + \|Z^{(n)}-\alpha \|_{C^{1+\epsilon}_{*}}
\end{equation*}
due to \eqref{bEst}.
The para-coefficients satisfy
\begin{equation*}
 \|(1+\W^{(n)})(1-\bar{\Y}^{(n)})-1 \|_{C^{\frac{3}{2}+\epsilon}_{*}} + \|(J^{(n)})^{-\frac{1}{2}}(1-\Y^{(n)})^2-1 \|_{C^{\frac{3}{2}+\epsilon}_{*}} \lesssim \|\W^{(n)}\|_{C^{\frac{3}{2}+\epsilon}_{*}}
\end{equation*}
by the Moser-type estimate \eqref{MoserTwo}, and  
\begin{equation*}
\|(J^{(n)})^{-\frac{1}{2}}(1-\Y^{(n)})^3 \W^{(n)}_\alpha\|_{C^{\frac{1}{2}+\epsilon}_{*}} \lesssim \|\W^{(n)}\|_{C^{\frac{1}{2}+\epsilon}_{*}}
\end{equation*}
by \eqref{MoserTwo} and \eqref{CsProduct}.
In addition, 
\begin{equation*}
 \|\partial_t((1+\W^{(n)})(1-\bar{\Y}^{(n)})) \|_{C^{\epsilon}_{*}} + \|\partial_t((J^{(n)})^{-\frac{1}{2}}(1-\Y^{(n)})^2) \|_{C^{\epsilon}_{*}} \lesssim \mathcal{A}^{(n)}_{\frac{3}{2}}<+\infty.
\end{equation*}
Hence, according to Proposition \ref{t:LinearExist}, the system \eqref{WRniteration} has a unique solution $(\W^{(n+1)}, R^{(n+1)})$.
As in \eqref{zjtimedot}, 
\begin{equation*} 
 \left|\frac{d}{dt}z_j^{(n+1)}(t)\right| \leq \|R^{(n)}(t, \cdot)\|_{L^\infty} + (N-1)\max_{j\neq k} \frac{|\lambda_k|}{2\pi} \frac{1}{|z_j^{(n)}(t) - z_k^{(n)}(t)|}.
\end{equation*}
$z_j^{(n+1)}(t)$ can be defined up to time $t=\tau$ as long as $\| \mathcal{A}^{(n)}_{1}\|_{L^1_t[0,\tau]}<+\infty$. 

We now consider the class of approximate solutions that satisfy the following three properties:
\begin{align} 
    &\min_{j\neq k} |z_j(t) - z_k(t)| \geq \frac{1}{2} \min_{j\neq k} |z_j(0) - z_k(0)|, \label{distance}\\
    &\min_j \inf_{\alpha\in \mathbb{R}}|Z(t,\alpha)-z_j(t)|\geq\frac{1}{2}\min_j \inf_{\alpha\in \mathbb{R}}|Z_0(\alpha)-z_j(0)|,  \label{distanceTwo}\\
    &\|(\W,R, Z-\alpha)(t) \|_{\H^s\times L^2} + \sum_{j=1}^N |\lambda_j|\leq C\left(\|(\W_0,R_0, Z_0-\alpha) \|_{\H^s \times L^2}+ \sum_{j=1}^N |\lambda_j|\right), \label{Energygrowth}
\end{align}
$\forall t \in [0,T]$ for some fixed constant $C>1$.
Clearly, the zero-th approximate solution satisfies these three properties.
Assume that the $n$-th approximate solution satisfies the above three properties, we will check the $(n+1)$-th approximate solution is the same.

To check the distance property \eqref{distance}, we estimate for each pair of $j\neq k$ using the triangle inequality, 
\begin{align*}
 &|z^{(n+1)}_j(t)- z^{(n+1)}_k(t)|\geq |z^{(n+1)}_j(0)- z^{(n+1)}_k(0)|- |z^{(n+1)}_j(t)- z^{(n+1)}_j(0)| -|z^{(n+1)}_k(t)- z^{(n+1)}_k(0)|\\
 &\geq |z^{(n+1)}_j(0)- z^{(n+1)}_k(0)| - T\left(\left|\frac{d}{dt}z_j^{(n+1)}(t)\right|+ \left|\frac{d}{dt}z_k^{(n+1)}(t)\right|\right)\\
 &\geq |z^{(n+1)}_j(0)- z^{(n+1)}_k(0)| - T\left(2C\|R_0 \|_{H^s}+ \frac{N}{\pi}\max_j |\lambda_j|\max_{j\neq k}\frac{1}{|z_j(0)-z_k(0)|}\right).
\end{align*}
Hence, the distance property \eqref{distance} is true as long as the time $T$ is chosen small enough.
For the distance property \eqref{distanceTwo}, we fix an integer $j$,
\begin{align*}
&|Z^{(n+1)}(t,\alpha) - z_j^{(n+1)}(t)|\geq |Z_0(\alpha) - z_j(0)|-|Z^{(n+1)}(t,\alpha) - Z_0(\alpha)|-|z^{(n+1)}_j(t) - z_j^{(n+1)}(0)|\\
 \geq& \min_j \inf_{\alpha\in \mathbb{R}}|Z_0(\alpha)-z_j(0)|-\int_0^t|\partial_tZ^{(n+1)}(\tau, \alpha)|+ \Big|\frac{d}{dt}z_j^{(n+1)}(\tau)\Big|\,d\tau \\
\geq& \min_j \inf_{\alpha\in \mathbb{R}}|Z_0(\alpha)-z_j(0)|- T\left(C\|(\W_0, R_0)\|_{\H^s}+ \frac{N}{\pi}\max_j |\lambda_j|\max_{j\neq k}\frac{1}{|z_j(0)-z_k(0)|}\right).
\end{align*}
Taking the minimum in $j$ on the left-hand side of the inequality and choosing $T$ sufficiently small, we obtain the distance property \eqref{distanceTwo}.
As a by-product of \eqref{distanceTwo}, we get the $L^2$ estimate
\begin{equation*}
 \| v^{rot,(n)}\circ Z^{(n)}\|_{L^2} \lesssim \sum_{j=1}^N|\lambda_j|.
\end{equation*}
As for the uniform bound of the energy, comparing the paradifferential systems \eqref{LinearIterationEqn} and \eqref{WRniteration}, one notes that $(\W^{(n+1)}, R^{(n+1)})$ solve the system \eqref{LinearIterationEqn} with
\begin{align*}
&\tilde{A} = (1+\W^{(n)})(1-\bar{\Y}^{(n)}) -1, \quad\tilde{B} = (J^{(n)})^{-\frac{1}{2}}(1-\Y^{(n)})^2-1, \\
&\tilde{C} = (J^{(n)})^{-\frac{1}{2}}(1-\Y^{(n)})^3 \W^{(n)}_\alpha, \quad\tilde{b} = b^{(n)}, \quad \tilde{G} = G^{(n)}, \quad \tilde{K} = K^{(n)}.
\end{align*}
Hence, by \eqref{LinEstimate} and \eqref{SourceTermN}, there exists an energy $E^{(n+1)}_s$ such that
\begin{equation} \label{ModifiedEnergyN}
 E^{(n+1)}_s \approx \sigma\|\W \|^2_{H^{s+\frac{1}{2}}}+\|R \|_{H^s}^2+ \sum_{j=1}^N |\lambda_j|^2, \quad \frac{d}{dt}E^{(n+1)}_s \lesssim  E_s^{(n+1)} + E_s^{(0)}+ \|Z_0 -\alpha \|^2_{L^2}.  
\end{equation}
Here, in the above inequality, the implicit constant depends on $\mathcal{A}^{(n)}_{\frac{3}{2}}$, which is uniform bounded from above by Sobolev embedding and the properties \eqref{distance} and \eqref{Energygrowth}.
We further compute the time derivative of $\|Z^{(n+1)}-\alpha\|_{L^2}^2$.
As in \eqref{ZEqnTwo}, $Z^{(n+1)}$ solves the equation
\begin{equation*}
  (\partial_t + b^{(n)}\partial_\alpha)Z^{(n+1)} = \bar{R}^{(n)} + \sum_{j=1}^N v^{rot,(n)}_j \circ Z^{(n)}.   
\end{equation*}
Hence, we compute
\begin{align*}
 &\frac{d}{dt}\|Z^{(n+1)}-\alpha\|_{L^2}^2 \leq 2\left|\Re \int (Z-\alpha)\bar{Z}_t \,d\alpha \right| \\
 \lesssim & \|Z^{(n+1)}-\alpha\|_{L^2} \left(\|b^{(n)}\|_{L^\infty}(\| \W^{(n+1)}\|_{L^2}+\| b^{(n)}\|_{L^2})+\|R^{(n)} \|_{L^2}+ \| v^{rot,(n)}\circ Z^{(n)}\|_{L^2}\right) \\
 \lesssim & \|Z^{(n+1)}-\alpha\|_{L^2}^2 + \|\W^{(n+1)}\|_{H^{s+\frac{1}{2}}}^2 + \|R^{(n+1)} \|_{H^s}^2 + \sum_{j=1}^N |\lambda_j|^2.
\end{align*}
Again, here the implicit constant in the inequality is uniformly bounded.
Adding this inequality to the inequality in \eqref{ModifiedEnergyN} and using Gronwall's inequality, we get
\begin{equation*}
    \|(\W^{(n+1)},R^{(n+1)}, Z^{(n+1)}-\alpha)(t) \|_{\H^s\times L^2} + \sum_{j=1}^N |\lambda_j|\leq e^{Ct}\left(\|(\W_0,R_0, Z_0-\alpha) \|_{\H^s \times L^2}+ \sum_{j=1}^N |\lambda_j|\right),
\end{equation*}
for some constant $C$.
Hence, the $(n+1)$-th approximate solution also satisfies \eqref{Energygrowth}.

Next, we show that the approximate solutions $\{\W^{(n)}, R^{(n)}, \{z_j^{(n)}\}_{j=1}^N  \}_{n=0}^{+\infty}$ form a Cauchy sequence.
Define the difference between two approximate solutions
\begin{equation*}
    \hat{\W}^{(n)} : = \W^{(n+1)}-\W^{(n)}, \quad \hat{R}^{(n)} : = R^{(n+1)}-R^{(n)}, \quad \hat{Z}^{(n)} : = Z^{(n+1)}-Z^{(n)}, \quad \hat{z}^{(n)}_j : = z_j^{(n+1)}-z_j^{(n)}.
\end{equation*}
Then $\left(\hat{\W}^{(n)}, \hat{R}^{(n)}, \{\hat{z}_j^{(n)}\}_{j=1}^N  \right)$ solve the system
\begin{equation} \label{WRndifference}
\left\{
\begin{aligned}
&  T_{D_t}^{(n)}\hat{\W}^{(n)}  = -\partial_\alpha T_{(1+\W^{(n)})(1-\bar{\Y}^{(n)})}\hat{R}^{(n)} + \hat{G}^{(n)}\\
&  T_{D_t}^{(n)}\hat{R}^{(n)}  = -i\sigma T_{(J^{(n)})^{-\frac{1}{2}}(1-\Y^{(n)})^2}\hat{\W}^{(n)}_{\alpha \alpha}+ 3i\sigma T_{(J^{(n)})^{-\frac{1}{2}}(1-\Y^{(n)})^3 \W^{(n)}_\alpha} \hat{\W}^{(n)}_\alpha + \hat{K}^{(n)}\\
&\frac{d}{dt}\hat{z}_j^{(n)} = \mathcal{U}^{(n)}(t, z_j^{(n)}) - \mathcal{U}^{(n-1)}(t, z_j^{(n-1)}) + \sum_{k\neq j} \frac{\lambda_k i}{2\pi}\frac{\overline{\hat{z}^{(n-1)}_k - \hat{z}^{(n-1)}_j}}{\overline{\left(z_j^{(n)} - z_k^{(n)}\right)\left(z_j^{(n-1)} - z_k^{(n-1)}\right)}},
\end{aligned} 
\right.
\end{equation}
where the source terms $(\hat{G}^{(n)}, \hat{K}^{(n)})$ are given by
\begin{align*}
\hat{G}^{(n)} = & G^{(n)} -G^{(n-1)} -T_{b^{(n)}-b^{(n-1)}}\W^{(n)}_\alpha- \partial_\alpha T_{(1+\W^{(n)})(1-\bar{\Y}^{(n)})-(1+\W^{(n-1)})(1-\bar{\Y}^{(n-1)})}R^{(n)} , \\
\hat{K}^{(n)} = & K^{(n)} -K^{(n-1)}-T_{b^{(n)}-b^{(n-1)}}R^{(n)}_\alpha -i\sigma T_{(J^{(n)})^{-\frac{1}{2}}(1-\Y^{(n)})^2-(J^{(n-1)})^{-\frac{1}{2}}(1-\Y^{(n-1)})^2}\W^{(n)}_{\alpha \alpha} \\
& + 3i\sigma T_{(J^{(n)})^{-\frac{1}{2}}(1-\Y^{(n)})^3 \W^{(n)}_\alpha-(J^{(n-1)})^{-\frac{1}{2}}(1-\Y^{(n-1)})^3 \W^{(n-1)}_\alpha} \W^{(n)}_\alpha.
\end{align*}
Direct computation using \eqref{Energygrowth} gives for $s>\frac{3}{2}$, 
\begin{align*}
    \|(\hat{G}^{(n)},\hat{K}^{(n)} )\|_{\H^{s-\frac{3}{2}}} \lesssim_{\mathcal{A}^{(n)}_{\frac{3}{2}}, \mathcal{A}^{(n+1)}_{\frac{3}{2}}}& \left(\|(\W_0,R_0, Z_0-\alpha)
    \|_{\H^s \times L^2}+ \sum_{j=1}^N |\lambda_j|\right)\cdot\\
    &\|(\hat{\W}^{(n-1)}, \hat{R}^{(n-1)}, \hat{Z}^{(n-1)} )\|_{\H^{s-\frac{3}{2}}\times L^2}.
\end{align*}

In addition, $\hat{Z}^{(n)}$ satisfies the equation 
\begin{equation*}
  (\partial_t + b^{(n)}\partial_\alpha)\hat{Z}^{(n)} = \bar{\hat{R}}^{(n-1)} + \sum_{j=1}^N \frac{\lambda_j i}{2\pi}\frac{\overline{\hat{z}_j^{(n-1)} - \hat{Z}^{(n-1)}}}{\overline{\left(Z^{(n)} - z_j^{(n)}\right)\left(Z^{(n-1)} - z_j^{(n-1)}\right)}}-(b^{(n)}-b^{(n-1)})Z_\alpha^{(n)}.   
\end{equation*}
We compute the time derivative of $\|\hat{Z}^{(n)}\|_{L^2}^2$,
\begin{align*}
 &\frac{d}{dt}\|\hat{Z}^{(n)}\|_{L^2}^2 \leq 2\left|\Re \int \hat{Z}^{(n)}\bar{\hat{Z}}^{(n)}_t \,d\alpha \right| \lesssim  \|\hat{Z}^{(n)}\|_{L^2} \Big(\|b^{(n)}\|_{L^\infty}\| \hat{\W}^{(n)}\|_{L^2}\\
 +&\|\hat{R}^{(n-1)} \|_{L^2}+ \sum_{j=1}^N |\hat{z}_j^{(n-1) }|+N\|\hat{Z}^{(n-1)} \|_{L^2} + (\|\W^{(n)} \|_{L^2}+1)\|b^{(n)}-b^{(n-1)} \|_{L^\infty}\Big) \\
 \lesssim&   \|(\hat{\W}^{(n)}, \hat{R}^{(n-1)}, \hat{Z}^{(n-1)})\|_{\H^{s-\frac{3}{2}}\times L^2}^2  + \sum_{j=1}^N|\hat{z}_j^{(n-1) }|^2,
\end{align*}
where the implicit constant in the inequality depends on $\|(\W_0,R_0, Z_0-\alpha) \|_{\H^s \times L^2}+ \sum_{j=1}^N |\lambda_j|$.

We estimate the time derivative of $\hat{z}^{(n)}_j$.
\begin{align*}
&\left|\frac{d}{dt}\hat{z}^{(n)}_j(t)\right|  \leq \frac{1}{2\pi} \int \Bigg|\frac{(1+\W^{(n)})\cdot \bar{\hat{R}}^{(n-1)} (t,\alpha)}{Z^{(n)}(t, \alpha) -z^{(n)}_j(t)} \Bigg|\,d\alpha + \frac{1}{2\pi} \int \Bigg|\frac{\hat{\W}^{(n-1)}\cdot \bar{R}^{(n-1)} (t,\alpha)}{Z^{(n)}(t, \alpha) -z^{(n)}_j(t)} \Bigg|\,d\alpha \\
&+ \frac{1}{2\pi} \int \Bigg|\frac{(1+\W^{(n-1)})\cdot \bar{R}^{(n-1)} (t,\alpha)}{Z^{(n)}(t, \alpha) -z^{(n)}_j(t)}-\frac{(1+\W^{(n-1)})\cdot \bar{R}^{(n-1)} (t,\alpha)}{Z^{(n-1)}(t, \alpha) -z^{(n-1)}_j(t)} \Bigg|\,d\alpha\\
&+ \sum_{k\neq j}\frac{|\lambda_k|}{2\pi} \left| \frac{\hat{z}^{(n-1)}_k(t) - \hat{z}^{(n-1)}_j(t)}{\left(z_j^{(n)}(t) - z_k^{(n)}(t)\right)\left(z_j^{(n-1)}(t) - z_k^{(n-1)}(t)\right)}\right|\\
&\lesssim \left(\|(\W_0, R_0, Z_0-\alpha) \|_{\H^s \times L^2}+ \sum_{j=1}^N |\lambda_j|\right)\left(\|( \hat{\W}^{(n-1)}, \hat{R}^{(n-1)}, \hat{Z}^{(n-1)})\|_{\H^{s-\frac{3}{2}}\times L^2} + \sum_{j=1}^N |\hat{z}^{(n-1)}_j |\right), 
\end{align*}
where we use \eqref{distance} and \eqref{distanceTwo} to bound denominators in each term.

Applying Proposition \ref{t:LinearExist} for $(\hat{\W}^{(n)}, \hat{R}^{(n)})$ in the first two equations in \eqref{WRndifference} and adding the time derivative of $\|\hat{Z}^{(n)}\|_{L^2}^2$ and $|\hat{z}^{(n)}_j|^2$, we obtain that for $s>\frac{3}{2}$ and $n\geq 1$, there exists a sequence of energies $\mathcal{E}_s^{(n)}$, such that
\begin{equation*}
 \mathcal{E}_s^{(n)} \approx  \sigma\|\hat{\W}^{(n)} \|^2_{H^{s-1}}+\|\hat{R}^{(n)} \|_{H^{s-\frac{3}{2}}}^2 +\|\hat{Z}^{(n)} \|_{L^2}^2 + \sum_{j=1}^N|\hat{z}_j^{(n)}|^2, \quad \frac{d}{dt}\mathcal{E}_s^{(n)} \lesssim  \mathcal{E}_s^{(n)} + \mathcal{E}_s^{(n-1)}.   
\end{equation*}
Hence, we get the energy bound
\begin{equation*}
    \mathcal{E}^{(n)}_s(t) \leq (e^{Ct}-1) \sup_{0\leq t\leq T} \mathcal{E}^{(n-1)}_s(t), \quad \forall t\in [0,T],
\end{equation*}
where $C$ is a positive constant that depends on $\|(\W_0,R_0, Z_0-\alpha) \|_{\H^s \times L^2}+ \sum_{j=1}^N |\lambda_j|$.
By choosing the time $T$ small enough, 
\begin{equation*}
    \sup_{0\leq t\leq T} \mathcal{E}^{(n)}_s(t) \leq \frac{1}{2}\sup_{0\leq t\leq T} \mathcal{E}^{(n-1)}_s(t),
\end{equation*}
so that $\{\W^{(n)},R^{(n)}, Z^{(n)}-\alpha, \{ z^{(n)}_j\}_{j=1}^N\}_{n=1}^\infty$ forms a Cauchy sequence in $C([0,T]; \H^{s-\frac{3}{2}}\times H^s\times \Omega_t^N)$.
In addition, $\{\W^{(n)},R^{(n)}, Z^{(n)}-\alpha\}$ is uniformly bounded in $C([0,T]; \H^{s}\times H^{s+\frac{3}{2}})$, and $|\frac{d}{dt}z_j^{(n)}(t)|$ is uniformly bounded on $[0,T]$.
We conclude that
\begin{equation*}
 \{\W^{(n)},R^{(n)}, Z^{(n)}-\alpha, \{ z^{(n)}_j\}_{j=1}^N\} \rightarrow \{\tilde{\W}, \tilde{R}, \tilde{Z}-\alpha, \{\tilde{z}_j\}_{j=1}^N\}\in C([0,T]; \H^{s}\times H^{s+\frac{3}{2}}\times \Omega_t^N). 
\end{equation*}
We have that
\begin{align*}
    &\frac{1}{\overline{Z^{(n)}(t,\alpha)-z^{(n)}_j(t)}} \rightarrow \frac{1}{\overline{\tilde{Z}(t,\alpha)-\tilde{z}_j(t)}} \quad \text{in } C([0,T];H^{s+\frac{1}{2}}), \\
    &\mathcal{U}^{(n)}(t,z^{(n)}_j) \rightarrow \frac{i}{2\pi}\int_{\mathbb{R}} \frac{(1+\tilde{\W})\cdot \bar{\tilde{R}} (t,\alpha)}{\tilde{Z}(t, \alpha) -\tilde{z}_j(t)}\,d\alpha \quad \text{uniformly in }\Omega_t.
\end{align*}
Hence, by sending $n\rightarrow \infty$ in \eqref{WRniteration} and \eqref{ODEIterationN}, $\{\tilde{\W}, \tilde{R}, \{\tilde{z}_j\}_{j=1}^N\}$ solve the water wave system \eqref{e:WW} in $C([0,T]; \H^{s}\times \Omega_t^N)$.
In addition, the solution satisfies \eqref{distance}, \eqref{distanceTwo}, \eqref{Energygrowth}.
In the ODEs \eqref{ODEIterationN}, one further gets $\frac{d}{dt}z^{(n)}_j(t)\rightarrow \frac{d}{dt}z_j(t)$ in $C([0,T]; \Omega)$, so that $\{\tilde{z}_j\}_{j=1}^N \in C^1([0,T]; \Omega_t^N)$.

The uniqueness of \eqref{e:WW} is obtained by considering the energy estimate of the difference equation of two solutions with the same initial data as in \eqref{WRndifference}.
We get that if $(\W^i, R^i, \{z_j^i\}_{j=1}^N)$, $i=1,2$, are two solutions of \eqref{e:WW} with initial data $(\W^i_0, R^i_0, \{z_j^i(0)\}_{j=1}^N)$, $i=1,2$, then for $s>\frac{3}{2}$, 
\begin{equation} \label{Lipschitz}
\begin{aligned}
 &\|(\W^1, R^1, \{z_j^1\}_{j=1}^N) - (\W^2, R^2, \{z_j^2\}_{j=1}^N) \|_{L^\infty([0,T]; \H^{s-\frac{3}{2}})\times W^{1,\infty}([0,T]; \mathbb{C}^N)} \\
 \lesssim_{\mathcal{A}_{\frac{3}{2},1}, \mathcal{A}_{\frac{3}{2},2}} &\|(\W^1_0, R^1_0, \{z_j^1(0)\}_{j=1}^N) - (\W^2_0, R^2_0, \{z_j^2(0)\}_{j=1}^N) \|_{\H^{s-\frac{3}{2}}\times \mathbb{C}^N}.
\end{aligned}
\end{equation}
Setting the same initial data for two solutions, we get the uniqueness of the solution from the Lipschitz bound.

From the a priori energy bound of the solution in $L^\infty([0,T]; \H^s\times \Omega_t^N)$ and the Lipschitz bound \eqref{Lipschitz} in $L^\infty([0,T]; \H^{s-\frac{3}{2}}\times \mathbb{C}^N)$,  for any $\sigma <s$, the solution map
\begin{equation*}
(\W_0, R_0, \{z_j(0) \}_{j=1}^N)\in \H^s\times \Omega_t^N \rightarrow (\W, R, \{z_j \}_{j=1}^N)\in C^0([0,T];\H^\sigma)\times C^1([0,T];\Omega_t^N)
\end{equation*}
is uniformly continuous.
Suppose there exists a sequence of initial data $\{(\W_{0,n}, R_{0,n}, \{z_{j,n}(0) \}_{j=1}^N) \}_{n=1}^\infty$ that converges to $(\W_{0}, R_{0}, \{z_{j}(0) \}_{j=1}^N)$ in $\H^s \times \mathbb{C}^N$.
Then using the argument in Theorem $6.12$ in Alazard-Burq-Zuily \cite{MR2805065}, the corresponding solutions satisfy
\begin{equation*}
\lim_{n \rightarrow \infty} \| (\W_{n}, R_{n}, \{z_{j,n} \}_{j=1}^N) - (\W, R, \{z_{j}\}_{j=1}^N)\|_{C^0([0,T];\H^s)\times C^1([0,T];\mathbb{C}^N)} = 0.
\end{equation*}
This shows the continuous dependence of initial data for the solution map.

Finally, the blow-up criterion follows directly from the energy estimate Theorem \ref{t:MainEnergyEstimate}.

\section{Cubic lifespan of solutions for 2 vortices with small symmetrized data} \label{s:CubicLifespan}

Recall that for holomorphic functions $f$ on the real line, $\Re f = H\Im f$.
Since the Hilbert transform $H$ maps even functions to odd functions, and vice versa, we can define the space of complex even and odd functions in $H^s(\R)$ by
\begin{align*}
 &H^s_e : = \{f\in H^s(\R): \Re\{f\} \text{ is odd, and } \Im\{f\} \text{ is even}\},\\
 &H^s_o : = \{f\in H^s(\R): \Re\{f\} \text{ is even, and } \Im\{f\} \text{ is odd}\}.
\end{align*}

If the initial data is symmetric with respect to the vertical axis, then $R_0\in H^{s}_e$, $Z_0 - \alpha \in H^{s+\frac{3}{2}}_e$ and $\W_0 \in H_o^{s+\frac{1}{2}}$.
The following result asserts that the solution keeps this symmetry.
\begin{proposition}
Let $s>\frac{3}{2}$.
If the initial data is symmetric about the vertical axis in the sense that
\begin{enumerate}
\item The initial surface and velocity are symmetric: $(\W_0, R_0)\in H^{s+\frac{1}{2}}_o \times H^s_e$.
\item The number of vortices is even, and the initial positions and strengths of vortices are symmetric. 
The initial vorticity $\omega_0$ can be represented by
\begin{equation*}
    \omega_0 = \sum_{j=1}^{N_0}(\lambda_j\delta_{z_{j,1}(0)} -\lambda_j\delta_{z_{j,2}(0)}),
\end{equation*}
where $\lambda_j\in \R$ and $z_{j,1}(0), z_{j,2}(0)\in \Omega_0$ for $1\leq j \leq N_0$.
Moreover, $\Re \{z_{j,1}(0) \} = -\Re \{z_{j,2}(0) \}$, $\Im \{z_{j,1}(0) \} = \Im \{z_{j,2}(0) \}$.
\end{enumerate}
Then the water wave system \eqref{e:WW} admits a unique solution $\left(\W, R, \{z_{j,1}\}_{j=1}^{N_0},  \{z_{j,2}\}_{j=1}^{N_0}\right)\in C([0,T];H^{s+\frac{1}{2}}_o \times H^s_e) \times C^1([0,T];\Omega_t^{2N_0})$, and $\Re \{z_{j,1}(t) \} = -\Re \{z_{j,2}(t) \}$, $\Im \{z_{j,1}(t) \} = \Im \{z_{j,2}(t) \}$ for $1\leq j \leq N_0$.
\end{proposition}
The proof of this proposition is virtually the same as Theorem $5$ in \cite{MR4179726}, and we ask the interested readers to check there.

In the rest of this section, we will choose $T = \delta \epsilon^{-2} $, for some positive constant $\delta$.
We will take into account the symmetry of the initial data and prove Theorem \ref{t:lifespan}.
For the simplicity of the computation, we will write the coefficient of the surface tension $\sigma = 1$ in the following.
We consider the control norms defined by
\begin{align*}
\CalAZS : = \|\W \|_{ C^{\epsilon}_{*}}+ \|R \|_{ C^{-\frac{1}{2}+\epsilon}_{*}}, \quad \CalAOS : = \|\W \|_{ C^{1+\epsilon}_{*}}+ \|R \|_{ C^{\frac{1}{2}+\epsilon}_{*}},\quad \CalATS : = \|\W \|_{ C^{\frac{3}{2}+\epsilon}_{*}}+ \|R \|_{ C^{1+\epsilon}_{*}}.
\end{align*}
These control norms do not contain the vortices interaction terms, as we will show that two vortices won't be too close on $[0,T]$. 
The key part of the proof is the following  energy estimate:
\begin{theorem} \label{t:CubicEnergy}
Under the conditions in Theorem \ref{t:lifespan}, there exists an energy functional $\mathcal{E}_s$ associated with the solution that satisfies the following properties:
\begin{enumerate}
     \item Norm equivalence:
       \begin{equation}
           \mathcal{E}_s \approx_\CalAZS \|(\W,  R)\|^2_{\mathcal{H}^s}. \label{normEqnTwo}
       \end{equation}
     \item Energy estimate:
      \begin{equation}
     \frac{d}{dt}\mathcal{E}_s\lesssim_\CalAZS \left(\CalAOS\CalATS + (1+ t)^{-\frac{3}{2}}\right) \mathcal{E}_s+ \epsilon^2(1+ t)^{-\frac{3}{2}}.\label{EnergyEst}
 \end{equation} 
 \end{enumerate} 
\end{theorem}
In order to prove Theorem \ref{t:lifespan}, we make the following bootstrap assumption for the solution: 
For some large constant $\tilde{M} \gg 2$,
\begin{equation} \label{bootstrap}
    \|(\W, R, Z-\alpha)(t) \|_{\H^s \times L^2}\leq \tilde{M}\epsilon, \quad \forall t\in[0,T].
\end{equation}

The rest of this section is organized as follows.
In Section \ref{s:Preliminary}, we derive estimates that are related to the point vortices.
Next, in Section \ref{s:Paralin}, we construct normal form variables $(\tilde{\W}, \tilde{R})$, so that using $(\tilde{\W}, \tilde{R})$, most of the non-perturbative quadratic terms in the system are removed.
Then, in Section \ref{s:CubicEnergy}, we construct a modified energy and derive the energy estimate Theorem \ref{t:CubicEnergy}. 
At the end, in Section \ref{s:prooflifespan}, we close the bootstrap assumption \eqref{bootstrap}, and finalize the proof of Theorem \ref{t:lifespan}.

\subsection{Preliminary estimates} \label{s:Preliminary}
In this section, we consider some estimates that are related to terms of vortices.
We show that many of these bounds are related to some negative powers of $d_S(t)$ \eqref{DefDistance}, and also obtain a lower bound for $d_S(t)$.

We first estimate $\mathcal{U}$ and $\mathcal{U}_z$ in $L^\infty$.
By the maximum principle for holomorphic functions and Sobolev embedding, 
\begin{equation} \label{MathcalUEst}
\|\mathcal{U}(t, \cdot) \|_{L^\infty(\Omega_t)} = \|\mathcal{U}(t, \cdot) \|_{L^\infty(\Gamma_t)} = \| R(t)\|_{L^\infty} \lesssim \tilde{M}\epsilon.   
\end{equation}
$\bar{\mathcal{U}}_z(t, \cdot)$ is a holomorphic function in $\Omega_t$ with boundary value $\frac{\bar{R}_\alpha}{\bar{Z}_\alpha}$.
Again by the maximum principle,
\begin{equation} \label{UzEst}
 \|\mathcal{U}_z(t, z) \|_{L^\infty} = \left\|\frac{R_\alpha}{Z_\alpha} \right\|_{L^\infty} \leq \|R_\alpha\|_{L^\infty}(1+\|\Y\|_{L^\infty}) \lesssim 2\tilde{M}\epsilon.
\end{equation}

Next, we control the growth of real and imaginary part of $z(t)$, $x(t): = \Re z(t)$ and $y(t): = \Im z(t)$. 
\begin{lemma} 
Assume the conditions in Theorem \ref{t:lifespan} and \eqref{bootstrap}, then for $0\leq t\leq T$:
\begin{equation}
     C_1\frac{\lambda}{x(0)}\leq \frac{d}{dt}y(t)\leq C_2\frac{\lambda}{x(0)}, \quad \frac{1}{C_3}\leq \frac{x(t)}{x(0)} \leq C_3,  \label{xtgrowth}
\end{equation}
for some constants $C_2> C_1>0$, $C_3>1$.
In other words, point vortices mostly move in the vertical direction.
In addition, the distance of the vortices to the surface $d_S(t)$ \eqref{DefDistance} satisfies
\begin{equation} \label{DistanceSurface}
d_S(t) \gtrsim  1+ t, \quad \text{when } \lambda<0; \quad d_S(t)\gg \epsilon^{-2} , \quad \text{when } \lambda>0. 
\end{equation}
\end{lemma}

\begin{proof}
By taking the imaginary part of $z_2(t)$, $y(t)$ solves the equation 
\begin{equation*}
    \frac{d}{dt}y(t) = \Im \mathcal{U}(t, z_2(t)) + \frac{\lambda}{4\pi x(t)}.
\end{equation*}
Using the bootstrap assumption \eqref{bootstrap} and assumption $(4)$ in Theorem \ref{t:lifespan},
\begin{equation*}
 |\Im \mathcal{U}(t, z_2(t))|<  |\mathcal{U}(t, z_2(t))| \leq \| \mathcal{U}(t, \cdot)\|_{L^\infty(\Omega_t)} = \|R(t,\cdot) \|_{L^\infty}\lesssim \epsilon \ll \frac{|\lambda|}{ x(0)}.
\end{equation*}
Hence, the estimate for $\frac{d}{dt}y(t)$ holds as long as one can prove the estimate for $x(t)$ in \eqref{xtgrowth}. 

Since $\Re \mathcal{U}$ is an odd function, $\Re \{ \mathcal{U}(t, iy)\} = 0$.
By the mean value theorem,
\begin{equation*}
    \frac{d}{dt}x(t) = \Re \mathcal{U}(t, x(t)+iy(t)) - \Re \mathcal{U}(t, iy(t)) = \Re \mathcal{U}_z(t, \tilde{x}(t)+iy(t))x(t),
\end{equation*}
for some $\tilde{x}(t)\in (0, x(t))$.
Without loss of generality, one can assume that $\Re  \mathcal{U}_z(t, \tilde{x}(t)+iy(t))\geq 0$, and $x(t)$ is increasing on $0\leq t\leq T$. 
Using the bound \eqref{UzEst}, we get 
\begin{equation*}
    \frac{x(t)}{x(0)}\leq e^{2\tilde{M}\epsilon t}, \quad \frac{1}{x(t)}\geq \frac{e^{-2\tilde{M}\epsilon t}}{x(0)}, \quad t\in [0,T],
\end{equation*}
which is not enough to prove the bound \eqref{xtgrowth}.
To obtain a better estimate for $\frac{x(t)}{x(0)}$, one needs to get a refined bound for $|\Re \mathcal{U}_z(t,z)|$.

We first prove a weaker bound for $d_S(t)$:
\begin{equation} \label{WeakdSBound}
    d_S(t) \gtrsim 1+\epsilon t, \quad \text{when } \lambda< 0; \quad d_S(t)\gg \epsilon^{-2} , \quad \text{when } \lambda>0. 
\end{equation}
Since $\frac{|\lambda|}{x(0)}\geq M \gg 1$ in the assumption of Theorem \ref{t:lifespan},
\begin{align*}
    &y(t)-y(0) \leq \tilde{M}\epsilon t + \lambda \int_0^t \frac{e^{-2\tilde{M}\epsilon \tau}}{4\pi x(0)} \,d\tau \leq \tilde{M} \epsilon t - \frac{M}{4\pi}\int_0^t e^{-2\tilde{M}\epsilon \tau} \,d\tau \\
    =& \tilde{M} \epsilon t -\frac{M}{8\epsilon\tilde{M}\pi}(1-e^{-2\tilde{M}\epsilon t}) \leq -C\epsilon t,
\end{align*}
for some constant $C>0$.
If $\lambda>0$, then $\frac{1}{x(t)}\leq \frac{1}{x(0)}$ since we assume that $x(t)$ is increasing, so that
\begin{equation*}
    y(t) - y(0) \leq \tilde{M}\epsilon t + \frac{|\lambda|}{4\pi x(0)}t.
\end{equation*}
Therefore, $|y(t)| \gg \epsilon^{-2} $ on $[0,T]$.
Because $d_S(t) \geq |y(t)|-\|Z(t)-\alpha\|_{L^\infty}$, we obtain the estimate \eqref{WeakdSBound}.

Returning back to the proof of the bound for $\frac{x(t)}{x(0)}$, we estimate $|\Re \mathcal{U}_z(t, \tilde{x}(t)+iy(t))|$.
Taking the $z$-derivative in \eqref{IrrotationalV}, 
\begin{equation*}
    \partial_z \mathcal{U}(t,z) = -\frac{i}{2\pi}\int_\R \frac{(1+\W)\cdot \bar{R}(t,\alpha)}{(Z(t,\alpha)-z)^2}\,d\alpha,
\end{equation*}
we get using \eqref{Lkestimate}, for $t\in [0,T]$, and $\lambda<0$,
\begin{align*}
 |\Re \mathcal{U}_z(t, \tilde{x}(t)+iy(t))| &\leq \frac{1}{2\pi}(1+\|\W\|_{L^\infty})\left(\int_\R \frac{1}{|\tilde{x}(t)+iy(t)-Z(t,\alpha)|^4}\,d\alpha\right)^{\frac{1}{2}}\|R\|_{L^2} \\
 &\leq 2\epsilon \left(\inf_{\alpha \in \R}\min_{1\leq j \leq N} |Z(t,\alpha)- \tilde{x}(t)-iy(t)|\right)^{-\frac{3}{2}}\lesssim \epsilon (1+\epsilon t)^{-\frac{3}{2}}.
\end{align*}
Hence, we have
\begin{equation*}
\frac{d}{dt}\ln \frac{x(t)}{x(0)} \lesssim \epsilon (1+\epsilon t)^{-\frac{3}{2}},
\end{equation*}
so that for some constant $C>0$,
\begin{equation*}
x(t) \leq x(0) \exp \left\{C \epsilon \int_0^{+\infty}(1+\epsilon t)^{-\frac{3}{2}} \,dt \right\} = x(0)\exp \left\{2C\right\}.
\end{equation*}
When $\lambda>0$,  
\begin{equation*}
 \ln \frac{x(t)}{x(0)} \lesssim \int_0^T \epsilon(\epsilon^{-2})^{-\frac{3}{2}}\,dt \lesssim \delta\epsilon^2 <2C.
\end{equation*}
Choosing the constant $C_3 = \exp \{ 2C\}$ gives the proof for $\frac{x(t)}{x(0)}$, which leads to the estimate \eqref{xtgrowth}.
Integrating $\frac{d}{dt}y(t)$, we get 
\begin{equation*}
    |y(t) - y(0)| \geq \min\{C_1, C_2 \}\frac{|\lambda|}{x(0)}t \gtrsim t.
\end{equation*}
Again using $d_S(t) \geq |y(t)|-\|Z(t)-\alpha\|_{L^\infty}$, we get the estimate for $d_S(t)$ \eqref{DistanceSurface}.
\end{proof}

Under the assumption of Theorem \ref{t:lifespan}, we can write $v^{rot}\circ Z$ as
\begin{equation*}
 v^{rot}\circ Z = \sum_{j=1}^2 v_j^{rot}\circ Z = -\frac{i\lambda}{\pi}\frac{x(t)}{\overline{(Z-z_1)(Z-z_2)}}.   
\end{equation*}
Then we have the following estimates due to Lemma 6.4 in \cite{MR4179726}. 
\begin{lemma}[\hspace{1sp}\cite{MR4179726}]
With the assumption in Theorem \ref{t:lifespan}, and the bootstrap assumption \eqref{bootstrap}, then for $s>\frac{3}{2}$, $\lambda_1 = \lambda$, $\lambda_2 = -\lambda$,
\begin{align}
 \|v^{rot}\circ Z \|_{H^s} \lesssim \epsilon d_S(t)^{-\frac{3}{2}}, \label{vrotZHs32} \\
 \left\|\sum_{j=1}^2 \frac{\lambda_j i}{2\pi}\frac{1}{(Z(t,\alpha)-z_j(t))^2} \right\|_{H^s}\lesssim \epsilon d_S(t)^{-\frac{5}{2}}. \label{valpharotZHs52}
\end{align}
\end{lemma}

With the above estimates in hand, we estimate one of the terms in $\partial_t (\overline{v^{rot}\circ Z})$.
\begin{lemma}
With the condition in Theorem \ref{t:lifespan}, and the bootstrap assumption \eqref{bootstrap}, for $s>\frac{3}{2}$, $\lambda_1 = \lambda$, $\lambda_2 = -\lambda$,
\begin{equation}
 \left\|\sum_{j=1}^2 \frac{\lambda_j i}{2\pi}\frac{\frac{d}{dt}z_j(t)}{(Z(t,\alpha)-z_j(t))^2} \right\|_{H^s}\lesssim \epsilon d_S(t)^{-\frac{5}{2}}+ \epsilon^2 d_S(t)^{-\frac{3}{2}}. \label{dtvrotZHs}
\end{equation}
\end{lemma}

\begin{proof}
Using the equation for $\frac{d}{dt}z_j(t)$:
\begin{equation*}
 \frac{d}{dt}z_j(t) = \frac{\lambda i}{4\pi x(t)} + \mathcal{U}(t,z_2(t)),
\end{equation*}
we can write
\begin{align*}
&\sum_{j=1}^2 \frac{\lambda_j i}{2\pi}\frac{\frac{d}{dt}z_j(t)}{(Z(t,\alpha)-z_j(t))^2} = -\sum_{j=1}^2 \frac{\lambda_j \lambda}{8\pi^2 x(t)} \frac{1}{(Z(t,\alpha)-z_j(t))^2}\\
+& \mathcal{U}(t, z_1(t))\sum_{j=1}^2 \frac{\lambda_j i}{2\pi}\frac{1}{(Z(t,\alpha)-z_j(t))^2}+ \frac{\lambda i(\mathcal{U}(t, z_1)-\mathcal{U}(t,z_2))}{2\pi}\frac{1}{(Z(t,\alpha)-z_2(t))^2}.
\end{align*}
Note that $\mathcal{U}(t, z_j(t))$ only depends on $t$, we use \eqref{LkHsEstimate}, \eqref{xtgrowth}, \eqref{valpharotZHs52}, and the assumption $|\lambda x(0)|\leq \epsilon$ to estimate,
\begin{align*}
 &\left\|\sum_{j=1}^2 \frac{\lambda_j \lambda}{8\pi^2 x(t)} \frac{1}{(Z-z_j)^2} \right\|_{H^s} +  \left\|\mathcal{U}(t, z_1(t))\sum_{j=1}^2 \frac{\lambda_j i}{2\pi}\frac{1}{(Z-z_j)^2}\right\|_{H^s} \\
 \leq& \frac{C_3|\lambda|}{4\pi x(0)}  \left\|\sum_{j=1}^2 \frac{\lambda_j i}{4\pi} \frac{1}{(Z-z_j)^2} \right\|_{H^s} + \|\mathcal{U}(t,\cdot) \|_{L^\infty}  \left\|\sum_{j=1}^2 \frac{\lambda_j i}{4\pi} \frac{1}{(Z-z_j)^2} \right\|_{H^s}\lesssim \epsilon d_S(t)^{-\frac{5}{2}},\\
 & \left\| \frac{\lambda i(\mathcal{U}(t, z_1)-\mathcal{U}(t,z_2))}{2\pi}\frac{1}{(Z-z_2)^2} \right\|_{H^s} \leq \frac{|\lambda|x(t)}{2\pi}\| \mathcal{U}_z\|_{L^\infty} \left\|\frac{1}{(Z-z_2)^2} \right\|_{H^s} \lesssim \epsilon^2 d_S(t)^{-\frac{3}{2}}.
\end{align*}
Here, in the last step, we use the mean value theorem. 
\end{proof}

Recall that by \eqref{ZEqnTwo}, $D_t Z = \bar{R} + v^{rot}\circ Z$, so that using the bootstrap assumption \eqref{bootstrap} and \eqref{vrotZHs32}, 
\begin{equation*}
    \|D_t Z\|_{H^s} \lesssim \epsilon + \epsilon d_S(t)^{-\frac{3}{2}} \lesssim \epsilon,
\end{equation*}
so that by the product estimate and \eqref{valpharotZHs52},
\begin{equation} \label{DtZSumdS}
\left\|D_t Z  \cdot\sum_{j=1}^2 \frac{\lambda_j i}{2\pi}\frac{1}{(Z(t,\alpha)-z_j(t))^2}\right\|_{H^s} \lesssim \epsilon^2 d_S(t)^{-\frac{5}{2}}.
\end{equation}
Note that one has the direct computation
\begin{equation*}
  D_t \sum_{j=1}^2 \frac{\lambda_j i}{2\pi}\frac{1}{Z-z_j} = \sum_{j=1}^2 \frac{\lambda_j i}{2\pi}\frac{\frac{d}{dt}z_j}{(Z-z_j)^2} - D_t Z \sum_{j=1}^2 \frac{\lambda_j i}{2\pi}\frac{1}{(Z-z_j)^2},  
\end{equation*}
we get by using \eqref{dtvrotZHs} and \eqref{DtZSumdS},
\begin{equation}\label{Dtsum2z}
  \left\|D_t \sum_{j=1}^2 \frac{\lambda_j i}{2\pi}\frac{1}{Z-z_j} \right\|_{H^s} \lesssim \epsilon d_S(t)^{-\frac{5}{2}}+ \epsilon^2 d_S(t)^{-\frac{3}{2}}. 
\end{equation}
Similarly, we have
\begin{equation}\label{1Wsum2z}
  \left\|\sum_{j=1}^2 \frac{\lambda_j i}{2\pi}\frac{1+\bar{\W}}{(\overline{Z-z_j})^2}\right\|_{H^{s+\frac{1}{2}}} \lesssim \left(1+\|\W\|_{H^{s+\frac{1}{2}}}\right) \left\|\sum_{j=1}^2 \frac{\lambda_j i}{2\pi}\frac{1}{(Z-z_j)^2}\right\|_{H^{s+\frac{1}{2}}} \lesssim\epsilon d_S(t)^{-\frac{5}{2}}. 
\end{equation}

Finally, we obtain a refined estimate for the advection velocity $b= b_1 + b_2$.
\begin{lemma}
Two parts of the advection velocity $b$ satisfy the estimates, for $s>\frac{3}{2}$, 
\begin{equation} \label{bEstTwo}
 \|b_1\|_{C^\epsilon_{*}}\lesssim_\CalAZS \|R\|_{C^\epsilon_{*}} , \quad \|b_1\|_{H^s} \lesssim_\CalAZS \|R\|_{H^s}, \quad \|b_2\|_{C^\epsilon_{*}}+\|b_2\|_{H^s}\lesssim_\CalAZS   \epsilon d_S(t)^{-\frac{3}{2}}.    
\end{equation}
\end{lemma}

\begin{proof}
As in the proof of \eqref{bEst}, it suffices to consider the estimate $b_2$ using \eqref{vrotZHs32},
\begin{align*}
 &\|T_{1-\Y}\bar{\nP}(v^{rot}\circ Z) \|_{C^\epsilon_{*}} + \| \bar{\nP}\Pi(\Y, \bar{\nP}(v^{rot}\circ Z))\|_{C^\epsilon_{*}} \lesssim (1+\|\Y \|_{L^\infty}) \|\bar{\nP}(v^{rot}\circ Z) \|_{C^\epsilon_{*}} \lesssim_\CalAZS \epsilon d_S(t)^{-\frac{3}{2}}, \\
  &\|T_{1-\Y}\bar{\nP}(v^{rot}\circ Z) \|_{H^s} + \| \bar{\nP}\Pi(\Y, \bar{\nP}(v^{rot}\circ Z))\|_{H^s} \lesssim (1+\|\Y \|_{L^\infty}) \|\bar{\nP}(v^{rot}\circ Z) \|_{H^s} \lesssim_\CalAZS \epsilon d_S(t)^{-\frac{3}{2}}.
\end{align*}
This gives the proof of \eqref{bEstTwo}.
\end{proof}
In summary, all terms in the water wave system that involve point vortices satisfy the bound \eqref{perturbative}.

\subsection{Normal form variables} \label{s:Paralin}
 We proceed further with a normal form idea which by now is commonly used in water wave systems, beginning with the work of Hunter-Ifrim-Tataru \cite{MR3535894} and Alazard-Delort \cite{MR3429478}. More recently, the normal form performed in a paradifferential setting was a key ingredient in the work of Ai-Ifrim-Tataru \cite{ai2023dimensional, MR4483135}, where it permitted  to achive a close-to-scaling well-posedness theory.
 Here, we follow this later approach and proceed  a similar fashion as in \cite{ai2023dimensional, MR4483135}.  
 From the analysis in Section \ref{s:Estimate} and  Section \ref{s:Preliminary}, we can write $(\W_t, R_t)$ as the sum of leading terms of order $\frac{3}{2}$ plus other source terms:
\begin{equation} \label{WRSysEqn}
\left\{
\begin{aligned}
&  \partial_t\W  = - T_{(1+\W)(1-\bar{\Y})}R_\alpha + G_1 + G_2 + G_3\\
&  \partial_tR  = -i T_{J^{-\frac{1}{2}}(1-\Y)^2}\W_{\alpha \alpha}+ K_1 + K_2 + K_3,
\end{aligned} 
\right.
\end{equation}
where $(G_1, K_1)$ are non-perturbative low-high terms given by
\begin{align*}
G_1 &= -T_{b}\W_\alpha - T_{\W_\alpha(1-\bar{\Y})}R + T_{(1+\W)(1-\bar{\Y})^2\bar{\W}_\alpha}R - T_{R_\alpha + \bar{R}_\alpha}\W,\\
K_1 & = -T_{b}R_\alpha - T_{R_\alpha}R + 3i T_{J^{-\frac{1}{2}}(1-\Y)^3 \W_\alpha} \W_\alpha +\frac{5}{2}iT_{\W_{\alpha \alpha}}\W - \frac{i}{2}T_{\bar{\W}_{\alpha \alpha}}\W,
\end{align*}
and $(G_2, K_2)$ are non-perturbative balanced terms given by
\begin{align*}
G_2 =& \nP \partial_\alpha\Pi(R, \bar{\W})-\nP \partial_\alpha\Pi(R+ \bar{R}, \W),\\
K_2 =& -\nP\Pi(\bar{R}, R_\alpha) -  \Pi(R_\alpha, R)+\frac{5}{2}i\Pi\left(\W, \W_{\alpha \alpha}\right) +\frac{3}{2}i\Pi\left(\W_{\alpha}, \W_{\alpha }\right)\\
&+ \frac{i}{2}\nP \Pi\left(\bar{\W}, \W_{\alpha \alpha}\right)-\frac{i}{2}\nP \Pi\left(\bar{\W}_{\alpha \alpha}, \W\right).
\end{align*}
These two source terms satisfy the Zygmund estimates for $s>0$,
\begin{equation} \label{GKOneTwoEst}
 \|(G_1, K_1) \|_{C^s_{*}} \lesssim \left(\CalAOS +\epsilon d_S(t)^{-\frac{3}{2}}\right)\|(\W, R) \|_{C^{s+1}_{*}}, \quad \|(G_2, K_2) \|_{C^s_{*}}\lesssim \CalATS \| (\W, R)\|_{C^s_*}.
\end{equation}

The rest source terms $(G_3, K_3)$ are either cubic or higher perturbative terms or source terms that involve point vortices such as $\nP[a\Y]$.
These terms satisfy the bounds
\begin{equation} \label{perturbative}
\|(G_3, K_3)\|_{\H^s} \lesssim_\CalAZS \CalAOS \CalATS \|(\W, R)\|_{\H^s}  + \epsilon d_S(t)^{-\frac{3}{2}}.
\end{equation} 
We remark that many terms in $(G_1, K_1)$ and $(G_2, K_2)$ do not have the para-coefficients as in Lemma $3.1$ of \cite{wan2024}.
This is because here we are not working in the low regularity setting as in \cite{wan2024}, and these perturbative terms can be put in $(G_3, K_3)$.

In this section, we compute normal form variables $(\W^{[1]}, R^{[1]})$, $(\W^{[2]}, R^{[2]})$ in a way that is akin to the analysis in Section $3.2$ of \cite{wan2024}, and remove most non-perturbative  low-high terms in $(G_1, K_1)$, together with all non-perturbative balanced terms $(G_2, K_2)$ in the system \eqref{WRSysEqn}.

\subsubsection{Construction of normal form variables $(\W^{[1]}, R^{[1]})$}
Here, we compute normal form variables $(\W^{[1]}, R^{[1]})$ that remove source terms $(G_1, K_1)$ modulo perturbative source terms $(G_3, K_3)$ except main part of transport terms, namely $(T_{b_1}\W_\alpha, T_{b_1}R_\alpha)$,
\begin{equation*} 
\left\{
\begin{aligned}
&  \partial_t\W^{[1]}  + T_{(1+\W)(1-\bar{\Y})}R^{[1]}_\alpha = - G_1 + T_{b_1} \W_\alpha + G_3 \\
&  \partial_t R^{[1]}  + i T_{J^{-\frac{1}{2}}(1-\Y)^2}\W^{[1]}_{\alpha \alpha} = -K_1 + T_{b_1} R_\alpha+K_3.
\end{aligned} 
\right.
\end{equation*}
We consider $(\W^{[1]}, R^{[1]})$ as the sum of low-high bilinear forms of the following type:
\begin{equation*} 
\begin{aligned}
\W^{[1]}   &= B^h_{lh}(\W, T_{1-\Y}\W) + C^h_{lh}(R, T_{J^{\frac{1}{2}}(1+\W)^2(1-\bar{\Y})}R) + B^a_{lh}(\bar{\W}, T_{1-\bar{\Y}} \W) + C^a_{lh}(\bar{R}, T_{J^{\Half}(1+\W)}R), \\
R^{[1]} &= A^h_{lh}(R, T_{1-\Y}\W) + D^h_{lh}(\W, T_{1-\Y}R) + A^a_{lh}(\bar{R}, T_{(1+\bar{\W})(1-\Y)^2}\W) + D^a_{lh}(\bar{\W}, T_{1-\bar{\Y}} R),
\end{aligned}
\end{equation*}
where $A^h_{lh}, B^h_{lh}, C^h_{lh}, D^h_{lh}$ are low-high bilinear forms of the holomorphic type, and $A^a_{lh}, B^a_{lh}, C^a_{lh}, D^a_{lh}$ are low-high bilinear forms of the mixed type.
Para-coefficients such as $T_{1-\Y}$ in the bilinear forms are needed in order to cancel the terms in $(G_1, K_1)$ with non-trivial para-coefficients.
Using direct computation \eqref{WRSysEqn} to substitute $(\W_t, R_t)$ by the leading terms $(- T_{(1+\W)(1-\bar{\Y})}R_\alpha, -i T_{J^{-\frac{1}{2}}(1-\Y)^2}\W_{\alpha \alpha})$ plus higher order terms, we get
\begin{align*}
    &\partial_t \W^{[1]}+T_{(1-\bar{\Y})(1+\W)} \partial_\alpha R^{[1]}+ \text{cubic, higher, or perturbative  terms}   \\
    =& \partial_\alpha A^h_{lh}(R,  \W)-B^h_{lh}(R_\alpha, \W)  -i C^h_{lh}(R,  \W_{\alpha \alpha})    
    -B^h_{lh}(\W, R_\alpha)-iC^h_{lh}(\W_{\alpha \alpha}, R)+\partial_\alpha D^h_{lh}(\W, R) \\
    &+\partial_\alpha A^a_{lh}(\bar{R}, \W)-B^a_{lh}(\bar{R}_\alpha, \W)- iC^a_{lh}(\bar{R}, \W_{\alpha \alpha}) 
    -B^a_{lh}(\bar{\W}, R_\alpha)+iC^a_{lh}(\bar{\W}_{\alpha \alpha}, R)+\partial_\alpha D^a_{lh}(\bar{\W}, R), \\
    &\partial_t R^{[1]} + i T_{(1-\Y)^2J^{-\Half}} \partial_\alpha^2 \W^{[1]} + \text{cubic, higher, or perturbative  terms} \\
    =& -A^h_{lh}(R, R_\alpha) +i\partial_\alpha^2 C^h_{lh}(R, R) -D^h_{lh}(R_\alpha, R) 
     -iA^h_{lh}(\W_{\alpha \alpha}, \W)+i\partial_\alpha^2 B^h_{lh}(\W, \W)-iD^h_{lh}(\W,  \W_{\alpha \alpha}) \\
     & -A^a_{lh}(\bar{R}, R_\alpha)+ i\partial_\alpha^2 C^a_{lh}(\bar{R}, R) - D^a_{lh}(\bar{R}_\alpha, R)
      +i A^a_{lh}(\bar{\W}_{\alpha \alpha}, \W)+ i\partial_\alpha^2 B^a_{lh}(\bar{\W}, \W) - iD^a_{lh}(\bar{\W}, \W_{\alpha \alpha}).
\end{align*}
For low-high holomorphic terms, holomorphic bilinear forms satisfy the  system 
\begin{equation*}
\left\{
    \begin{array}{lr}
    \partial_\alpha A^h_{lh}(R,  \W)-B^h_{lh}(R_\alpha, \W)  -i C^h_{lh}(R,  \W_{\alpha \alpha}) =  T_{R_\alpha} \W &\\
    -B^h_{lh}(\W, R_\alpha)-iC^h_{lh}(\W_{\alpha \alpha}, R)+\partial_\alpha D^h_{lh}(\W, R) = T_{\W_\alpha}R  &\\
    -A^h_{lh}(R, R_\alpha) +i\partial_\alpha^2 C^h_{lh}(R, R) -D^h_{lh}(R_\alpha, R) =   T_{R_\alpha} R &\\
     -iA^h_{lh}(\W_{\alpha \alpha}, \W)+i\partial_\alpha^2 B^h_{lh}(\W, \W)-iD^h_{lh}(\W,  \W_{\alpha \alpha}) = -3iT_{\W_\alpha} \W_\alpha -\frac{5}{2}i T_{\W_{\alpha \alpha}}\W.&  
    \end{array}
\right.
\end{equation*}
Define the symbol $\chi_{1}(\xi, \eta)$ that selects the low-high frequencies \eqref{ChiOnelh}.
By taking the Fourier transform, the symbols $\mathfrak{a}^h_{lh}$, $\mathfrak{b}^h_{lh}$, $\mathfrak{c}^h_{lh}$, $\mathfrak{d}^h_{lh}$ of low-high holomorphic bilinear forms $A^h_{lh}, B^h_{lh}, C^h_{lh}, D^h_{lh}$ then solve the system
\begin{equation*}
\left\{
    \begin{array}{lr}
    (\xi+\eta) \mathfrak{a}^h_{lh} - \xi \mathfrak{b}^h_{lh} + \eta^2 \mathfrak{c}^h_{lh} = \xi\chi_1(\xi, \eta) &\\
    \eta \mathfrak{b}^h_{lh} - \xi^2 \mathfrak{c}^h_{lh} - (\xi+\eta) \mathfrak{d}^h_{lh} = -\xi \chi_1(\xi, \eta)  &\\
    \eta \mathfrak{a}^h_{lh} + (\xi + \eta)^2 \mathfrak{c}^h_{lh} + \xi \mathfrak{d}^h_{lh} = -\xi\chi_1(\xi, \eta)  &\\
     \xi^2 \mathfrak{a}^h_{lh} -(\xi+ \eta)^2 \mathfrak{b}^h_{lh} + \eta^2 \mathfrak{d}^h_{lh} = 3\xi\eta \chi_1(\xi, \eta) +\frac{5}{2} \xi^2 \chi_1(\xi, \eta).&  
    \end{array}
\right.
\end{equation*}
The solution of this  system is given by
\begin{equation*}
 \begin{aligned}
&\mathfrak{a}^h_{lh}(\xi, \eta) = \dfrac{(-9\xi^3- 7\xi^2\eta + 4\xi\eta^2 +4\eta^3) \chi_1(\xi, \eta)}{2 \eta(9\xi^2 + 14\xi \eta + 9 \eta^2)}, \\
&\mathfrak{b}^h_{lh}(\xi, \eta) = -\dfrac{(9\xi^3+ 34\xi^2\eta +37\xi\eta^2 +12 \eta^3)\chi_1(\xi, \eta)}{2 \eta(9\xi^2 + 14\xi \eta + 9 \eta^2)}, \\
&\mathfrak{c}^h_{lh}(\xi, \eta) = -\dfrac{ (3\xi^2 + \xi \eta +2\eta^2)\chi_1(\xi, \eta)}{ 
\eta(9\xi^2 + 14\xi \eta + 9 \eta^2)}, \\
&\mathfrak{d}^h_{lh}(\xi, \eta) = \dfrac{(6\xi^3+ 5\xi^2\eta -7\xi\eta^2 -12 \eta^3)\chi_1(\xi, \eta)}{2\eta(9\xi^2 + 14\xi \eta + 9 \eta^2)}. 
\end{aligned}   
\end{equation*}

For low-high mixed terms, mixed bilinear forms satisfy the  system 
\begin{equation*}
\left\{
    \begin{array}{lr}
    \partial_\alpha A^a_{lh}(\bar{R}, \W)-B^a_{lh}(\bar{R}_\alpha, \W)- iC^a_{lh}(\bar{R}, \W_{\alpha \alpha}) =  T_{\bar{R}_\alpha} \W &\\
     -B^a_{lh}(\bar{\W}, R_\alpha)+iC^a_{lh}(\bar{\W}_{\alpha \alpha}, R)+\partial_\alpha D^a_{lh}(\bar{\W}, R)= -T_{\bar{\W}_\alpha}R  &\\
     -A^a_{lh}(\bar{R}, R_\alpha)+ i\partial_\alpha^2 C^a_{lh}(\bar{R}, R) - D^a_{lh}(\bar{R}_\alpha, R)= 0 &\\
      i A^a_{lh}(\bar{\W}_{\alpha \alpha}, \W)+ i\partial_\alpha^2 B^a_{lh}(\bar{\W}, \W) - iD^a_{lh}(\bar{\W}, \W_{\alpha \alpha})= \frac{i}{2} T_{\bar{\W}_{\alpha \alpha}}\W.&  
    \end{array}
\right.
\end{equation*}
Taking the Fourier transform, the symbols $\mathfrak{a}^a_{lh}$, $\mathfrak{b}^a_{lh}$, $\mathfrak{c}^a_{lh}$, $\mathfrak{d}^a_{lh}$ of low-high mixed bilinear forms $A^a_{lh}, B^a_{lh}, C^a_{lh}, D^a_{lh}$ solve the algebraic system
\begin{equation*}
\left\{
    \begin{array}{lr}
    (\zeta -\eta) \mathfrak{a}^a_{lh} + \eta \mathfrak{b}^a_{lh} + \zeta^2 \mathfrak{c}^a_{lh} =  -\eta \chi_1(\eta, \zeta) &\\
    \zeta \mathfrak{b}^a_{lh} + \eta^2 \mathfrak{c}^a_{lh} - (\zeta-\eta) \mathfrak{d}^a_{lh} =  -\eta \chi_1(\eta, \zeta)   &\\
    \zeta \mathfrak{a}^a_{lh} + ( \zeta -\eta)^2 \mathfrak{c}^a_{lh} - \eta \mathfrak{d}^h_{lh} = 0  &\\
     \eta^2 \mathfrak{a}^a_{lh} +( \zeta -\eta)^2 \mathfrak{b}^a_{lh} - \zeta^2 \mathfrak{d}^a_{lh} =  \frac{1}{2}\eta^2 \chi_1(\eta, \zeta).&  
    \end{array}
\right.
\end{equation*}
The solution of this algebraic system is given by
\begin{equation*}
\begin{aligned}
&\mathfrak{a}^a_{lh}(\eta, \zeta) = -\dfrac{(6\eta^3-11\eta^2\zeta + 10\eta\zeta^2 -2\zeta^3) \chi_1(\eta, \zeta)}{2 (\zeta -\eta)(4\eta^2 -4\eta \zeta + 9 \zeta^2)}, \\
&\mathfrak{b}^a_{lh}(\eta, \zeta) = \dfrac{(2\eta^3+\eta^2\zeta +\eta\zeta^2 -6 \zeta^3)\chi_1(\eta, \zeta)}{2 (\zeta -\eta)(4\eta^2 -4\eta \zeta + 9 \zeta^2)}, \\
&\mathfrak{c}^a_{lh}(\eta, \zeta) = \dfrac{ (2\eta^2 - \zeta^2)\chi_1(\eta, \zeta)}{ (\zeta -\eta)(4\eta^2 -4\eta \zeta + 9 \zeta^2)}, \\
&\mathfrak{d}^a_{lh}(\eta, \zeta) = \dfrac{(4\eta^3 -14\eta^2\zeta +13\eta\zeta^2 - 6\zeta^3)\chi_1(\eta, \zeta)}{2 (\zeta -\eta)(4\eta^2 -4\eta \zeta + 9 \zeta^2)}. 
\end{aligned}    
\end{equation*}

Since the polynomials $9\xi^2 + 14\xi \eta + 9 \eta^2,\text{ }4\eta^2 -4\eta \zeta + 9 \zeta^2 > 0$ unless the frequencies $\xi = \eta = \zeta = 0$, we get using \eqref{BFHsCmStar} and \eqref{BFHsLinfty}, for $s>\frac{3}{2}$,
\begin{equation} \label{WOneBound}
 \|(\W^{[1]}, R^{[1]}) \|_{H^s} \lesssim \CalAZS \|(\W, R)\|_{H^s}.
\end{equation}
The normal form variables $(\W^{[1]}, R^{[1]}) $ remove $(G_1, K_1)$ except sub-leading terms $(T_{b_1}\W_\alpha, T_{b_1}R_\alpha)$, and introduce cubic and higher terms.
These new terms introduced are perturbative, and can be put into $(G_3, K_3)$.

\subsubsection{Construction of normal form variables $(\W^{[2]}, R^{[2]})$}
Next, we compute normal form variables $(\W^{[2]}, R^{[2]})$ that remove source terms $(G_2, K_2)$ modulo perturbative source terms $(G_3, K_3)$,
\begin{equation*} 
\left\{
\begin{aligned}
&  \partial_t\W^{[2]}  + T_{(1+\W)(1-\bar{\Y})}R^{[2]}_\alpha = - G_2 + G_3 \\
&  \partial_t R^{[2]}  + i T_{J^{-\frac{1}{2}}(1-\Y)^2}\W^{[2]}_{\alpha \alpha} = -K_2+K_3.
\end{aligned} 
\right.
\end{equation*}
We choose $(\W^{[2]}, R^{[2]})$ as the sum of balanced bilinear forms of the following type:
\begin{align*}
\W^{[2]}   &= B^h_{hh}(\W, \W) + C^h_{hh}(R, R) + B^a_{hh}(\bar{\W},  \W) + C^a_{hh}(\bar{R}, R), \\
R^{[2]} &= A^h_{hh}(R, \W) + A^a_{hh}(\bar{R}, \W) + D^a_{hh}(\bar{\W}, R),
\end{align*}
where $A^h_{hh}, B^h_{hh}, C^h_{hh}$ are balanced bilinear forms of the holomorphic type, and $A^a_{hh}, B^a_{hh}, C^a_{hh}, D^a_{hh}$ are balanced bilinear forms of the mixed type.
Here, $B^h_{hh}$ and $C^h_{hh}$ are symmetric in two arguments.
Substituting $(\W_t, R_t)$ by their leading terms $(- R_\alpha, -i \W_{\alpha \alpha})$ plus higher-order terms, we get
\begin{align*}
    &\partial_t \W^{[2]}+T_{(1-\bar{\Y})(1+\W)} \partial_\alpha R^{[2]}+ \text{cubic, higher, or perturbative  terms}   \\
    =& \partial_\alpha A^h_{hh}(R,  \W)-B^h_{hh}(R_\alpha, \W)  -i C^h_{hh}(R,  \W_{\alpha \alpha})    
    -B^h_{hh}(\W, R_\alpha)-iC^h_{hh}(\W_{\alpha \alpha}, R) \\
    +&\partial_\alpha A^a_{hh}(\bar{R}, \W)-B^a_{hh}(\bar{R}_\alpha, \W)- iC^a_{hh}(\bar{R}, \W_{\alpha \alpha}) 
    -B^a_{hh}(\bar{\W}, R_\alpha)+iC^a_{hh}(\bar{\W}_{\alpha \alpha}, R)+\partial_\alpha D^a_{hh}(\bar{\W}, R), \\
    &\partial_t R^{[2]} + i T_{(1-\Y)^2J^{-\Half}} \partial_\alpha^2 \W^{[2]} + \text{cubic, higher, or perturbative  terms} \\
    =& -A^h_{hh}(R, R_\alpha) +i\partial_\alpha^2 C^h_{hh}(R, R) 
     -iA^h_{hh}(\W_{\alpha \alpha}, \W)+i\partial_\alpha^2 B^h_{hh}(\W, \W)  -A^a_{hh}(\bar{R}, R_\alpha)\\
     &+ i\partial_\alpha^2 C^a_{hh}(\bar{R}, R) - D^a_{hh}(\bar{R}_\alpha, R)
      +i A^a_{hh}(\bar{\W}_{\alpha \alpha}, \W)+ i\partial_\alpha^2 B^a_{hh}(\bar{\W}, \W) - iD^a_{hh}(\bar{\W}, \W_{\alpha \alpha}).
\end{align*}

For balanced holomorphic terms, holomorphic bilinear forms satisfy the system
\begin{equation*}
\left\{
    \begin{array}{lr}
    \partial_\alpha A^h_{hh}(R,  \W)-2B^h_{hh}(R_\alpha, \W)  -2i C^h_{hh}(R,  \W_{\alpha \alpha}) =  \partial_\alpha \Pi(R, \W)  &\\
    -A^h_{hh}(R, R_\alpha) +i\partial_\alpha^2 C^h_{hh}(R, R) =  \Pi(R_\alpha, R)  &\\
     -iA^h_{hh}(\W_{\alpha \alpha}, \W)+i\partial_\alpha^2 B^h_{hh}(\W, \W) = -\frac{5}{2}i\Pi\left(\W, \W_{\alpha \alpha}\right) -\frac{3}{2}i\Pi\left(\W_{\alpha}, \W_{\alpha }\right).&  
    \end{array}
\right.
\end{equation*}

Define the symbol $\chi_{2}(\xi, \eta)$ that selects the balanced frequencies \eqref{ChiTwohh} and $m_{sym}$ be the symmetrization of the bilinear symbol $m$, symbols $\mathfrak{a}^h_{hh}$, $\mathfrak{b}^h_{hh}$, $\mathfrak{c}^h_{hh}$ of bilinear forms  $A^h_{hh}, B^h_{hh}, C^h_{hh}$ solve the system
\begin{equation*}
\left\{
    \begin{array}{lr}
    (\xi+\eta) \mathfrak{a}^h_{hh} - 2\xi \mathfrak{b}^h_{hh} + 2\eta^2 \mathfrak{c}^h_{hh} = (\xi +\eta) \chi_2(\xi, \eta) &\\
    (\eta\mathfrak{a}^h_{hh})_{sym} + (\xi + \eta)^2 \mathfrak{c}^h_{hh} = -\frac{1}{2}(\xi+\eta)\chi_2(\xi, \eta)  &\\
     (\xi^2 \mathfrak{a}^h_{hh})_{sym} -(\xi+ \eta)^2 \mathfrak{b}^h_{hh}  = \frac{1}{4}(5\xi^2+6\xi \eta + 5 \eta^2) \chi_2(\xi, \eta).&  
    \end{array}
\right.
\end{equation*}
The solution of the system is given by 
\begin{equation*}
 \begin{aligned}
&\mathfrak{a}^h_{hh}(\xi, \eta) = -\dfrac{(27\xi^4+39\xi^3\eta + 23\xi^2 \eta^2-3\xi \eta^3-6\eta^4)\chi_2(\xi, \eta)}{2\xi 
\eta(9\xi^2 + 14\xi \eta + 9 \eta^2)}, \\
&\mathfrak{b}^h_{hh}(\xi, \eta) = -\dfrac{3(\xi + \eta)^2(9\xi^2+10\xi\eta + 9\eta^2)\chi_2(\xi, \eta)}{4\xi 
\eta(9\xi^2 + 14\xi \eta + 9 \eta^2)}, \\
&\mathfrak{c}^h_{hh}(\xi, \eta) = -\dfrac{3(\xi + \eta)^3\chi_2(\xi, \eta)}{2\xi 
\eta(9\xi^2 + 14\xi \eta + 9 \eta^2)}. 
\end{aligned}   
\end{equation*}

For low-high mixed terms, mixed bilinear forms solve the system
\begin{equation*}
\left\{
    \begin{array}{lr}
    \partial_\alpha A^a_{hh}(\bar{R}, \W)-B^a_{hh}(\bar{R}_\alpha, \W)- iC^a_{hh}(\bar{R}, \W_{\alpha \alpha}) = \nP\partial_\alpha \Pi(\bar{R}, \W)   &\\
     -B^a_{hh}(\bar{\W}, R_\alpha)+iC^a_{hh}(\bar{\W}_{\alpha \alpha}, R)+\partial_\alpha D^a_{hh}(\bar{\W}, R)= -\nP\partial_\alpha \Pi(\bar{\W}, R)  &\\
     -A^a_{hh}(\bar{R}, R_\alpha)+ i\partial_\alpha^2 C^a_{hh}(\bar{R}, R) - D^a_{hh}(\bar{R}_\alpha, R)= \nP \Pi(\bar{R}, R_\alpha) &\\
      i A^a_{hh}(\bar{\W}_{\alpha \alpha}, \W)+ i\partial_\alpha^2 B^a_{hh}(\bar{\W}, \W) - iD^a_{hh}(\bar{\W}, \W_{\alpha \alpha})= -\frac{i}{2}\nP \Pi\left(\bar{\W}, \W_{\alpha \alpha}\right)+ \frac{i}{2}\nP \Pi\left(\bar{\W}_{\alpha \alpha}, \W\right).&  
    \end{array}
\right.
\end{equation*}

Then taking the Fourier transform, symbols $\mathfrak{a}^a_{hh}$, $\mathfrak{b}^a_{hh}$, $\mathfrak{c}^a_{hh}$, $\mathfrak{d}^a_{hh}$ of bilinear forms  $A^a_{hh}, B^a_{hh}, C^a_{hh},  D^a_{hh}$ solve the algebraic system
\begin{equation*}
\left\{
    \begin{array}{lr}
    (\zeta -\eta) \mathfrak{a}^a_{hh} + \eta \mathfrak{b}^a_{hh} + \zeta^2 \mathfrak{c}^a_{hh} = (\zeta -\eta) \chi_2(\eta, \zeta)1_{\zeta< \eta} &\\
    \zeta \mathfrak{b}^a_{hh} + \eta^2 \mathfrak{c}^a_{hh} - (\zeta-\eta) \mathfrak{d}^a_{hh} =  (\zeta -\eta) \chi_2(\eta, \zeta)1_{\zeta< \eta}   &\\
    \zeta \mathfrak{a}^a_{hh} + ( \zeta -\eta)^2 \mathfrak{c}^a_{hh} - \eta \mathfrak{d}^h_{hh} = -\zeta\chi_2(\eta, \zeta)1_{\zeta< \eta}  &\\
     \eta^2 \mathfrak{a}^a_{hh} +( \zeta -\eta)^2 \mathfrak{b}^a_{hh} - \zeta^2 \mathfrak{d}^a_{hh} =  \frac{1}{2}(\eta^2-\zeta^2) \chi_2(\eta, \zeta)1_{\zeta< \eta}.&  
    \end{array}
\right.
\end{equation*}
Here, the indicator function $1_{\zeta< \eta}$ is the symbol for the holomorphic projection $\nP$. 
The solution of this algebraic system is 
\begin{equation*}
\begin{aligned}
&\mathfrak{a}^a_{hh}(\eta, \zeta) = \dfrac{3(2\eta^4-5\eta^3 \zeta +6\eta^2 \zeta^2  -3\eta \zeta^3 + 2\zeta^4) \chi_2(\eta, \zeta)1_{\zeta< \eta}}{2\eta (\eta -\zeta)(4\eta^2 -4\eta \zeta + 9 \zeta^2)}, \\
&\mathfrak{b}^a_{hh}(\eta, \zeta) = -\dfrac{(2\eta^4 -3\eta^3\zeta +5\eta^2\zeta^2 -17\eta\zeta^2 + 9\zeta^4)\chi_2(\eta, \zeta)1_{\zeta< \eta}}{2\eta (\eta - \zeta) (4\eta^2 -4\eta \zeta + 9 \zeta^2)}, \\
&\mathfrak{c}^a_{hh}(\eta, \zeta) = -\dfrac{ (2\eta^3-3\eta \zeta^2 + 3\zeta^3)\chi_2(\eta, \zeta)1_{\zeta< \eta}}{\eta(\eta -\zeta) (4\eta^2 -4\eta \zeta + 9 \zeta^2)}, \\
&\mathfrak{d}^a_{hh}(\eta, \zeta) = -\dfrac{(4\eta^4-22\eta^3 \zeta + 29\eta^2 \zeta^2 - 26\eta \zeta^3  +9\zeta^4)\chi_2(\eta, \zeta)1_{\zeta< \eta}}{2 \eta (\eta - \zeta)(4\eta^2 -4\eta \zeta + 9 \zeta^2)}. 
\end{aligned}    
\end{equation*}

The denominator of each symbol does not equal zero unless frequencies $\xi = \eta = \zeta=0$.
Hence, estimating using \eqref{HHBilinearHCH}, we get that for $s>\frac{3}{2}$,
\begin{equation} \label{WTwoBound}
 \|(\W^{[2]}, R^{[2]}) \|_{H^s} \lesssim \CalAZS \|(\W, R)\|_{H^s}.
\end{equation}
The normal form variables $(\W^{[2]}, R^{[2]})$ remove source terms $(G_2, K_2)$ and introduce cubic and higher terms.
These new terms introduced are perturbative, and can be put into $(G_3, K_3)$.

To conclude this section, adding the normal form variables $(\W^{[1]}, R^{[1]})$, $(\W^{[2]}, R^{[2]})$ to $(\W, R)$, and taking into account the bounds \eqref{WOneBound}, \eqref{WTwoBound}, we get the following result.
\begin{proposition} \label{t:SymbolWR}
Suppose that $(\W, R)$ solve the water wave system \eqref{WRSysEqn}.
Then there exist normal form variables $(\tilde{\W}, \tilde{R}): = (\W+ \W^{[1]}+ \W^{[2]}, R+R^{[1]}+R^{[2]})$ that solve the system 
 \begin{equation*}
 \left\{
    \begin{array}{lr}
     \partial_t\tilde{\W} +T_{(1-\bar{\Y})(1+\W)}\tilde{R}_\alpha = -T_{b_1}\W_\alpha+G_3  &\\
    \partial_t \tilde{R}+ i T_{J^{-\frac{1}{2}}(1-\Y)^2}\tilde{\W}_{\alpha\alpha} =-T_{b_1}R_\alpha +K_3 ,&
             \end{array}
\right.  
\end{equation*}
The source terms $(G_3, K_3)$ satisfy the bound \eqref{perturbative}.
Moreover, the modified normal form variables $(\tilde{\W}, \tilde{R})$ are equivalent to the original $(\W, R)$ in the sense that for $s>\frac{3}{2}$,
\begin{equation} \label{NormEquRelation}
    \|(\tilde{\W}, \tilde{R} )\|_{H^s} \approx_\CalAZS \|(\W, R) \|_{H^s}.
\end{equation}
\end{proposition}
The modified variables $(\tilde{\W}, \tilde{R})$ are also holomorphic functions like $(\W, R)$.

\subsection{Construction of the modified energy} \label{s:CubicEnergy}
 In this section, we construct the \emph{quasilinear modified energy} as introduced by Hunter-Ifrim-Tataru-Wong \cite{MR3348783} for the Burger-Hilbert equation and then successfully implemented and adjusted to fit a more general setting, namely,  the water wave models studied by Ifrim-Tataru and their collaborators \cite{MR3535894, MR3499085, MR3625189,  ai2023dimensional, MR4483135, MR3667289, MR3869381}.
We call these energies $\mathcal{E}_s$ and prove Theorem \ref{t:CubicEnergy}.
Instead of using $(\W, R)$ to build the main part of the energy, we use the modified normal form variables $(\tilde{\W}, \tilde{R})$.
Inspired by the leading term of the energy $E_s$ in Section \ref{s:Energy}, we choose the energy
\begin{equation*}
\mathcal{E}_s^{[2]} := \int T_{J^{-\frac{3}{2}}}\langle D \rangle^{s+\Half}\tilde{\W} \cdot \langle D \rangle^{s+\Half}\bar{\tilde{\W}}  + \langle D \rangle^{s}\tilde{R} \cdot \langle D \rangle^{s}\bar{\tilde{R}}\,d\alpha
\end{equation*}
as the leading part of $\mathcal{E}_s$.
This energy $\mathcal{E}_s^{[2]}$ satisfies the norm equivalence property \eqref{normEqnTwo} due to \eqref{NormEquRelation}.
For the time derivative of $\mathcal{E}_s^{[2]}$, we compute
\begin{align*}
    &\frac{d}{dt}\mathcal{E}_s^{[2]} = \int T_{\partial_tJ^{-\frac{3}{2}}}\langle D \rangle^{s+\Half}\tilde{\W} \cdot \langle D \rangle^{s+\Half}\bar{\tilde{\W}} \,d\alpha \\
    +& 2\Re\int T_{J^{-\frac{3}{2}}}\langle D \rangle^{s+\Half}\tilde{\W} \cdot \langle D \rangle^{s+\Half}\bar{\tilde{\W}}_t + \langle D \rangle^{s}\tilde{R} \cdot \langle D \rangle^{s}\bar{\tilde{R}}_t \,d\alpha \\
    =& \int T_{\partial_tJ^{-\frac{3}{2}}}\langle D \rangle^{s+\Half}\tilde{\W} \cdot \langle D \rangle^{s+\Half}\bar{\tilde{\W}} \,d\alpha \\
    -& 2\Re\int \langle D \rangle^{s+\Half}\tilde{\W} \cdot T_{J^{-\frac{3}{2}}}\langle D \rangle^{s+\Half}T_{(1-\Y)(1+\bar{\tilde{\W}})}\bar{\tilde{R}}_\alpha +i\langle D \rangle^{s}\tilde{R} \cdot \langle D \rangle^{s}T_{J^{-\frac{1}{2}}(1-\bar{\Y})^2}\bar{\tilde{\W}}_{\alpha \alpha} \,d\alpha\\
    -& 2\Re\int T_{J^{-\frac{3}{2}}}\langle D \rangle^{s+\Half}\tilde{\W} \cdot \langle D \rangle^{s+\Half}T_{b_1}\bar{\W}_\alpha +\langle D \rangle^{s}\tilde{R} \cdot \langle D \rangle^{s}T_{b_1}\bar{R}_\alpha \,d\alpha\\
    +&2\Re\int T_{J^{-\frac{3}{2}}}\langle D \rangle^{s+\Half}\tilde{\W} \cdot \langle D \rangle^{s+\Half}\bar{G}_3 + \langle D \rangle^{s}\tilde{R} \cdot \langle D \rangle^{s}\bar{K}_3 \,d\alpha.
\end{align*}
In the following, we write perturbative integrals for integrals $I$ that satisfy the bound
\begin{equation} \label{IboundDSt}
   |I(t)|\lesssim_\CalAZS \left(\CalAOS\CalATS + d_S(t)^{-\frac{3}{2}}\right) \|(\W, R)\|_{\H^s}^2+ \epsilon^2 d_S(t)^{-\frac{3}{2}}.
\end{equation}
We now consider each term in $\frac{d}{dt}\mathcal{E}_s^{[2]}$.
For the first term in $\frac{d}{dt}\mathcal{E}_s^{[2]}$, 
\begin{equation*}
    \partial_t J^{s} = 2sJ^{-\frac{3}{2}} \Re (1-\Y)\partial_t\W = 2s \Re R_\alpha + E, \quad \|E\|_{L^\infty} \lesssim \CalAOS \CalATS + \epsilon d_S(t)^{-\frac{3}{2}}.
\end{equation*}
Hence, using the norm equivalence relation \eqref{NormEquRelation}, we can write
\begin{align*}
 &\int T_{\partial_tJ^{-\frac{3}{2}}}\langle D \rangle^{s+\Half}\tilde{\W} \cdot \langle D \rangle^{s+\Half}\bar{\tilde{\W}} \,d\alpha  \\
 =& -3\int T_{\Re R_\alpha}\langle D \rangle^{s+\Half}\tilde{\W} \cdot \langle D \rangle^{s+\Half}\bar{\tilde{\W}} \,d\alpha + \text{perturbative integrals}.   
\end{align*}
For the second term in $\frac{d}{dt}\mathcal{E}_s^{[2]}$, since $\| \W\|_{L^\infty} \ll 1$, we use Taylor expansion to write
\begin{align*}
 &J^{-\frac{3}{2}} = 1-\frac{3}{2}(\W+\bar{\W})+ \text{higher terms}, \quad (1-\Y)(1+\bar{\W}) = 1-\W +\bar{\W}+ \text{higher terms}, \\
 &J^{-\frac{1}{2}}(1-\bar{\Y})^2 = 1-\frac{1}{2}\W - \frac{5}{2}\bar{\W} + \text{higher terms}.
\end{align*}
Hence, the quadratic integral terms get canceled, the quartic and higher integral terms are perturbative using paradifferential calculus \eqref{CompositionPara}, so that only cubic terms of this integral remains. 
\begin{align*}
 -& 2\Re\int \langle D \rangle^{s+\Half}\tilde{\W} \cdot T_{J^{-\frac{3}{2}}}\langle D \rangle^{s+\Half}T_{(1-\Y)(1+\bar{\W})}\bar{\tilde{R}}_\alpha +i\langle D \rangle^{s}\tilde{R} \cdot \langle D \rangle^{s}T_{J^{-\frac{1}{2}}(1-\bar{\Y})^2}\bar{\tilde{\W}}_{\alpha \alpha} \,d\alpha\\
 = & 3\Re\int i\langle D \rangle^{s+\Half}\tilde{\W} \cdot T_{\W + \bar{\W}}\langle D \rangle^{s+\frac{3}{2}}\bar{\tilde{R}} \,d\alpha+ \Re\int i\langle D \rangle^{s+\frac{3}{2}}\tilde{R} \cdot \langle D \rangle^{s-\frac{3}{2}}T_{\W+ 5\bar{\W}} \langle D \rangle^2\bar{\tilde{\W}} \,d\alpha \\
 &+2\Re \int i\langle D \rangle^{s+\Half}\tilde{\W} \cdot\langle D \rangle^{s+\frac{1}{2}} T_{\W - \bar{\W}} \langle D \rangle\bar{\tilde{R}} \,d\alpha + \text{perturbative integrals}.
\end{align*}
For the third term in $\frac{d}{dt}\mathcal{E}_s^{[2]}$, we write $b_1 = 2\Re R - 2\Re \nP[\bar{\Y} R]$.
The terms with $2\Re R$ are cubic terms that are not perturbative.
The terms with $2\Re \nP[\bar{\Y} R]$ are quartic, and is perturbative using  paradifferential calculus \eqref{CompositionPara},  and integration by parts to shift the $\alpha$-derivative to para-coefficients at low frequencies.
In addition, one can replace $(\bar{\W}_\alpha, \bar{R}_\alpha)$ by $(\bar{\tilde{\W}}_\alpha, \bar{\tilde{R}}_\alpha)$, since the difference integral is perturbative using integration by parts to shift the $\alpha$-derivatives to para-coefficients.
Hence,
\begin{align*}
&- 2\Re\int T_{J^{-\frac{3}{2}}}\langle D \rangle^{s+\Half}\tilde{\W} \cdot \langle D \rangle^{s+\Half}T_{b_1}\bar{\W}_\alpha +\langle D \rangle^{s}\tilde{R} \cdot \langle D \rangle^{s}T_{b_1}\bar{R}_\alpha \,d\alpha\\
=& - 4\Re\int \langle D \rangle^{s+\Half}\tilde{\W} \cdot \langle D \rangle^{s+\Half}T_{\Re R}\bar{\W}_\alpha +\langle D \rangle^{s}\tilde{R} \cdot \langle D \rangle^{s}T_{\Re R}\bar{R}_\alpha \,d\alpha + \text{perturbative integrals}\\
=& - 4\Re\int \langle D \rangle^{s+\Half}\tilde{\W} \cdot \langle D \rangle^{s+\Half}T_{\Re R}\bar{\tilde{\W}}_\alpha +\langle D \rangle^{s}\tilde{R} \cdot \langle D \rangle^{s}T_{\Re R}\bar{\tilde{R}}_\alpha \,d\alpha + \text{perturbative integrals} .
\end{align*}
The last term in $\frac{d}{dt}\mathcal{E}_s^{[2]}$ is perturbative due to the bound for $(G_3, K_3)$ \eqref{perturbative}.
Therefore, collecting all the computations above, 
\begin{align*}
&\frac{d}{dt}\mathcal{E}_s^{[2]} =    -3\Re \int T_{ R_\alpha}\langle D \rangle^{s+\Half}\tilde{\W} \cdot \langle D \rangle^{s+\Half}\bar{\tilde{\W}} \,d\alpha +3\Re\int i\langle D \rangle^{s+\Half}\tilde{\W} \cdot T_{\W + \bar{\W}}\langle D \rangle^{s+\frac{3}{2}}\bar{\tilde{R}} \,d\alpha \\
&+ \Re\int i\langle D \rangle^{s+\frac{3}{2}}\tilde{R} \cdot \langle D \rangle^{s-\frac{3}{2}}T_{\W+ 5\bar{\W}} \langle D \rangle^2\bar{\tilde{\W}} \,d\alpha +2\Re \int i\langle D \rangle^{s+\Half}\tilde{\W} \cdot\langle D \rangle^{s+\frac{1}{2}} T_{\W - \bar{\W}} \langle D \rangle\bar{\tilde{R}} \,d\alpha\\
& - 4\Re\int \langle D \rangle^{s+\Half}\tilde{\W} \cdot \langle D \rangle^{s+\Half}T_{\Re R}\bar{\tilde{\W}}_\alpha +\langle D \rangle^{s}\tilde{R} \cdot \langle D \rangle^{s}T_{\Re R}\bar{\tilde{R}}_\alpha \,d\alpha + \text{perturbative integrals}.
\end{align*}
In order to finish the proof of Theorem \ref{t:CubicEnergy}, we need to find a cubic energy $\mathcal{E}_s^{[3]}$ such that
\begin{equation*}
    |\mathcal{E}_s^{[3]}| \lesssim \CalAZS \|(\W, R) \|^2_{\H^s}, \quad \frac{d}{dt}\mathcal{E}_s^{[3]} = -\frac{d}{dt}\mathcal{E}_s^{[2]}+ \text{perturbative integrals}.
\end{equation*}

We first construct the cubic energy $\mathcal{E}_{s,1}^{[3]}$ as the first part of $\mathcal{E}_s^{[3]}$.
It satisfies the equation
\begin{equation*}
 \frac{d}{dt}\mathcal{E}_{s,1}^{[3]} = 3\Re\int T_{R_\alpha}\langle D \rangle^{s+\Half}\tilde{\W} \cdot \langle D \rangle^{s+\Half}\bar{\tilde{\W}} \,d\alpha + \text{perturbative integrals}.  
\end{equation*}
We consider the cubic energy $\mathcal{E}_{s,1}^{[3]}$ of the following form
\begin{align*}
    \mathcal{E}_{s,1}^{[3]} =& \Re\int T_{J^{\frac{1}{2}}(1+\bar{\W})^2}A_1(R, \tilde{\W}, \bar{\tilde{R}}) + T_{J^{\frac{1}{2}}(1+\W)^2}B_1(R, \tilde{R}, \bar{\tilde{\W}}) \,d\alpha \\
    +& \Re \int T_{(1-\Y)(1+\bar{\W})}C_1(\W, \tilde{\W}, \bar{\tilde{\W}}) + T_{J^{\frac{1}{2}}(1+\bar{\W})^2}D_1(\W, \tilde{R}, \bar{\tilde{R}}) \,d\alpha,
\end{align*}
where $A_1,B_1,C_1, D_1$ are paradifferential cubic forms. 
Para-coefficients such as $J^{\frac{1}{2}}(1+\bar{\W})^2$ are needed in order to cancel the para-coefficients in the system.
Here, we  write $\mathfrak{a}_1(\xi, \eta, \zeta)$ for the symbol of $A_1(R, \tilde{\W}, \bar{\tilde{R}})$, and the other three symbols are defined in the same way.
When the time derivative falls on para-coefficients, it produces quartic and higher-order perturbative integrals.
We compute
\begin{align*}
 &\frac{d}{dt}\mathcal{E}_{s,1}^{[3]} = \Re\int A_1(R, \tilde{\W}, i\bar{\tilde{\W}}_{\alpha \alpha}) + B_1(R, -i\tilde{\W}_{\alpha\alpha}, \bar{\tilde{\W}})+ C_1(-R_\alpha, \tilde{\W}, \bar{\tilde{\W}}) \, d\alpha \\
&+ \Re \int B_1(-i \W_{\alpha \alpha}, \tilde{R}, \bar{\tilde{\W}}) + C_1(\W, -\tilde{R}_{\alpha}, \bar{\tilde{\W}})+ D_1(\W, \tilde{R}, i\bar{\tilde{\W}}_{\alpha\alpha}) \, d\alpha \\
&+ \Re \int A_1(-i\W_{\alpha \alpha }, \tilde{\W}, \bar{\tilde{R}}) + C_1(\W, \tilde{\W}, -\bar{\tilde{R}}_\alpha)+ D_1(\W, -i\tilde{\W}_{\alpha \alpha}, \bar{\tilde{R}}) \, d\alpha \\
&+ \Re \int A_1(R, -\tilde{R}_\alpha, \bar{\tilde{R}}) + B_1(R, \tilde{R}, -\bar{\tilde{R}}_\alpha)+ D_1(-R_\alpha, \tilde{R}, \bar{\tilde{R}}) \, d\alpha  + \text{quartic and higher integrals}.
\end{align*}
Using the Plancherel theorem, $\xi + \eta = \zeta$, and symbols $\mathfrak{a}_1(\xi, \eta, \zeta), \mathfrak{b}_1(\xi, \eta, \zeta), \mathfrak{c}_1(\xi, \eta, \zeta), \mathfrak{d}_1(\xi, \eta, \zeta)$ solve the system 
\begin{equation*}
\left\{
    \begin{array}{lr}
    (\xi+\eta)^2 \mathfrak{a}_1 - \eta^2 \mathfrak{b}_1 + \xi \mathfrak{c}_1 = -3 \xi\langle\eta \rangle^{s+\frac{1}{2}} \langle \xi +\eta \rangle^{s+\frac{1}{2}}\chi_1(\xi, \eta) &\\
    \xi^2 \mathfrak{b}_1 - \eta \mathfrak{c}_1 - (\xi+\eta)^2 \mathfrak{d}_1 = 0  &\\
    \xi^2 \mathfrak{a}_1 + (\xi+\eta) \mathfrak{c}_1 + \eta^2 \mathfrak{d}_1 = 0  &\\
     -\eta \mathfrak{a}_1 +(\xi +\eta) \mathfrak{b}_1 - \xi \mathfrak{d}_1 =  0.&  
    \end{array}
\right.
\end{equation*}
We get the solutions of the system:
\begin{equation*}
 \begin{aligned}
&\mathfrak{a}_1 = -\dfrac{9(\xi+\eta)^2}{2 \eta(9\xi^2 + 14\xi \eta + 9 \eta^2)}\langle\eta \rangle^{s+\frac{1}{2}} \langle \xi +\eta \rangle^{s+\frac{1}{2}}\chi_1(\xi, \eta), \\
&\mathfrak{b}_1 = -\dfrac{3(2\xi^2+ 3\xi\eta +3 \eta^2)}{2 \eta(9\xi^2 + 14\xi \eta + 9 \eta^2)}\langle\eta \rangle^{s+\frac{1}{2}} \langle \xi +\eta \rangle^{s+\frac{1}{2}}\chi_1(\xi, \eta), \\
&\mathfrak{c}_1 = \dfrac{ 3\xi(3\xi^2 + 3\xi \eta +2\eta^2)}{ 
\eta(9\xi^2 + 14\xi \eta + 9 \eta^2)}\langle\eta \rangle^{s+\frac{1}{2}} \langle \xi +\eta \rangle^{s+\frac{1}{2}}\chi_1(\xi, \eta), \\
&\mathfrak{d}_1 = -\dfrac{6\xi(\xi + \eta)}{2\eta(9\xi^2 + 14\xi \eta + 9 \eta^2)}\langle\eta \rangle^{s+\frac{1}{2}} \langle \xi +\eta \rangle^{s+\frac{1}{2}}\chi_1(\xi, \eta). 
\end{aligned}   
\end{equation*}
Using \eqref{BFHsCmStar}, the energy $\mathcal{E}_{s,1}^{[3]}$ satisfies the estimate
\begin{equation*}
    |\mathcal{E}_{s,1}^{[3]}| \lesssim_\CalAZS \CalAZS\|(\W, R) \|_{\H^s}^2.
\end{equation*}
In addition, the extra quartic and higher integral terms are perturbative.

Next, we construct the cubic energy $\mathcal{E}_{s,2}^{[3]}$ as the second part of $\mathcal{E}_s^{[3]}$ such that 
\begin{align*}
 \frac{d}{dt}\mathcal{E}_{s,2}^{[3]} =& -3\Re\int i\langle D \rangle^{s+\Half}\tilde{\W} \cdot T_{\W + \bar{\W}}\langle D \rangle^{s+\frac{3}{2}}\bar{\tilde{R}} \,d\alpha+ \Re\int i  \langle D \rangle^{s-\frac{3}{2}}T_{5\W + \bar{\W}} \langle D \rangle^2\tilde{\W} \cdot\langle D \rangle^{s+\frac{3}{2}}\bar{\tilde{R}} \,d\alpha\\
 &-2\Re \int i\langle D \rangle^{s+\Half}\tilde{\W} \cdot\langle D \rangle^{s+\frac{1}{2}} T_{\W - \bar{\W}} \langle D \rangle\bar{\tilde{R}} \,d\alpha+\text{perturbative integrals}.
\end{align*}
We consider the cubic energy $\mathcal{E}_{s,2}^{[3]}$ of the following form
\begin{align*}
\mathcal{E}_{s,2}^{[3]} &= \Re\int T_{J^{\frac{1}{2}}(1+\bar{\W})^2}A_2^h(R, \tilde{\W}, \bar{\tilde{R}}) + T_{J^{\frac{1}{2}}(1+\W)^2}B_2^h(R, \tilde{R}, \bar{\tilde{\W}}) \,d\alpha\\
&+ \Re \int T_{(1-\Y)(1+\bar
\W)}C_2^h(\W, \tilde{\W}, \bar{\tilde{\W}}) + T_{J^{\frac{1}{2}}(1+\bar{\W})^2}D_2^h(\W, \tilde{R}, \bar{\tilde{R}}) \,d\alpha\\
&+\Re\int T_{J^{\frac{1}{2}}(1+\bar{\W})^2}A_2^a(\bar{R}, \tilde{\W}, \bar{\tilde{R}}) + T_{J^{\frac{1}{2}}(1+\W)^2}B_2^a(\bar{R}, \tilde{R}, \bar{\tilde{\W}}) \,d\alpha \\
&+\Re \int T_{(1+\W)(1-\bar{\Y})}C_2^a(\bar{\W}, \tilde{\W}, \bar{\tilde{\W}}) + T_{J^{\frac{1}{2}}(1+\W)^2}D_2^a(\bar{\W}, \tilde{R}, \bar{\tilde{R}}) \,d\alpha,
\end{align*}
where $A_2^h,B_2^h,C_2^h, D_2^h$ and $A_2^a,B_2^a,C_2^a, D_2^a$ are paradifferential cubic forms. 
Here, we  write $\mathfrak{a}_2^h(\xi, \eta, \zeta)$ for the symbol of $A_2^h(R, \tilde{\W}, \bar{\tilde{R}})$, and  $\mathfrak{a}_2^a(\xi, \zeta, \eta)$ for the symbol of $A_2^a(\bar{R}, \tilde{\W}, \bar{\tilde{R}})$.
Other symbols are defined in the same way.
Then $\xi + \eta = \zeta$.

When the time derivative falls on para-coefficients, it produces quartic and higher-order perturbative integrals.
The time derivative of the energy $\mathcal{E}_{s,2}^{[3]}$ is given by
\begin{align*}
 &\frac{d}{dt}\mathcal{E}_{s,2}^{[3]} = \Re\int A_2^h(R, \tilde{\W}, i\bar{\tilde{\W}}_{\alpha \alpha}) + B_2^h(R, -i\tilde{\W}_{\alpha\alpha}, \bar{\tilde{\W}})+ C_2^h(-R_\alpha, \tilde{\W}, \bar{\tilde{\W}}) \, d\alpha \\
&+ \Re \int B_2^h(-i \W_{\alpha \alpha}, \tilde{R}, \bar{\tilde{\W}}) + C_2^h(\W, -\tilde{R}_{\alpha}, \bar{\tilde{\W}})+ D_2^h(\W, \tilde{R}, i\bar{\tilde{\W}}_{\alpha\alpha}) \, d\alpha \\
&+ \Re \int A_2^h(-i\W_{\alpha \alpha }, \tilde{\W}, \bar{\tilde{R}}) + C_2^h(\W, \tilde{\W}, -\bar{\tilde{R}}_\alpha)+ D_2^h(\W, -i\tilde{\W}_{\alpha \alpha}, \bar{\tilde{R}}) \, d\alpha \\
&+ \Re \int A_2^h(R, -\tilde{R}_\alpha, \bar{\tilde{R}}) + B_2^h(R, \tilde{R}, -\bar{\tilde{R}}_\alpha)+ D_2^h(-R_\alpha, \tilde{R}, \bar{\tilde{R}}) \, d\alpha\\
&+\Re\int A_2^a(\bar{R}, \tilde{\W}, i\bar{\tilde{\W}}_{\alpha \alpha}) + B_2^a(\bar{R}, -i\tilde{\W}_{\alpha\alpha}, \bar{\tilde{\W}})+ C_2^a(-\bar{R}_\alpha, \tilde{\W}, \bar{\tilde{\W}}) \, d\alpha \\
&+ \Re \int B_2^a(i \bar{\W}_{\alpha \alpha}, \tilde{R}, \bar{\tilde{\W}}) + C_2^a(\bar{\W}, -\tilde{R}_{\alpha}, \bar{\tilde{\W}})+ D_2^a(\bar{\W}, \tilde{R}, i\bar{\tilde{\W}}_{\alpha\alpha}) \, d\alpha \\
&+ \Re \int A_2^a(i\bar{\W}_{\alpha \alpha }, \tilde{\W}, \bar{\tilde{R}}) + C_2^a(\bar{\W}, \tilde{\W}, -\bar{\tilde{R}}_\alpha)+ D_2^a(\bar{\W}, -i\tilde{\W}_{\alpha \alpha}, \bar{\tilde{R}}) \, d\alpha \\
&+ \Re \int A_2^a(\bar{R}, -\tilde{R}_\alpha, \bar{\tilde{R}}) + B_2^a(\bar{R}, \tilde{R}, -\bar{\tilde{R}}_\alpha)+ D_2^a(-\bar{R}_\alpha, \tilde{R}, \bar{\tilde{R}}) \, d\alpha  + \text{quartic and higher integrals}.
\end{align*}
To eliminate non-perturbative terms with para-coefficient $T_{\W}$, symbols $\mathfrak{a}_2^h(\xi, \eta, \zeta)$, $\mathfrak{b}_2^h(\xi, \eta, \zeta)$, $\mathfrak{c}_2^h(\xi, \eta, \zeta)$, $\mathfrak{d}_2^h(\xi, \eta, \zeta)$ need to solve the system
\begin{equation*}
\left\{
    \begin{array}{lr}
    (\xi+\eta)^2 \mathfrak{a}_2^h - \eta^2 \mathfrak{b}_2^h + \xi \mathfrak{c}_2^h = 0 &\\
    \xi^2 \mathfrak{b}_2^h - \eta \mathfrak{c}_2^h - (\xi+\eta)^2 \mathfrak{d}_2^h = 0  &\\
    \xi^2 \mathfrak{a}_2^h + (\xi+\eta) \mathfrak{c}_2^h + \eta^2 \mathfrak{d}_2^h = -3\langle\eta\rangle^{s+\frac{1}{2}}\langle\xi+\eta\rangle^{s+\frac{3}{2}}+5\langle \eta\rangle^2 \langle\xi +\eta\rangle^{2s} -2\langle\eta\rangle^{2s+1}\langle\xi+\eta\rangle &\\
     -\eta \mathfrak{a}_2^h +(\xi +\eta) \mathfrak{b}_2^h - \xi \mathfrak{d}_2^h = 0.&  
    \end{array}
\right.
\end{equation*}
The solutions of the symbols are given by 
\begin{equation*}
 \begin{aligned}
&\mathfrak{a}_2^h = -\dfrac{(3\xi^2 + 3\xi\eta + 2 \eta^2)\chi_1(\xi, \eta)}{ \xi\eta(9\xi^2 + 14\xi \eta + 9 \eta^2)}(-3\langle\eta\rangle^{s+\frac{1}{2}}\langle\xi+\eta\rangle^{s+\frac{3}{2}}+5\langle \eta\rangle^2 \langle\xi +\eta\rangle^{2s} -2\langle\eta\rangle^{2s+1}\langle\xi+\eta\rangle), \\
&\mathfrak{b}_2^h = -\dfrac{2(\xi+\eta)^2\chi_1(\xi, \eta)}{ \xi\eta(9\xi^2 + 14\xi \eta + 9 \eta^2)}(-3\langle\eta\rangle^{s+\frac{1}{2}}\langle\xi+\eta\rangle^{s+\frac{3}{2}}+5\langle \eta\rangle^2 \langle\xi +\eta\rangle^{2s} -2\langle\eta\rangle^{2s+1}\langle\xi+\eta\rangle) , \\
&\mathfrak{c}_2^h = \dfrac{ 3(\xi+\eta)^3\chi_1(\xi, \eta)}{ 
\xi \eta(9\xi^2 + 14\xi \eta + 9 \eta^2)}(-3\langle\eta\rangle^{s+\frac{1}{2}}\langle\xi+\eta\rangle^{s+\frac{3}{2}}+5\langle \eta\rangle^2 \langle\xi +\eta\rangle^{2s} -2\langle\eta\rangle^{2s+1}\langle\xi+\eta\rangle) , \\
&\mathfrak{d}_2^h = -\dfrac{(2\xi^2 +3\xi \eta + 3\eta^2)\chi_1(\xi, \eta)}{\xi \eta(9\xi^2 + 14\xi \eta + 9 \eta^2)}(-3\langle\eta\rangle^{s+\frac{1}{2}}\langle\xi+\eta\rangle^{s+\frac{3}{2}}+5\langle \eta\rangle^2 \langle\xi +\eta\rangle^{2s} -2\langle\eta\rangle^{2s+1}\langle\xi+\eta\rangle). 
\end{aligned}   
\end{equation*}

Similarly, to eliminate non-perturbative terms with para-coefficient $T_{\bar{\W}}$, symbols $\mathfrak{a}_2^a(\xi, \zeta, \eta)$, $\mathfrak{b}_2^a(\xi, \zeta, \eta)$, $\mathfrak{c}_2^a(\xi, \zeta, \eta)$, $\mathfrak{d}_2^a(\xi, \zeta, \eta)$ need to solve the system
\begin{equation*}
\left\{
    \begin{array}{lr}
    \eta^2 \mathfrak{a}_2^a - (\xi+\eta)^2 \mathfrak{b}_2^a - \xi \mathfrak{c}_2^a = 0 &\\
    \xi^2 \mathfrak{b}_2^a + (\xi+\eta) \mathfrak{c}_2^a + \eta^2 \mathfrak{d}_2^a = 0  &\\
    \xi^2 \mathfrak{a}_2^a - \eta \mathfrak{c}_2^a - (\xi+\eta)^2 \mathfrak{d}_2^a = 3\langle \xi+\eta\rangle^{s+\frac{1}{2}}\langle \eta\rangle^{s+\frac{3}{2}}-\langle \xi+\eta\rangle^2 \langle\eta\rangle^{2s} -2\langle \xi+\eta\rangle^{2s+1}\langle\eta\rangle  &\\
     -(\xi+\eta) \mathfrak{a}_2^a +\eta \mathfrak{b}_2^a + \xi \mathfrak{d}_2^a = 0.&  
    \end{array}
\right.
\end{equation*}
The solutions of these symbols are given by 
\begin{equation*}
 \begin{aligned}
&\mathfrak{a}_2^h = -\dfrac{(3\xi^2 + 3\xi\eta + 2 \eta^2)\chi_1(\xi, \eta)}{ \xi\eta(9\xi^2 + 14\xi \eta + 9 \eta^2)}(3\langle \xi+\eta\rangle^{s+\frac{1}{2}}\langle \eta\rangle^{s+\frac{3}{2}}-\langle \xi+\eta\rangle^2 \langle\eta\rangle^{2s} -2\langle \xi+\eta\rangle^{2s+1}\langle\eta\rangle), \\
&\mathfrak{b}_2^h = -\dfrac{2(\xi+\eta)\chi_1(\xi, \eta)}{ \xi(9\xi^2 + 14\xi \eta + 9 \eta^2)}(3\langle \xi+\eta\rangle^{s+\frac{1}{2}}\langle \eta\rangle^{s+\frac{3}{2}}-\langle \xi+\eta\rangle^2 \langle\eta\rangle^{2s} -2\langle \xi+\eta\rangle^{2s+1}\langle\eta\rangle) , \\
&\mathfrak{c}_2^h = \dfrac{ (2\xi^2 +3\xi \eta + 3\eta^2)\chi_1(\xi, \eta)}{ 
\xi(9\xi^2 + 14\xi \eta + 9 \eta^2)}(3\langle \xi+\eta\rangle^{s+\frac{1}{2}}\langle \eta\rangle^{s+\frac{3}{2}}-\langle \xi+\eta\rangle^2 \langle\eta\rangle^{2s} -2\langle \xi+\eta\rangle^{2s+1}\langle\eta\rangle) , \\
&\mathfrak{d}_2^h = -\dfrac{3(\xi+\eta)^2\chi_1(\xi, \eta)}{ \xi\eta(9\xi^2 + 14\xi \eta + 9 \eta^2)}(3\langle \xi+\eta\rangle^{s+\frac{1}{2}}\langle \eta\rangle^{s+\frac{3}{2}}-\langle \xi+\eta\rangle^2 \langle\eta\rangle^{2s} -2\langle \xi+\eta\rangle^{2s+1}\langle\eta\rangle). 
\end{aligned}   
\end{equation*}
Since $|\xi|\ll |\eta|$, and $\langle\eta \rangle \approx \langle\xi+\eta\rangle$, we have the estimates for symbols,
\begin{align*}
    &|-3\langle\eta\rangle^{s+\frac{1}{2}}\langle\xi+\eta\rangle^{s+\frac{3}{2}}+5\langle \eta\rangle^2 \langle\xi +\eta\rangle^{2s} -2\langle\eta\rangle^{2s+1}\langle\xi+\eta\rangle| \lesssim |\xi|\langle \xi+\eta\rangle^{2s+1}, \\
    &|3\langle \xi+\eta\rangle^{s+\frac{1}{2}}\langle \eta\rangle^{s+\frac{3}{2}}-\langle \xi+\eta\rangle^2 \langle\eta\rangle^{2s} -2\langle \xi+\eta\rangle^{2s+1}\langle\eta\rangle| \lesssim |\xi|\langle \xi+\eta\rangle^{2s+1}.
\end{align*}
Using the above inequalities and \eqref{BFHsCmStar}, the energy $\mathcal{E}_{s,2}^{[3]}$ satisfies the estimate
\begin{equation*}
    |\mathcal{E}_{s,2}^{[3]}| \lesssim_\CalAZS \CalAZS\|(\W, R) \|_{\H^s}^2.
\end{equation*}
The extra quartic and higher integral terms are also perturbative.

Finally, we first construct the cubic energy $\mathcal{E}_{s,3}^{[3]}$ as the last part of $\mathcal{E}_s^{[3]}$ such that
\begin{align*}
 \frac{d}{dt}\mathcal{E}_{s,3}^{[3]} &= 2\Re\int \langle D \rangle^{s+\Half}\tilde{\W} \cdot \langle D \rangle^{s+\Half}T_{R}\bar{\tilde{\W}}_\alpha +\langle D \rangle^{s}\tilde{R} \cdot \langle D \rangle^{s}T_{ R}\bar{\tilde{R}}_\alpha \,d\alpha \\
  + 2\Re&\int \langle D \rangle^{s+\Half}\tilde{\W} \cdot \langle D \rangle^{s+\Half}T_{\bar{R}}\bar{\tilde{\W}}_\alpha +\langle D \rangle^{s}\tilde{R} \cdot \langle D \rangle^{s}T_{\tilde{R}}\bar{\tilde{R}}_\alpha \,d\alpha+ \text{perturbative integrals}.
\end{align*}
We consider the cubic energy $\mathcal{E}_{s,3}^{[3]}$ of the following form
\begin{align*}
\mathcal{E}_{s,2}^{[3]} &= \Re\int T_{J^{\frac{1}{2}}(1+\bar{\W})^2}A_3^h(R, \tilde{\W}, \bar{\tilde{R}}) + T_{J^{\frac{1}{2}}(1+\W)^2}B_3^h(R, \tilde{R}, \bar{\tilde{\W}}) \,d\alpha\\
&+ \Re \int T_{(1-\Y)(1+\bar
\W)}C_3^h(\W, \tilde{\W}, \bar{\tilde{\W}}) + T_{J^{\frac{1}{2}}(1+\bar{\W})^2}D_3^h(\W, \tilde{R}, \bar{\tilde{R}}) \,d\alpha\\
&+\Re\int T_{J^{\frac{1}{2}}(1+\bar{\W})^2}A_3^a(\bar{R}, \tilde{\W}, \bar{\tilde{R}}) + T_{J^{\frac{1}{2}}(1+\W)^2}B_3^a(\bar{R}, \tilde{R}, \bar{\tilde{\W}}) \,d\alpha \\
&+\Re \int T_{(1+\W)(1-\bar{\Y})}C_3^a(\bar{\W}, \tilde{\W}, \bar{\tilde{\W}}) + T_{J^{\frac{1}{2}}(1+\W)^2}D_3^a(\bar{\W}, \tilde{R}, \bar{\tilde{R}}) \,d\alpha,
\end{align*}
where $A_3^h,B_3^h,C_3^h, D_3^h$ and $A_3^a,B_3^a,C_3^a, D_3^a$ are paradifferential cubic forms. 
Here, we  write $\mathfrak{a}_3^h(\xi, \eta, \zeta)$ for the symbol of $A_3^h(R, \tilde{\W}, \bar{\tilde{R}})$, and  $\mathfrak{a}_3^a(\xi, \zeta, \eta)$ for the symbol of $A_3^a(\bar{R}, \tilde{\W}, \bar{\tilde{R}})$.
Other symbols are defined in the same way.
Then $\xi + \eta = \zeta$.

When the time derivative falls on the para-coefficients, the quartic and higher-order integrals produced are perturbative.
We compute the time derivative of the energy
\begin{align*}
 &\frac{d}{dt}\mathcal{E}_{s,3}^{[3]} = \Re\int A_3^h(R, \tilde{\W}, i\bar{\tilde{\W}}_{\alpha \alpha}) + B_3^h(R, -i\tilde{\W}_{\alpha\alpha}, \bar{\tilde{\W}})+ C_3^h(-R_\alpha, \tilde{\W}, \bar{\tilde{\W}}) \, d\alpha \\
&+ \Re \int B_3^h(-i \W_{\alpha \alpha}, \tilde{R}, \bar{\tilde{\W}}) + C_3^h(\W, -\tilde{R}_{\alpha}, \bar{\tilde{\W}})+ D_3^h(\W, \tilde{R}, i\bar{\tilde{\W}}_{\alpha\alpha}) \, d\alpha \\
&+ \Re \int A_3^h(-i\W_{\alpha \alpha }, \tilde{\W}, \bar{\tilde{R}}) + C_3^h(\W, \tilde{\W}, -\bar{\tilde{R}}_\alpha)+ D_3^h(\W, -i\tilde{\W}_{\alpha \alpha}, \bar{\tilde{R}}) \, d\alpha \\
&+ \Re \int A_3^h(R, -\tilde{R}_\alpha, \bar{\tilde{R}}) + B_3^h(R, \tilde{R}, -\bar{\tilde{R}}_\alpha)+ D_3^h(-R_\alpha, \tilde{R}, \bar{\tilde{R}}) \, d\alpha\\
&+\Re\int A_3^a(\bar{R}, \tilde{\W}, i\bar{\tilde{\W}}_{\alpha \alpha}) + B_3^a(\bar{R}, -i\tilde{\W}_{\alpha\alpha}, \bar{\tilde{\W}})+ C_3^a(-\bar{R}_\alpha, \tilde{\W}, \bar{\tilde{\W}}) \, d\alpha \\
&+ \Re \int B_3^a(i \bar{\W}_{\alpha \alpha}, \tilde{R}, \bar{\tilde{\W}}) + C_3^a(\bar{\W}, -\tilde{R}_{\alpha}, \bar{\tilde{\W}})+ D_3^a(\bar{\W}, \tilde{R}, i\bar{\tilde{\W}}_{\alpha\alpha}) \, d\alpha \\
&+ \Re \int A_3^a(i\bar{\W}_{\alpha \alpha }, \tilde{\W}, \bar{\tilde{R}}) + C_3^a(\bar{\W}, \tilde{\W}, -\bar{\tilde{R}}_\alpha)+ D_3^a(\bar{\W}, -i\tilde{\W}_{\alpha \alpha}, \bar{\tilde{R}}) \, d\alpha \\
&+ \Re \int A_3^a(\bar{R}, -\tilde{R}_\alpha, \bar{\tilde{R}}) + B_3^a(\bar{R}, \tilde{R}, -\bar{\tilde{R}}_\alpha)+ D_3^a(-\bar{R}_\alpha, \tilde{R}, \bar{\tilde{R}}) \, d\alpha  + \text{quartic and higher integrals}.
\end{align*}

To eliminate non-perturbative integral terms, symbols $\mathfrak{a}_3^h(\xi, \eta, \zeta), \mathfrak{b}_3^h(\xi, \eta, \zeta), \mathfrak{c}_3^h(\xi, \eta, \zeta), \mathfrak{d}_3^h(\xi, \eta, \zeta)$ need to solve the system 
\begin{equation*}
\left\{
    \begin{array}{lr}
    (\xi+\eta)^2 \mathfrak{a}_3^h - \eta^2 \mathfrak{b}_3^h + \xi \mathfrak{c}_3^h = 2 (\xi+ \eta)\langle\eta \rangle^{2s+1} \chi_1(\xi, \eta) &\\
    \xi^2 \mathfrak{b}_3^h - \eta \mathfrak{c}_3^h - (\xi+\eta)^2 \mathfrak{d}_3^h = 0  &\\
    \xi^2 \mathfrak{a}_3^h + (\xi+\eta) \mathfrak{c}_3^h + \eta^2 \mathfrak{d}_3^h = 0  &\\
     -\eta \mathfrak{a}_3^h +(\xi +\eta) \mathfrak{b}_3^h - \xi \mathfrak{d}_3^h = -2 (\xi+ \eta)\langle\eta \rangle^{2s} \chi_1(\xi, \eta).&  
    \end{array}
\right.
\end{equation*}
At low frequency $|\eta| \approx 1$, the symbols are all bounded and the estimates are trivial, it only suffices to work on the high frequency regime $|\eta|\gg 1$.
Since $\tilde{\W}$ and $\tilde{R}$ are holomorphic functions, at high frequencies $\eta \approx -\langle \eta \rangle$, so that at high frequencies $|\eta|\gg 1$,
\begin{equation*}
 \begin{aligned}
&\mathfrak{a}_3^h \approx -\dfrac{2(\xi+\eta)(5\xi^2 + 9\xi \eta + 6\eta^2)}{2 \xi(9\xi^2 + 14\xi \eta + 9 \eta^2)}\langle\eta \rangle^{2s}\chi_1(\xi, \eta), \\
&\mathfrak{b}_3^h \approx -\dfrac{2(\xi + \eta)(3\xi^3+ 11\xi^2\eta +12 \xi\eta^2 +6\eta^3)}{\xi\eta(9\xi^2 + 14\xi \eta + 9 \eta^2)}\langle\eta \rangle^{2s} \chi_1(\xi, \eta), \\
&\mathfrak{c}_3^h \approx \dfrac{ 2(\xi+\eta)(5\xi^2 + 7\xi \eta +4\eta^2)}{ 
9\xi^2 + 14\xi \eta + 9 \eta^2}\langle\eta \rangle^{2s} \chi_1(\xi, \eta), \\
&\mathfrak{d}_3^h \approx -\dfrac{2(\xi + \eta)(3\xi^2 +5\xi \eta + 4\eta^2)}{\eta(9\xi^2 + 14\xi \eta + 9 \eta^2)}\langle\eta \rangle^{2s} \chi_1(\xi, \eta). 
\end{aligned}   
\end{equation*}

Similarly, symbols $\mathfrak{a}_3^a(\xi, \zeta, \eta),\mathfrak{b}_3^a(\xi, \zeta, \eta),\mathfrak{c}_3^a(\xi, \zeta, \eta),\mathfrak{d}_3^a(\xi, \zeta, \eta)$
need to solve the system 
\begin{equation*}
\left\{
    \begin{array}{lr}
    \eta^2 \mathfrak{a}_3^a - (\xi+\eta)^2 \mathfrak{b}_3^a - \xi \mathfrak{c}_3^a = 2  \eta\langle\eta \rangle^{2s+1} \chi_1(\xi, \eta) &\\
    \xi^2 \mathfrak{b}_3^a + (\xi+\eta) \mathfrak{c}_3^a + \eta^2 \mathfrak{d}_3^a = 0  &\\
    \xi^2 \mathfrak{a}_3^a - \eta \mathfrak{c}_3^a - (\xi+\eta)^2 \mathfrak{d}_3^a = 0  &\\
     -(\xi+\eta) \mathfrak{a}_3^a +\eta \mathfrak{b}_3^a + \xi \mathfrak{d}_3^a = -2 \eta\langle\eta \rangle^{2s} \chi_1(\xi, \eta).&  
    \end{array}
\right.
\end{equation*}
Again it suffices to consider the high frequency regime where $\eta \approx -\langle \eta \rangle$, 
\begin{equation*}
 \begin{aligned}
&\mathfrak{a}_3^a \approx \dfrac{2(3\xi^3 + 11\xi^2\eta + 12 \xi \eta^2 + 6\eta^3)}{ \xi(9\xi^2 + 14\xi \eta + 9 \eta^2)}\langle\eta \rangle^{2s}\chi_1(\xi, \eta), \\
&\mathfrak{b}_3^a \approx \dfrac{2\eta(5\xi^2+ 9\xi\eta +6\eta^2)}{2 \xi(9\xi^2 + 14\xi \eta + 9 \eta^2)}\langle\eta \rangle^{2s} \chi_1(\xi, \eta), \\
&\mathfrak{c}_3^a \approx -\dfrac{ 2\eta(5\xi^2 + 7\xi \eta +4\eta^2)}{ 
9\xi^2 + 14\xi \eta + 9 \eta^2}\langle\eta \rangle^{2s} \chi_1(\xi, \eta), \\
&\mathfrak{d}_3^a \approx -\dfrac{2(3\xi^2 +5\xi \eta + 4\eta^2)}{(9\xi^2 + 14\xi \eta + 9 \eta^2)}\langle\eta \rangle^{2s} \chi_1(\xi, \eta). 
\end{aligned}   
\end{equation*}
One part of the energy $\mathcal{E}_{s,3}^{[3]}$
\begin{align*}
\Re &\int T_{(1-\Y)(1+\bar
\W)}C_3^h(\W, \tilde{\W}, \bar{\tilde{\W}}) + T_{J^{\frac{1}{2}}(1+\bar{\W})^2}D_3^h(\W, \tilde{R}, \bar{\tilde{R}}) \\
&+ T_{(1-\bar{\Y})(1+
\W)}C_3^a(\bar{\W}, \tilde{\W}, \bar{\tilde{\W}}) + T_{J^{\frac{1}{2}}(1+\W)^2}D_3^a(\bar{\W}, \tilde{R}, \bar{\tilde{R}})\,d\alpha
\end{align*}
satisfies the energy bound
\begin{equation} \label{EnergyEqual}
    |\mathcal{E}|\lesssim_\CalAZS \CalAZS\|(\W, R) \|_{\H^s}^2
\end{equation}
according to \eqref{BFHsCmStar}.
The other four parts of the energy $\mathcal{E}_{s,3}^{[3]}$
\begin{align*}
 &\Re \int T_{J^{\frac{1}{2}}(1+\bar{\W})^2}A_3^h(R, \tilde{\W}, \bar{\tilde{R}})   \,d\alpha, \quad \Re \int T_{J^{\frac{1}{2}}(1+\W)^2}B_3^h(R, \tilde{R}, \bar{\tilde{\W}})\,d\alpha, \\
 &\Re \int T_{J^{\frac{1}{2}}(1+\bar{\W})^2}A_3^a(\bar{R}, \tilde{\W}, \bar{\tilde{R}})\,d\alpha, \quad\Re \int   T_{J^{\frac{1}{2}}(1+\W)^2}B_3^a(\bar{R}, \tilde{R}, \bar{\tilde{\W}}) \,d\alpha
\end{align*}
do not satisfy this energy bound \eqref{EnergyEqual}.
However, since $|\xi|\ll |\eta|$, 
\begin{align*}
 &\mathfrak{a}_3^h(\xi, \eta, \zeta) \approx -\frac{2\eta}{3\xi}\langle \eta \rangle^{2s}\chi_1(\xi,\eta)+ \text{lower order terms in }\eta, \\
 &\mathfrak{b}_3^a(\xi, \eta, \zeta) \approx \frac{2\eta}{3\xi}\langle \eta \rangle^{2s}\chi_1(\xi,\eta)+ \text{lower order terms in }\eta,  \\
 &\mathfrak{b}_3^h(\xi, \eta, \zeta) \approx -\frac{4\eta}{3\xi}\langle \eta \rangle^{2s}\chi_1(\xi,\eta)+ \text{lower order terms in }\eta, \\
 &\mathfrak{a}_3^a(\xi, \eta, \zeta) \approx \frac{4\eta}{3\xi}\langle \eta \rangle^{2s}\chi_1(\xi,\eta)+ \text{lower order terms in }\eta.
\end{align*}
The leading term of the symbol $\mathfrak{a}_3^h(\xi, \eta, \zeta)$ cancels the leading term of the symbol $\mathfrak{b}_3^a(\xi, \zeta, \eta)$, and  the leading term of the symbol $\mathfrak{b}_3^h(\xi, \eta, \zeta)$ cancels the leading term of the symbol $\mathfrak{a}_3^a(\xi, \zeta, \eta)$.
The remaining lower order terms of symbols satisfy the need for the energy bound \eqref{EnergyEqual}.
Hence, taking the real part of the energy yields that two energies
\begin{align*}
 &\Re \int T_{J^{\frac{1}{2}}(1+\bar{\W})^2}A_3^h(R, \tilde{\W}, \bar{\tilde{R}})   + T_{J^{\frac{1}{2}}(1+\W)^2}B_3^a(\bar{R}, \tilde{R}, \bar{\tilde{\W}})\,d\alpha, \\
 &\Re \int T_{J^{\frac{1}{2}}(1+\bar{\W})^2}A_3^a(\bar{R}, \tilde{\W}, \bar{\tilde{R}})\,d\alpha + T_{J^{\frac{1}{2}}(1+\W)^2}B_3^h(R, \tilde{R}, \bar{\tilde{\W}})   \,d\alpha
\end{align*}
satisfy the bound \eqref{EnergyEqual}.
Therefore, the energy $\mathcal{E}_{s,3}^{[3]}$ satisfies the estimate
\begin{equation*}
    |\mathcal{E}_{s,3}^{[3]}| \lesssim_\CalAZS \CalAZS\|(\W, R) \|_{\H^s}^2,
\end{equation*}
and quartic and higher integral terms are  perturbative.

Collecting all parts of the modified energy
\begin{equation*}
    \mathcal{E}_s = \mathcal{E}_s^{[2]} + \mathcal{E}_{s,1}^{[3]} + \mathcal{E}_{s,2}^{[3]}  +\mathcal{E}_{s,3}^{[3]}, 
\end{equation*}
it satisfies the norm equivalence \eqref{normEqnTwo} and the time derivative of the energy is bounded by the right hand side of \eqref{IboundDSt}.
Using the estimate for estimate for $d_S(t)$ \eqref{DistanceSurface}, for either $\lambda>0$ or $\lambda<0$,
\begin{equation*}
    d_S(t)^{-\frac{3}{2}} \lesssim (1+ t)^{-\frac{3}{2}}, \quad t\in[0,T].
\end{equation*}
We get the modified energy estimate \eqref{EnergyEst}, which finishes the proof of Theorem \ref{t:CubicEnergy}.

\subsection{The proof of cubic lifespan} \label{s:prooflifespan}
This section is devoted to the proof of Theorem \ref{t:lifespan}.

We need to close the bootstrap assumption \eqref{bootstrap}.
Applying the Gronwall's inequality for the energy estimate \eqref{EnergyEst},
\begin{align*}
    \mathcal{E}_s(t) &\leq \exp \left\{C_1\int_0^t \CalAOS(s)\CalATS(s) +  (1+ s)^{-\frac{3}{2}} ds \right\}\\
    \cdot &\left(\mathcal{E}_s(0) + C_2\int_0^t \epsilon^2 (1+ s)^{-\frac{3}{2}}  \exp\left\{ -C_3\int_0^s \CalAOS(\tau)\CalATS(\tau) +  (1+ \tau)^{-\frac{3}{2}} d\tau\right\} ds\right),
\end{align*}
for positive constants $C_1, C_2, C_3$ that depend on $\sup_{s\in [0,t]}\CalAZS(s)$.
Note that 
\begin{equation*}
    \int_0^t  (1+ s)^{-\frac{3}{2}} ds \leq \int_0^{+\infty} (1+s)^{-\frac{3}{2}} ds = 2.
\end{equation*}
Using the bootstrap assumption \eqref{bootstrap}, and the Sobolev embedding, the constant $C$ is bounded, and
\begin{align*}
    \CalAOS(t) \leq \CalATS(t) \lesssim \epsilon, \quad \forall t\in [0,T],
\end{align*}
so that we have
\begin{equation*}
    \int_0^t \CalAOS(s)\CalATS(s) \,ds \leq \int_0^T \CalAOS(s)\CalATS(s) \,ds \lesssim \epsilon^2\cdot T \lesssim \delta.
\end{equation*}
Therefore, we get the bound for energy 
\begin{equation*}
\|(\W, R)(t) \|_{\H^s}^2 \leq e^{(\delta+2)C_1} (1+2C_2)\epsilon^2 < M^2 \epsilon^2, \quad t\in [0,T].
\end{equation*}
This closes the first part of the bootstrap assumption \eqref{bootstrap} and also proves the first part of energy bound \eqref{WRZNormBound}.

For the bound of $\|Z-\alpha \|_{L^2}$ in \eqref{bootstrap}, we compute using the equation for $Z$ \eqref{ZEqnTwo}
\begin{align*}
    \frac{d}{dt}\|Z-\alpha\|^2_{L^2} \leq \|Z_t\|_{L^2}\|Z-\alpha \|_{L^2} \leq \|Z-\alpha \|_{L^2} (\|b(1+\W) \|_{L^2} + \|R\|_{L^2} + \|v^{rot}\circ Z \|_{L^2}) \lesssim \epsilon^2.
\end{align*}
Hence, we have the energy bound
\begin{equation*}
 \|Z(t)-\alpha \|^2_{L^2} \leq \exp \left\{C\int_0^t \epsilon^2 \,ds\right\}\| Z(0)-\alpha\|_{L^2}^2 \leq e^{C\delta} \epsilon^2 < M^2\epsilon^2.
\end{equation*}
This closes the second part of the bootstrap assumption \eqref{bootstrap} and also proves the second part of the energy bound \eqref{WRZNormBound}.

Given initial datum $(\W_0, R_0, z_1(0), z_2(0))$ that satisfy the condition in Theorem \ref{t:lifespan}, the solution $(\W(t), R(t), z_1(t), z_2(t))$ exist locally in time according to Theorem \ref{t:Wellposed}.
We also have the a priori energy bound \eqref{WRZNormBound}. 
Two vortices are away from the surface, and the distance of two point vortices is bounded from below. 
Using the Sobolev embedding, this shows that for control norms,
\begin{equation*}
    \| \CalAO(t)\|_{L^\infty_t[0, T]} + \| \CalAT(t)\|_{L^1_t[0,T]} < +\infty.
\end{equation*}
Hence, the solution can be extended to at least cubic lifespan $T = \delta\epsilon^{-2}$.

\bibliography{ww}

\begin{thebibliography}{10}

\bibitem{MR4161284}
Albert Ai.
\newblock Low regularity solutions for gravity water waves.
\newblock {\em Water Waves}, 1(1):145--215, 2019.

\bibitem{MR4098033}
Albert Ai.
\newblock Low regularity solutions for gravity water waves {II}: {T}he 2{D} case.
\newblock {\em Ann. PDE}, 6(1):Paper No. 4, 117, 2020.

\bibitem{ai2023improved}
Albert Ai.
\newblock Improved low regularity theory for gravity-capillary waves, 2023.

\bibitem{MR4483135}
Albert Ai, Mihaela Ifrim, and Daniel Tataru.
\newblock Two-dimensional gravity waves at low regularity {II}: {G}lobal solutions.
\newblock {\em Ann. Inst. H. Poincar\'e{} C Anal. Non Lin\'eaire}, 39(4):819--884, 2022.

\bibitem{ai2023dimensional}
Albert Ai, Mihaela Ifrim, and Daniel Tataru.
\newblock Two dimensional gravity waves at low regularity {I}: Energy estimates, 2023.

\bibitem{MR4035330}
Albert~Lee Ai.
\newblock {\em Low {R}egularity {S}olutions for {G}ravity {W}ater {W}aves}.
\newblock ProQuest LLC, Ann Arbor, MI, 2019.
\newblock Thesis (Ph.D.)--University of California, Berkeley.

\bibitem{MR2805065}
T.~Alazard, N.~Burq, and C.~Zuily.
\newblock On the water-wave equations with surface tension.
\newblock {\em Duke Math. J.}, 158(3):413--499, 2011.

\bibitem{MR3260858}
T.~Alazard, N.~Burq, and C.~Zuily.
\newblock On the {C}auchy problem for gravity water waves.
\newblock {\em Invent. Math.}, 198(1):71--163, 2014.

\bibitem{MR3465379}
T.~Alazard, N.~Burq, and C.~Zuily.
\newblock Cauchy theory for the gravity water waves system with non-localized initial data.
\newblock {\em Ann. Inst. H. Poincar\'e{} C Anal. Non Lin\'eaire}, 33(2):337--395, 2016.

\bibitem{MR2931520}
Thomas Alazard, Nicolas Burq, and Claude Zuily.
\newblock Strichartz estimates for water waves.
\newblock {\em Ann. Sci. \'Ec. Norm. Sup\'er. (4)}, 44(5):855--903, 2011.

\bibitem{MR3852259}
Thomas Alazard, Nicolas Burq, and Claude Zuily.
\newblock Strichartz estimates and the {C}auchy problem for the gravity water waves equations.
\newblock {\em Mem. Amer. Math. Soc.}, 256(1229):v+108, 2018.

\bibitem{MR3429478}
Thomas Alazard and Jean-Marc Delort.
\newblock Global solutions and asymptotic behavior for two dimensional gravity water waves.
\newblock {\em Ann. Sci. \'Ec. Norm. Sup\'er. (4)}, 48(5):1149--1238, 2015.

\bibitem{MR2768550}
Hajer Bahouri, Jean-Yves Chemin, and Rapha\"{e}l Danchin.
\newblock {\em Fourier analysis and nonlinear partial differential equations}, volume 343 of {\em Grundlehren der mathematischen Wissenschaften [Fundamental Principles of Mathematical Sciences]}.
\newblock Springer, Heidelberg, 2011.

\bibitem{MR1637554}
Klaus Beyer and Matthias G\"unther.
\newblock On the {C}auchy problem for a capillary drop. {I}. {I}rrotational motion.
\newblock {\em Math. Methods Appl. Sci.}, 21(12):1149--1183, 1998.

\bibitem{bieri2015motion}
Lydia Bieri, Shuang Miao, Sohrab Shahshahani, and Sijue Wu.
\newblock On the motion of a self-gravitating incompressible fluid with free boundary and constant vorticity: An appendix, 2015.

\bibitem{chen2024}
Robin~Ming Chen, Kristoffer Varholm, Samuel Walsh, and Miles~H. Wheeler.
\newblock Vortex-carrying solitary gravity waves of large amplitude, 2024.

\bibitem{MR1780703}
Demetrios Christodoulou and Hans Lindblad.
\newblock On the motion of the free surface of a liquid.
\newblock {\em Comm. Pure Appl. Math.}, 53(12):1536--1602, 2000.

\bibitem{MR2291920}
Daniel Coutand and Steve Shkoller.
\newblock Well-posedness of the free-surface incompressible {E}uler equations with or without surface tension.
\newblock {\em J. Amer. Math. Soc.}, 20(3):829--930, 2007.

\bibitem{MR3487264}
Thibault de~Poyferr\'{e} and Quang-Huy Nguyen.
\newblock Strichartz estimates and local existence for the gravity-capillary waves with non-{L}ipschitz initial velocity.
\newblock {\em J. Differential Equations}, 261(1):396--438, 2016.

\bibitem{MR128208}
I.~G. Filippov.
\newblock Solution of the problem of the motion of a vortex under the surface of a fluid, for {F}roude numbers near unity.
\newblock {\em J. Appl. Math. Mech.}, 24:698--716, 1960.

\bibitem{MR151069}
I.~G. Filippov.
\newblock On the motion of a vortex below the surface of a liquid.
\newblock {\em J. Appl. Math. Mech.}, pages 357--365, 1961.

\bibitem{MR2787587}
Thierry Gallay.
\newblock Interaction of vortices in weakly viscous planar flows.
\newblock {\em Arch. Ration. Mech. Anal.}, 200(2):445--490, 2011.

\bibitem{MR3858400}
Olivier Glass, Alexandre Munnier, and Franck Sueur.
\newblock Point vortex dynamics as zero-radius limit of the motion of a rigid body in an irrotational fluid.
\newblock {\em Invent. Math.}, 214(1):171--287, 2018.

\bibitem{MR3625189}
Benjamin Harrop-Griffiths, Mihaela Ifrim, and Daniel Tataru.
\newblock Finite depth gravity water waves in holomorphic coordinates.
\newblock {\em Ann. PDE}, 3(1):Paper No. 4, 102, 2017.

\bibitem{MR3535894}
John~K. Hunter, Mihaela Ifrim, and Daniel Tataru.
\newblock Two dimensional water waves in holomorphic coordinates.
\newblock {\em Comm. Math. Phys.}, 346(2):483--552, 2016.

\bibitem{MR3348783}
John~K. Hunter, Mihaela Ifrim, Daniel Tataru, and Tak~Kwong Wong.
\newblock Long time solutions for a {B}urgers-{H}ilbert equation via a modified energy method.
\newblock {\em Proc. Amer. Math. Soc.}, 143(8):3407--3412, 2015.

\bibitem{ifrim2023}
Mihaela Ifrim, Ben Pineau, Daniel Tataru, and Mitchell~A. Taylor.
\newblock Sharp hadamard local well-posedness, enhanced uniqueness and pointwise continuation criterion for the incompressible free boundary euler equations, 2023.

\bibitem{MR4462478}
Mihaela Ifrim, James Rowan, Daniel Tataru, and Lizhe Wan.
\newblock The {B}enjamin-{O}no approximation for 2{D} gravity water waves with constant vorticity.
\newblock {\em Ars Inven. Anal.}, pages Paper No. 3, 33, 2022.

\bibitem{MR3499085}
Mihaela Ifrim and Daniel Tataru.
\newblock Two dimensional water waves in holomorphic coordinates {II}: {G}lobal solutions.
\newblock {\em Bull. Soc. Math. France}, 144(2):369--394, 2016.

\bibitem{MR3667289}
Mihaela Ifrim and Daniel Tataru.
\newblock The lifespan of small data solutions in two dimensional capillary water waves.
\newblock {\em Arch. Ration. Mech. Anal.}, 225(3):1279--1346, 2017.

\bibitem{MR3869381}
Mihaela Ifrim and Daniel Tataru.
\newblock Two-dimensional gravity water waves with constant vorticity {I}: {C}ubic lifespan.
\newblock {\em Anal. PDE}, 12(4):903--967, 2019.

\bibitem{MR3862598}
Alexandru~D. Ionescu and Fabio Pusateri.
\newblock Global regularity for 2{D} water waves with surface tension.
\newblock {\em Mem. Amer. Math. Soc.}, 256(1227):v+124, 2018.

\bibitem{MR2138139}
David Lannes.
\newblock Well-posedness of the water-waves equations.
\newblock {\em J. Amer. Math. Soc.}, 18(3):605--654, 2005.

\bibitem{MR2178961}
Hans Lindblad.
\newblock Well-posedness for the motion of an incompressible liquid with free surface boundary.
\newblock {\em Ann. of Math. (2)}, 162(1):109--194, 2005.

\bibitem{MR1220946}
Carlo Marchioro and Mario Pulvirenti.
\newblock Vortices and localization in {E}uler flows.
\newblock {\em Comm. Math. Phys.}, 154(1):49--61, 1993.

\bibitem{MR1245492}
Carlo Marchioro and Mario Pulvirenti.
\newblock {\em Mathematical theory of incompressible nonviscous fluids}, volume~96 of {\em Applied Mathematical Sciences}.
\newblock Springer-Verlag, New York, 1994.

\bibitem{MR3052498}
Camil Muscalu and Wilhelm Schlag.
\newblock {\em Classical and multilinear harmonic analysis. {V}ol. {I}}, volume 137 of {\em Cambridge Studies in Advanced Mathematics}.
\newblock Cambridge University Press, Cambridge, 2013.

\bibitem{MR3724757}
Huy~Quang Nguyen.
\newblock A sharp {C}auchy theory for the 2{D} gravity-capillary waves.
\newblock {\em Ann. Inst. H. Poincar\'{e} C Anal. Non Lin\'{e}aire}, 34(7):1793--1836, 2017.

\bibitem{MR1946720}
Masao Ogawa and Atusi Tani.
\newblock Free boundary problem for an incompressible ideal fluid with surface tension.
\newblock {\em Math. Models Methods Appl. Sci.}, 12(12):1725--1740, 2002.

\bibitem{MR3053431}
Jalal Shatah, Samuel Walsh, and Chongchun Zeng.
\newblock Travelling water waves with compactly supported vorticity.
\newblock {\em Nonlinearity}, 26(6):1529--1564, 2013.

\bibitem{MR2388661}
Jalal Shatah and Chongchun Zeng.
\newblock Geometry and a priori estimates for free boundary problems of the {E}uler equation.
\newblock {\em Comm. Pure Appl. Math.}, 61(5):698--744, 2008.

\bibitem{MR2763036}
Jalal Shatah and Chongchun Zeng.
\newblock Local well-posedness for fluid interface problems.
\newblock {\em Arch. Ration. Mech. Anal.}, 199(2):653--705, 2011.

\bibitem{MR4179726}
Qingtang Su.
\newblock Long time behavior of 2{D} water waves with point vortices.
\newblock {\em Comm. Math. Phys.}, 380(3):1173--1266, 2020.

\bibitem{MR4656809}
Qingtang Su.
\newblock On the transition of the {R}ayleigh-{T}aylor instability in 2d water waves with point vortices.
\newblock {\em Ann. PDE}, 9(2):Paper No. 19, 81, 2023.

\bibitem{MR108145}
A.~M. Ter-Krikorov.
\newblock Exact solution of the problem of the motion of a vortex under the surface of a liquid.
\newblock {\em Izv. Akad. Nauk SSSR Ser. Mat.}, 22:177--200, 1958.

\bibitem{MR3485858}
Kristoffer Varholm.
\newblock Solitary gravity-capillary water waves with point vortices.
\newblock {\em Discrete Contin. Dyn. Syst.}, 36(7):3927--3959, 2016.

\bibitem{MR4164269}
Kristoffer Varholm, Erik Wahl\'en, and Samuel Walsh.
\newblock On the stability of solitary water waves with a point vortex.
\newblock {\em Comm. Pure Appl. Math.}, 73(12):2634--2684, 2020.

\bibitem{wan2024}
Lizhe Wan.
\newblock Low regularity well-posedness for two-dimensional deep water waves, 2024.

\bibitem{MR4891579}
Lizhe Wan.
\newblock On the capillary water waves with constant vorticity.
\newblock {\em J. Differential Equations}, 435:113308, 2025.

\bibitem{MR3585049}
Chao Wang and ZhiFei Zhang.
\newblock Break-down criterion for the water-wave equation.
\newblock {\em Sci. China Math.}, 60(1):21--58, 2017.

\bibitem{MR4263411}
Chao Wang, Zhifei Zhang, Weiren Zhao, and Yunrui Zheng.
\newblock Local well-posedness and break-down criterion of the incompressible {E}uler equations with free boundary.
\newblock {\em Mem. Amer. Math. Soc.}, 270(1318):v + 119, 2021.

\bibitem{MR1471885}
Sijue Wu.
\newblock Well-posedness in {S}obolev spaces of the full water wave problem in {$2$}-{D}.
\newblock {\em Invent. Math.}, 130(1):39--72, 1997.

\bibitem{MR2507638}
Sijue Wu.
\newblock Almost global wellposedness of the 2-{D} full water wave problem.
\newblock {\em Invent. Math.}, 177(1):45--135, 2009.

\bibitem{MR2410409}
Ping Zhang and Zhifei Zhang.
\newblock On the free boundary problem of three-dimensional incompressible {E}uler equations.
\newblock {\em Comm. Pure Appl. Math.}, 61(7):877--940, 2008.

\end{thebibliography}
\bibliographystyle{plain}
\end{document}